\documentclass[a4paper,11pt]{article}  

\usepackage{authblk}
\usepackage{appendix}
\usepackage{amsmath}
\usepackage{amssymb}
\usepackage{amsthm}
\usepackage[mathscr]{eucal}
\usepackage{bm}
\usepackage{bbm}
\usepackage{enumerate}
\usepackage{hyperref}
\usepackage{mathtools}
\usepackage[nocompress]{cite}

\usepackage{tikz}
\usepackage{graphicx}
\usepackage{subcaption}

\usepackage{todonotes}

\newcommand{\pr}{\mathbb{P}}
\newcommand{\Prob}[1]{\pr\left(#1\right)}
\newcommand{\Probn}[1]{\pr_n\left(#1\right)}

\newcommand{\e}{\mathbb{E}}
\newcommand{\expec}{\mathbb{E}}
\newcommand{\Exp}[1]{\e\left[#1\right]}
\newcommand{\Expn}[1]{\e_n\left[#1\right]}

\newcommand{\Varn}[1]{\textup{Var}_n\left(#1\right)}

\newcommand{\plim}{\ensuremath{\stackrel{\pr}{\longrightarrow}}}
\newcommand{\dlim}{\ensuremath{\stackrel{d}{\longrightarrow}}}

\newcommand{\ind}[1]{\mathbbm{1}_{\left\{#1\right\}}}
\newcommand{\indE}[1]{\mathbbm{1}_{#1}}

\newcommand{\bigO}[1]{O\left(#1\right)}
\newcommand{\bigOp}[1]{O_{\scriptscriptstyle\pr}\left(#1\right)}

\newcommand\mumberthis{\addtocounter{equation}{1}\tag{\theequation}}

\newcommand{\sss}[1]{\scriptscriptstyle{#1}}

\newcommand\abs[1]{\left|#1\right|}
\newcommand{\me}{\textup{e}}
\newcommand{\dd}{\,{\rm d}}
\newcommand{\op}{o_{\sss\pr}}

\newcommand{\eqn}[1]{\begin{equation}#1\end{equation}}
\newcommand{\eqan}[1]{\begin{align}#1\end{align}}

\newtheorem{theorem}{Theorem}[section]
\newtheorem{definition}{Definition}[section]
\newtheorem{lemma}[theorem]{Lemma}
\newtheorem{proposition}[theorem]{Proposition}
\newtheorem{corollary}[theorem]{Corollary}
\newtheorem{assumption}[definition]{Assumption}
\newtheorem{conjecture}[theorem]{Conjecture}
\newtheorem{condition}{Condition}[section]

\title{Limit theorems for assortativity and clustering in null models for scale-free networks}
\author[1]{Remco van der Hofstad}
\author[1,2]{Pim van der Hoorn}
\author[1,3]{Nelly Litvak}
\author[1,3]{Clara Stegehuis}
\affil[1]{Department of Mathematics and Computer Science, Eindhoven University of Technology}
\affil[2]{Department of Physics, Northeastern University, Boston}
\affil[3]{Department of Electrical Engineering, Mathematics and Computer Science, University of Twente}

\begin{document}

\maketitle

\begin{abstract}
An important problem in modeling networks is how to generate a randomly sampled graph with given degrees. A popular model is the configuration model, a network with assigned degrees and random connections. The erased configuration model is obtained when self-loops and multiple edges in the configuration model are removed. We prove an upper bound for the number of such erased edges for regularly-varying degree distributions with infinite variance, and use this result to prove central limit theorems for Pearson's correlation coefficient and the clustering coefficient in the erased configuration model. Our results explain the structural correlations in the erased configuration model and show that removing edges leads to different scaling of the clustering coefficient. We then prove that  for the rank-1 inhomogeneous random graph, another null model that creates scale-free simple networks, the results for Pearson's correlation coefficient as well as for the clustering coefficient are similar to the results for the erased configuration model.
\end{abstract}

%

\section{Introduction and results}

\subsection{Motivation}
The configuration model \cite{Bollobas1980,Wormald1980} is an important null model to generate graphs with a given degree sequence, by assigning each node a number of half-edges equal to its degree and connecting stubs at random to form edges. Conditioned on the resulting graph being simple, its distribution is uniform over all graphs with the same degree sequence \cite{VanDerHofstad2016a}. Due to this feature the configuration model is widely used to analyze the influence of degrees on other properties or processes on networks~\cite{janson2009b,hofstad2005,hofstad2007,newman2001,riordan2012,federico2017}. 

An important property that many networks share is that their degree distributions are regularly varying, with the exponent $\gamma$ of the degree distribution satisfying $\gamma\in(1,2)$, so that the degrees have infinite variance. In this regime of degrees, the configuration model results in a simple graph with vanishing probability. To still be able to generate simple graphs with approximately the desired degree distribution, the erased configuration model (ECM) removes self-loops and multiple edges of the configuration model~\cite{britton2006}, while the empirical degree distribution still converges to the original one \cite{VanDerHofstad2016a}. 

The degree distribution is a first order characteristic of the network structure, since it is independent of the way nodes are connected. An important second order network characteristic is the correlation between degrees of connected nodes, called degree-degree correlations or network assortativity. A classical measure for these correlations computes Pearson's correlation coefficient on the vector of joint degrees of connected nodes \cite{Newman2002,Newman2003a}. In the configuration model, Pearson's correlation coefficient tends to zero in the large graph limit \cite{Hofstad2014}, so that the configuration model is only able to generate networks with neutral degree correlations.

The configuration model creates networks with asymptotic neutral degree correlations~\cite{Hofstad2014}. By this we mean that, as the size of the network tends to infinity, the joint distribution of degrees on both sides of a randomly sampled edge factorizes as the product of the sized-biased distributions. As a result, the outcome of any degree-degree correlation measure converges to zero. Although one would expect fluctuations of such measures to be symmetric around zero, it has frequently been observed that constraining a network to be simple results in so-called structural negative correlations~\cite{Hoorn2015PhysRev,Catanzaro2005,yao2017ANND,stegehuis2017b}, where the majority of measured degree-degree correlations are negative, while still converging to zero in the infinite graph limit. This is most prominent in the case where the variance of the degree distribution is infinite.
To investigate the extent to which the edge removal procedure of the erased configuration model results in structural negative correlations, we first characterize the scaling of the number of edges that have been removed. Such results are known when the degree distribution has finite variance \cite{Angel2016,Janson2009,Janson2014}. However, for scale-free distributions with infinite variance only some preliminary upper bounds have been proven \cite{Hoorn2015}. Here we prove a new upper bound and obtain several useful corollaries. Our result improves the one in \cite{Hoorn2015} while strengthening \cite[Theorem 8.13]{Hoorn2016}. We then use this bound on the number of removed edges to investigate the consequences of the edge removal procedure on Pearson's correlation coefficient in the erased configuration model. We prove a central limit theorem, which shows that the correlation coefficient in the erased configuration model converges to a random variable with negative support when properly rescaled. Thus, our result confirms the existence of structural correlations in simple networks theoretically. 

We then investigate a `global' clustering coefficient, which is the number of triangles divided by the number of triplets connected by two edges, eventually including multiple edges, see~\eqref{eq:clustdef} and~\eqref{eq:clust} for the precise definition. Thus, the clustering coefficient measures the tendency of sets of three vertices to form a triangle. In the configuration model, the clustering coefficient tends to zero whenever the exponent of the degree distribution satisfies $\gamma>4/3$, whereas it tends to infinity for $\gamma<4/3$~\cite{newman2003book} in the infinite graph limit. In this paper, we obtain more detailed results on the behavior of the clustering coefficient in the configuration model in the form of a central limit theorem. We then investigate how the edge removal procedure of the erased configuration model affects the clustering coefficient and obtain a central limit theorem for the clustering coefficient in the erased configuration model. 

Interestingly, it was shown in~\cite{ostroumova2016,prokhorenkova2014} that in simple graphs with $\gamma\in (1,2)$ the clustering coefficient converges to zero. This again shows that constraining a graph to be simple may significantly impact network statistics. We obtain a precise scaling for the clustering coefficient in ECM, which is sharper than the general upper bound in~\cite{ostroumova2016}.

 We further show that the results on Pearson's correlation coefficient and the clustering coefficient for the erased configuration model can easily be extended to another important random graph null model for simple scale-free networks: the rank-1 inhomogeneous random graph~\cite{chung2002,boguna2003}. In this model, every vertex is equipped with a weight $w_i$, and vertices are connected independently with some connection probability $p(w_i,w_j)$. We show that for a wide class of connection probabilities, the rank-1 inhomogeneous random graph also has structurally negative degree correlations, satisfying the same central limit theorem as in the erased configuration model. Furthermore, we show that for the particular choice $p(w_i,w_j)=1-\me^{-w_iw_j/(\mu n)}$, where $\mu$ denotes the average weight, the clustering coefficient behaves asymptotically the same as in the erased configuration model.

Note that in the (erased) configuration model as well as the inhomogeneous random graph, Pearson's coefficient for degree correlations and the global clustering coefficient naturally converge to zero. We would like to emphasize that this paper improves on the existing literature by establishing the {\it scaling laws} that govern the convergence of these statistics to zero. This is important because very commonly in the literature, various quantities measured in real-world networks are compared to null-models with same degrees but random rewiring. These rewired null-models are similar to a version of the inhomogeneous random graph~\cite{bannink2018,gao2018}. Without knowing the scaling of these quantities in the inhomogeneous random graph, it is not possible to asses how similar a small measured value on the real network is to that of the null model. Our results enable such analysis. In fact, we do even more, by also establishing exact limiting distributions of the rescaled Pearson's correlations coefficient and clustering coefficient, which are the most standard measures in statistical analysis of networks.

\subsection{Outline of the paper}
The remainder of paper is structured as follows. In the next three sections we formally introduce the models, the measures of interest and some additional notation. Then, in Section \ref{ssec:results}, we summarize our main results and discuss important insights obtained from them. We give a heuristic outline of our proof strategy in Section \ref{sec:proof_overview} and recall several results for regularly-varying degrees. Then we proceed with proving our result for Pearson's correlation coefficient in Section \ref{sec:pearson} and the clustering coefficient in Section \ref{sec:clustering}. We then show  in Section~\ref{sec:proof_hvm} how the proofs for Pearson's correlation coefficient and the clustering coefficient in the erased configuration model can be adapted to prove the central limit theorems for the rank-1 inhomogeneous random graph. Finally, Appendix \ref{sec:appendix} contains the proof for the number of erased edges, Theorem \ref{thm:scaling_erased_edges}, as well as some additional technical results. 

\subsection{Configuration model with scale-free degrees}

The first models of interest in this work are the configuration model and the erased configuration model. 
Given a vertex set $[n] := \{1, 2, \dots, n\}$ and a sequence ${\bf D}_n = \{D_1, D_2, \dots, D_n\}$, whose sum $\sum_{i \in [n]} D_i$ is even, the configuration model (\texttt{CM}) constructs a graph $G_n$ with this degree sequence by assigning to each node $i$, $D_i$ stubs and then proceeds by connecting stubs at random to form edges. This procedure will, in general, create a multi-graph with self-loops and multiple edges between two nodes. To make sure that the resulting graph is simple we can remove all self-loops and replace multiple edges between nodes by just one edge. This model is called the erased configuration model (\texttt{ECM}). 

We will denote by $\texttt{\texttt{CM}}({\bf D}_n)$ and $\texttt{ECM}({\bf D}_n)$ graphs generated by, 
respectively, the standard and erased configuration model, starting from the degree sequence ${\bf D}_n$. We often couple both constructions by first constructing a graph via the standard configuration model and then remove all the self-loops and multiple edges to create the erased configuration model.  In this case we  write $G_n$ for the graph created by the \texttt{CM} and $\widehat{G}_n$ for the \texttt{ECM} graph constructed from $G_n$. In addition we use the hats to distinguish between objects in the \texttt{CM} and the \texttt{ECM}. For example, $D_i$ denotes the degree of node $i$ in the graph $\texttt{CM}({\bf D}_n)$ while its degree in $\texttt{ECM}({\bf D}_n)$ is denoted by $\widehat{D}_i$. 
 
We consider degree sequences ${\bf D}_n = \{D_1, D_2, \dots, D_n\}$ where the degrees $D_i$ are i.i.d. copies of an 
integer-valued random variable $\mathcal{D}$ with regularly-varying distribution
\begin{equation}
	\Prob{\mathscr{D} > t} = \mathcal{L}(t) t^{-\gamma}, \quad \gamma > 1.
\label{eq:distribution_degrees}
\end{equation}
Here, the function $\mathcal{L}(t)$ is slowly varying at infinity and $\gamma$ is the exponent of the 
distribution. 

As commonly done in the literature, $D_n$ may include a correction term, equal to one, in order to make the sum $L_n = \sum_{i \in [n]} D_i$ even. We shall ignore this correction term, since it does not effect the asymptotic results. In the remainder of this paper $\mathscr{D}$ always refers to a random variable with distribution \eqref{eq:distribution_degrees}.

\subsection{Rank-1 inhomogeneous random graphs}
Another model that generates networks with scale-free degrees is the rank-1 inhomogeneous random graph~\cite{chung2002,boguna2003}. In this model, every vertex is equipped with a weight. We assume these weights are an i.i.d.\ sample from the scale-free distribution~\eqref{eq:distribution_degrees}. Then, vertices $i$ and $j$ with weights $w_i$ and $w_j$ are connected with some connection probability $p(w_i,w_j)$. Let the expected value of~\eqref{eq:distribution_degrees} be denoted by $\mu$. We then assume the following conditions on the connection probabilities, similarly to~\cite{hofstad2017b}:
\begin{condition}[Class of connection probabilities]\label{cond:conhvm}
Assume that
\[
	p(w_i,w_j)=\frac{w_iw_j}{\mu n}h\left(\frac{w_iw_j}{\mu n}\right),
\]
for some continuous function $h:[0,\infty)\mapsto [0,1]$ satisfying
\begin{enumerate}[(i)]
	\item $h(0)=1$ and $h(u)$ decreases to 0 as $u\to\infty$,
	\item $q(u) = uh(u)$ increases to 1 as $u\to\infty$.
	\item There exists $u_1>0$ such that $h(u)$ is differentiable on $(0,u_1]$ and $h'_+(0)<\infty$.
\end{enumerate}
\end{condition}

This class includes the commonly used connection probabilities $q(u)=(u \wedge 1)$, where $(x \wedge y) := \min(x,y)$ (the Chung Lu setting)~\cite{chung2002}, $q(u)=1-\me^{-u}$ (the Poisson random graph)~\cite{norros2006} and $q(u)=u/(1+u)$ (the maximal entropy random graph)~\cite{park2004,colomer2012,hoorn2017sparse}.
Note that within the class of connection probabilities satisfying Condition~\ref{cond:conhvm}, $q(u)\leq (u \wedge 1)$. Note that $p(w_i,w_j) = q\left(\frac{w_iw_j}{\mu n}\right)$.

\subsection{Central quantities}

Pearson's correlations coefficient $r(G_n)\in[-1,1]$ is a measure for degree-degree correlations. For an 
undirected multigraph $G_n$, this measure is defined as (see \cite{Hofstad2014}),
\begin{equation}
r(G_n) = \frac{\sum_{i,j \in [n]} X_{ij} D_i D_j - \frac{1}{L_n} \left(\sum_{i \in [n]} D_i^2\right)^2}
{\sum_{i \in [n]} D_i^3 - \frac{1}{L_n} \left(\sum_{i \in [n]} D_i^2\right)^2},
\label{eq:pearson}
\end{equation}
where $X_{ij}$ denotes the number of edges between nodes $i$ and $j$ in $G_n$ and self-loops are counted twice.  We write $r_n$ for Pearson's correlation coefficient on $G_n$ generated by \texttt{CM} and $\widehat{r}_n$ if $G_n$ is generated by \texttt{ECM}.

The clustering coefficient of graph $G_n$ is defined as
\begin{equation}\label{eq:clustdef}
C(G_n)=\frac{3\triangle_n}{\text{number of connected triples}},
\end{equation}
where $\triangle_n$ denotes the number of triangles in the graph. The clustering coefficient can be written as
\begin{equation}\label{eq:clust}
C(G_n)=\frac{6\triangle_n}{\sum_{i \in [n]} D_i(D_i-1)}=\frac{6\sum_{1\leq i< j< k\leq n} X_{ij}X_{jk}X_{ik}}{\sum_{i \in [n]} D_i(D_i-1)},
\end{equation}
where $X_{ij}$ again denotes the number of edges between vertex $i$ and $j$ in $G_n$. For simple graphs, $C(G_n) \in [0,1]$. However, for multigraphs, $C(G_n)$ may exceed one.
As with Pearson's correlation coefficient, we denote by $C_n$ the clustering coefficient in $G_n$ generated by \texttt{CM}, while $\widehat{C}_n$ is the clustering coefficient in $G_n$ generated by \texttt{ECM}.

\subsection{Notation}\label{sec:notation}
We write $\mathbb{P}_n$ and $\mathbb{E}_n$ for, respectively, the conditional probability and expectation, with respect to the sampled degree sequence ${\bf D}_n$. We use $\overset{d}\longrightarrow$ for convergence in distribution, and $\plim $ for convergence in probability. We say that a sequence of events $(\mathcal{E}_n)_{n\geq 1}$ happens with high probability (w.h.p.) if $\lim_{n\to\infty}\Prob{\mathcal{E}_n}=1$. Furthermore, we write $f(n)=o(g(n))$ if $\lim_{n\to\infty}f(n)/g(n)=0$, and $f(n)=O(g(n))$ if $|f(n)|/g(n)$ is uniformly bounded, where $(g(n))_{n\geq 1}$ is nonnegative. 

We say that $X_n=\bigOp{g(n)}$ for a sequence of random variables $(X_n)_{n\geq 1}$ if $|X_n|/g(n)$ is a tight sequence of random variables, and $X_n=\op(g(n))$ if $X_n/g(n)\plim 0$. Finally, we use $(x \wedge y)$ to denote the minimum of $x$ and $y$ and $(x\vee y)$ to denote the maximum of $x$ and $y$.

\subsection{Results}\label{ssec:results}

In this paper we study the interesting regime when $1<\gamma<2$, so that the degrees have finite mean but infinite variance. When $\gamma > 2$, the number of removed edges is constant in $n$ and hence asymptotically there will be no difference between the \texttt{CM} and \texttt{ECM}. We establish a new asymptotic upper bound for the number of erased edges in the \texttt{ECM} and prove new limit theorems for Pearson's correlation coefficient and the clustering coefficient. We further show that the limit theorems for Pearson and clustering for the inhomogeneous random graph are very similar to the ones obtained for the \texttt{ECM}.

Our limit theorems involve random variables with joint stable distributions, which we define as follows.
Let 
\begin{equation}\label{eq:def_Gamma}
\Gamma_i=\sum_{j=1}^i\xi_j \hspace{20pt} i \ge 1,
\end{equation} with $(\xi_j)_{j\geq 1}$ i.i.d exponential random variables with mean 1. Then we define,
for any integer $p \ge 2$, 
\begin{equation}\label{eq:def_S_gamma}
\mathcal{S}_{\gamma/p} = \sum_{i=1}^{\infty}\Gamma_i^{-p/\gamma}.
\end{equation}
We remark that for any $\alpha > 1$ we have that $\sum_{i=1}^{\infty}\Gamma_i^{-\alpha}$ has a stable distribution with stability index $\alpha$ (see \cite[Theorem 1.4.5]{samorodnitsky1994}).

In the remainder of this section we will present the theorems and highlight their most important aspects in view of the methods and current literature. We start with $\widehat{r}_n$.

\begin{theorem}[Pearson in the \texttt{ECM}]\label{thm:clt_pearson_ecm}
Let ${\bf D}_n$ be sampled from $\mathscr{D}$ with $1 < \gamma < 2$ and $\Exp{\mathscr{D}} = \mu$. Then, if $G_n = \emph{\texttt{ECM}}({\bf D}_n)$, there exists a slowly-varying function $\mathcal{L}_1$ such that,
\[
	\mu\mathcal{L}_1(n)n^{1 - \frac{1}{\gamma}} \, \widehat{r}_n \dlim -\frac{\mathcal{S}_{\gamma/2}^2}{\mathcal{S}_{\gamma/3}},
\]
where $\mathcal{S}_{\gamma/2}$ and $\mathcal{S}_{\gamma/3}$ are given by \eqref{eq:def_S_gamma}.
\end{theorem}

The following theorem shows that the correlation coefficient for all rank-1 inhomogeneous random graphs satisfying Condition~\ref{cond:conhvm} behaves the same as in the erased configuration model:
\begin{theorem}[Pearson in the rank-1 inhomogeneous random graph]\label{thm:clt_pearson_hvm}
Let ${\bf W}_n$ be sampled from $\mathscr{D}$ with $1 < \gamma < 2$ and $\Exp{\mathscr{D}} = \mu$. Then, when $G_n$ is a rank-1 inhomogeneous random graph with weights ${\bf W}_n$ and connection probabilities satisfying Condition~\ref{cond:conhvm}, there exists a slowly-varying function $\mathcal{L}_1$ such that,
\[
	\mu\mathcal{L}_1(n)n^{1 - \frac{1}{\gamma}} \, r(G_n) \dlim -\frac{\mathcal{S}_{\gamma/2}^2}{\mathcal{S}_{\gamma/3}},
\]
where $\mathcal{S}_{\gamma/2}$ and $\mathcal{S}_{\gamma/3}$ are given by \eqref{eq:def_S_gamma}.
\end{theorem}

Interestingly, the behavior of Pearson's correlation coefficient in the rank-1 inhomogeneous random graph does not depend on the exact form of the connection probabilities, as long as these connection probabilities satisfy Condition~\ref{cond:conhvm}.

\paragraph*{Asymptotically vanishing correlation coefficient.}
It has been known for some time, c.f. \cite[Theorem 3.1]{Hofstad2014}, that when the degrees ${\bf D}_n$ are sampled from a degree distribution with infinite third moment, any limit point of Pearson's correlation coefficient is non-positive. Theorem \ref{thm:clt_pearson_ecm} confirms this, showing that for the erased configuration model, with infinite second moment, the limit is zero. Moreover, Theorem \ref{thm:clt_pearson_ecm} gives the exact scaling in terms of the graph size $n$, which has not been available in the literature. Compare e.g. to \cite[Theorem 5.1]{Hoorn2016}, where only the scaling of the negative part of $\widehat{r}_n$ is given.

\paragraph*{Structural negative correlations.}
It has also been observed many times that imposing the requirement of simplicity on graphs gives rise to so-called structural negative correlations, see e.g. \cite{Hoorn2015PhysRev,Catanzaro2005,yao2017ANND,stegehuis2017b}. Our result is the first theoretical confirmation of the existence of structural negative correlations, as a result of the simplicity constraint on the graph. To see this, note that the distributions of the random variables $\mathcal{S}_{\gamma/2}$ and $\mathcal{S}_{\gamma/3}$ have support on the positive real numbers. Therefore, Theorem \ref{thm:clt_pearson_ecm} shows that when we properly rescale Pearson's correlation coefficient in the erased configuration model, the limit is a random variable whose distribution only has support on the negative real numbers. Note that this result implies that when multiple instances of \texttt{ECM} graphs are generated and Pearson's correlation coefficient is measured, the majority of the measurements will yield negative, although small, values. These small values have nothing to do with the network structure but are an artifact of the constraint that the resulting graph is simple. Interestingly, Theorem~\ref{thm:clt_pearson_hvm} shows that the same result holds for rank-1 inhomogeneous random graphs, indicating that structural negative correlations also exist in these models and thus further supporting the explanation that such negative correlations result from the constraint that the graphs are simple.

\paragraph*{Pearson in \texttt{ECM} versus \texttt{CM}.}
Currently we only have a limit theorem for the erased model, in the scale-free regime $1 < \gamma < 2$. 
Interestingly, and also somewhat unexpectedly, proving a limit theorem for \texttt{CM}, which is a simpler model, turns out to be more involved. The main reason for this is that in the \texttt{ECM}, the positive part of $r(G_n)$, determined by $\sum_{i,j} X_{ij} D_i D_j$ is negligible with respect to the other term since a polynomial in $n$ number of edges are removed (see Section \ref{ssec:pearson_ecm_heuristics} for more details). Therefore, the negative part determines the distribution of the limit. In the \texttt{CM} this is no longer true and hence the distribution is determined by the  tricky balance between the positive and the negative term, and their fluctuations. This requires more involved methods to analyze than we could develop so far. \color{red} Below, we state a conjecture about this case, and state a partial result concerning the scaling of its variance that supports this conjecture:

\begin{conjecture}[Scaling Pearson for \texttt{CM}]
\label{conj-Pearson-CM}
As $n\rightarrow \infty$, there exists some {\em random} $\sigma^2$ such that
	\eqn{
	\label{CLT-Pearson-CM}
	\sqrt{n}r_n \dlim \mathcal{N}(0,\sigma^2).
	}
\end{conjecture}

The intuition behind this conjecture is explained in Section~\ref{sec:int_conj}. 
Although we do not have a proof of this scaling limit of $r_n$ in the configuration model, the following result that shows that at least $\sqrt{n}r_n$ is a tight sequence of random variables:

\begin{lemma}[Convergence of $n {\rm Var}_n(r_n)$ for  \texttt{CM}]
\label{lem-Pearon-CM-Var}
As $n\rightarrow \infty$, with ${\rm Var}_n$ denoting the conditional variance given the i.i.d.\ degrees,
	\eqn{
		n{\rm Var}_n(r_n)\dlim \frac{2-\mathcal{S}_{\gamma/6}/\mathcal{S}_{\gamma/3}^2}{\mu}.
	}
\end{lemma}
	
\color{black}
\bigskip
We now present our results for the clustering coefficient. The following theorem gives a central limit theorem for the clustering coefficient in the configuration model:

\begin{theorem}[Clustering in the \texttt{CM}]\label{thm:CCM}
Let ${\bf D}_n$ be sampled from $\mathscr{D}$ with $1< \gamma < 2$ and $\Exp{\mathscr{D}} = \mu$. Then, if $G_n = \emph{\texttt{CM}}({\bf D}_n)$, there exists a slowly-varying function $\mathcal{L}_2$ such that
\begin{equation}\label{eq:ccm}
\frac{ C_n}{\mathcal{L}_2(n)n^{4/\gamma-3}}\dlim \frac{1}{\mu^3}\left(\tilde{C}_{\gamma/2}\mathcal{S}_{\gamma/2}^2-3\tilde{C}_{\gamma/4}\mathcal{S}_{\gamma/4}+\frac{2\tilde{C}_{\gamma/6}\mathcal{S}_{\gamma/6}}{\tilde{C}_{\gamma/2}\mathcal{S}_{\gamma/2}}\right),
\end{equation}
where $\mathcal{S}_{\gamma/2}$, $\mathcal{S}_{\gamma/4}$ and $\mathcal{S}_{\gamma/6}$ are given by \eqref{eq:def_S_gamma} and
\[
\tilde{C}_\alpha=\Big(\frac{1-\alpha}{\Gamma(2-\alpha)\cos(\pi\alpha/2)}\Big)^\alpha,
\]
with $\Gamma$ the Gamma-function.
\end{theorem}

\paragraph*{Infinite clustering.}
For $\gamma<4/3$, Theorem~\ref{thm:CCM} shows that $C_n$ tends to infinity. This observation shows that the global clustering coefficient may give nonphysical behavior when used on multi-graphs. In multi-graphs, several edges may close a triangle. In this case, the interpretation of the clustering coefficient as a fraction of connected triples does not hold. Rather, the clustering coefficient can be interpreted as the average number of closed triangles per wedge, where different wedges and triangles may involve the same nodes but have different edges between them. This interpretation shows that indeed in a multi-graph the clustering coefficient may go to infinity.

\paragraph*{What is a triangle?}
The result in Theorem~\ref{thm:CCM} depends on what we consider to be a triangle. In general, one can think of a triangle as a loop of length three. In the configuration model however, self-loops and multiple edges may be present. Then for example three self-loops at the same vertex also form a loop of length three. Similarly, a multiple edge between vertices $v$ and $w$ together with a self-loop at vertex $w$ can also form a loop of length three. In Theorem~\ref{thm:CCM}, we do not consider these cases as triangles. Excluding these types of ``triangles" gives the terms $\mathcal{S}_{\gamma/4}$ and $\mathcal{S}_{\gamma/6}/\mathcal{S}_{\gamma/2}$ in Theorem~\ref{thm:CCM}.
\\\\
To obtain the precise asymptotic behavior of the clustering coefficient in the erased configuration model, we need an extra assumption on the degree distribution~\eqref{eq:distribution_degrees}.
\begin{assumption}\label{ass:F}
The degree distribution \eqref{eq:distribution_degrees} satisfies for all $x\in\{1, 2, \dots \}$ and for some $K>0$
\[
	\Prob{\mathcal{D}=x}\leq K \mathcal{L}(x) x^{-\gamma-1}.
\]
\end{assumption} 

Note that for all $t \ge 2$
\[
	\Prob{\mathcal{D} = t} = \Prob{\mathcal{D} > t - 1} - \Prob{\mathcal{D} > t} = \mathcal{L}(t-1) (t-1)^{-\gamma} - \mathcal{L}(t) t^{-\gamma},
\] 
Hence, since $(t - 1)^{-\gamma} - t^{-\gamma} \sim \gamma t^{-\gamma - 1}$ as $t \to \infty$, it follows that Assumption \ref{ass:F} is satisfied whenever the slowly-varying function $\mathcal{L}(t)$ is monotonic increasing for all $t$ greater than some $T$.

\begin{theorem}[Clustering in the \texttt{ECM}]\label{thm:CECM}
Let ${\bf D}_n$ be sampled from $\mathscr{D}$, satisfying Assumption \ref{ass:F}, with $1 < \gamma < 2$ and $\Exp{\mathscr{D}} = \mu$. Then, if $G_n = \emph{\texttt{ECM}}({\bf D}_n)$, there exists a slowly-varying function $\mathcal{L}_3$ such that
\[
	\frac{\mathcal{L}_3(n)  \widehat{C}_n}{\mathcal{L}(\sqrt{\mu n})^3n^{(-3\gamma^2+6\gamma-4)/(2\gamma)}}\dlim \mu^{-\frac 32 \gamma }\frac{A_\gamma}{\mathcal{S}_{\gamma/2}},
\]
where $\mathcal{S}_{\gamma/2}$ is a stable random variable defined in \eqref{eq:def_S_gamma}, and 
\begin{equation}\label{eq:Agamma}
A_{\gamma}=\int_{0}^{\infty}\int_{0}^{\infty}\int_{0}^{\infty}\frac{1}{(xyz)^{\gamma+1}}(1-\me^{-xy})(1-\me^{-xz})(1-\me^{-yz})\dd x\dd y\dd z<\infty.
\end{equation}
\end{theorem}

We now investigate the behavior of the clustering coefficient in rank-1 inhomogeneous random graphs:
\begin{theorem}[Clustering in the rank-1 inhomogeneous random graph]\label{thm:CHVM}
Let ${\bf W}_n$ be sampled from $\mathscr{D}$, satisfying Assumption \ref{ass:F}, with $1 < \gamma < 2$ and $\Exp{\mathscr{D}} = \mu$. Then, if $G_n$ is an inhomogeneous random graph with weights ${\bf W}_n$ and connection probabilities satisfying Condition~\ref{cond:conhvm}, there exists a slowly-varying function $\mathcal{L}_3$ such that
\[
\frac{\mathcal{L}_3(n) C(G_n)}{\mathcal{L}(\sqrt{\mu n})^3n^{(-3\gamma^2+6\gamma-4)/(2\gamma)}}\dlim \mu^{-\frac 32 \gamma }\frac{1}{\mathcal{S}_{\gamma/2}}\int_{0}^{\infty} \hspace{-5pt} \int_{0}^{\infty} \hspace{-5pt} \int_{0}^{\infty} \frac{q(xy)q(xz)q(yz)}{(xyz)^{\gamma+1}}\dd x\dd y\dd z,
\]
where $q$ is as in Condition~\ref{cond:conhvm}(ii), $\mathcal{S}_{\gamma/2}$ is a stable random variable defined in \eqref{eq:def_S_gamma}, and 
\[
\int_{0}^{\infty}\int_{0}^{\infty}\int_{0}^{\infty}\frac{1}{(xyz)^{\gamma+1}}q(xy)q(xz)q(yz)\dd x\dd y\dd z<\infty.
\]
\end{theorem}

\begin{figure}[htbp]
\centering
\includegraphics[width=0.5\textwidth]{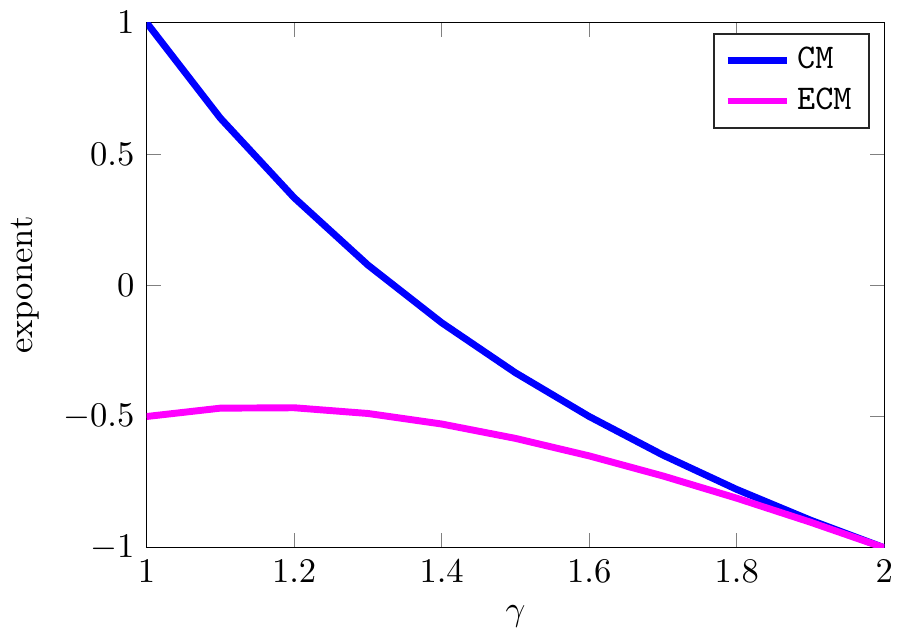}
\caption{Exponents of the clustering coefficient of $\texttt{CM}$ and $\texttt{ECM}$ for $\gamma\in(1,2)$}
\label{fig:exps}
\end{figure}

\paragraph*{Maximal clustering in the \texttt{ECM} and the inhomogeneous random graph.}
Figure~\ref{fig:exps} shows the exponents of $n$ in the main multiplicative term of the clustering coefficient, in the \texttt{CM} and the \texttt{ECM}. The exponent in Theorem~\ref{thm:CECM} is a  quadratic expression in $\gamma$, hence, there may be a value of $\gamma$ that maximizes the clustering coefficient. We set the derivative of the exponent equal to zero
\[
\frac{\dd}{\dd \gamma}(-3\gamma^2+6\gamma-4)/(2\gamma) = -3/2+2\gamma^{-2} =0,
\]
which solves to $\gamma = \sqrt{4/3}\approx 1.15$. Thus, the global clustering coefficient of an erased configuration model with $\gamma\in(1,2)$ is maximal for $\gamma\approx1.15$ where the scaling exponent of the clustering coefficient equals $-2\sqrt{3}+3\approx-0.46$. This maximal value arises from the trade off between the denominator and the numerator of the clustering coefficient in~\eqref{eq:clustdef}. When $\gamma$ becomes close to 1, there will be some vertices with very high degrees. This makes the denominator of~\eqref{eq:clustdef} very large. On the other hand, having more vertices of high degrees also causes the graph to contain more triangles. Thus, the numerator of~\eqref{eq:clustdef} also increases when $\gamma$ decreases. The above computation shows that in the erased configuration model, the optimal trade off between the number of triangles and the number of connected triples is attained at $\gamma\approx1.15$. Theorem~\ref{thm:CHVM} shows that the same phenomenon occurs in the rank-1 inhomogeneous random graph.

\paragraph*{Mean clustering in \texttt{CM} vs \texttt{ECM}.}
In the \texttt{CM}, the normalized clustering coefficient converges to a constant times a stable random variable squared. This stable random variable has an infinite mean, and therefore its square also has an infinite mean. In the \texttt{ECM} as well as in the rank-1 inhomogeneous random graph however, the normalized clustering coefficient converges to one divided by a stable random variable, which has a finite mean~\cite{pierce1997}. Thus, the rescaled clustering coefficient in the \texttt{ECM} and the rank-1 inhomogeneous random graph converges to a random variable with finite mean. Formally, $\Exp{\frac{ C_n}{n^{4/\gamma-3}}}=\infty$ and $\Exp{\frac{\widehat{C}_n}{n^{-3/2\gamma+3-2/\gamma}}}<\infty$.

\paragraph*{ECM and inhomogeneous random graphs.}
Theorems~\ref{thm:CECM} and~\ref{thm:CHVM} show that the clustering coefficient in the erased configuration model has the same scaling as the clustering coefficient in the rank-1 inhomogeneous random graph. In fact, choosing $q(u)=1-\me^{-u}$ in Condition~\ref{cond:conhvm} even gives the exact same behavior for clustering in the erased configuration model and in the inhomogeneous random graph. This shows that the erased configuration model behaves similarly as an inhomogeneous random graph with connection probabilities $p(w_i,w_j)=1-\me^{-w_iw_j/(\mu n)}$ in terms of clustering.

\paragraph*{Vertices of degrees $\sqrt{n}$.}
In the proof of Theorem~\ref{thm:CECM} we show that the main contribution to the number of triangles comes from vertices of degrees proportional to $\sqrt{n}$. Let us explain why this is the case. In the \texttt{ECM}, the probability that an edge exists between vertices $i$ and $j$ can be approximated by $1-\me^{-D_iD_j/L_n}$. Therefore, when $D_iD_j$ is proportional to $n$, the probability that an edge between $i$ and $j$ exists is bounded away from zero. Similarly, the probability that a triangle between vertices $i$,$j$ and $k$ exists is bounded away from zero as soon as $D_iD_j$, $D_iD_k$ and $D_jD_k$ are all proportional to $n$. This is indeed achieved when all three vertices have degrees proportional to $\sqrt{n}$. 
If, for example, vertex $i$ has degree of the order larger than $\sqrt{n}$, this means that vertices $j$ and $k$ can have degrees of the order smaller than $\sqrt{n}$ while $D_iD_j$ and $D_iD_k$ are still of order $n$. However, $D_jD_k$ also has to be of size $n$ for the probability of a triangle to be bounded away from zero. Now recall that the degrees follow a power-law distribution. Therefore, the probability that a vertex has degree much higher than $\sqrt{n}$ is much smaller than the probability that a vertex has degree of the order $\sqrt{n}$.  Thus, the most likely way for all three contributions to be proportional to $n$ is to have $D_i,D_j,D_k$ be proportional to $\sqrt{n}$.  Intuitively, this shows that the largest contribution to the number of triangles in the \texttt{ECM} comes from the vertices of degrees proportional to $\sqrt{n}$. This balancing of the number of vertices and the probability of forming a triangle also appears for other subgraphs~\cite{hofstad2017d}.

\paragraph*{Global and average clustering.}
Clustering can be measured by two different metrics: the global clustering coefficient and the average clustering coefficient~\cite{newman2001,watts1998}. In this paper, we study the global clustering coefficient, as defined in ~\eqref{eq:clustdef}. The average clustering coefficient is defined as the average over the local clustering coefficient of every vertex, where the local clustering coefficient of a vertex is the number of triangles the vertex participates in divided by the number of pairs of neighbors of the vertex. For the configuration model, the global clustering coefficient as well as the average clustering coefficient are known to scale as $n^{4/3-\gamma}$~\cite{newman2003book}. In particular, this shows that both clustering coefficients in the configuration model diverge when $\gamma<4/3$. Our main results, Theorems~\ref{thm:CCM} and~\ref{thm:CECM}, provide the exact limiting behavior of the global clustering coefficients for CM and ECM, respectively.

The average clustering coefficient in the rank-1 inhomogeneous random graph has been shown to scale as $n^{1-\gamma}\log(n)$~\cite{colomer2012,hofstad2017b}, which is very different from the scaling of the global clustering coefficient from Theorem~\ref{thm:CECM}. For example, the average clustering coefficient decreases in $\gamma$, whereas the global clustering coefficient first increases in $\gamma$, and then decreases in $\gamma$ (see Figure~\ref{fig:exps}). Furthermore, the average clustering coefficient decays only very slowly in $n$ as $\gamma$ approaches 1. The global clustering coefficient on the other hand decays as $n^{-1/2}$ when $\gamma$ approaches 1. This shows that the global clustering coefficient and the average clustering coefficient are two very different ways to characterize clustering.  

\paragraph{Joint convergence.}
Before we proceed with the proofs, we remark that each of the three limit theorems uses a coupling between the sum of different powers of degrees and the limit distributions $\mathcal{S}_{\gamma/p}$. It follows from the proofs of our main results, that these couplings hold simultaneously for all three measures. As a direct consequence, it follows that the rescaled measures convergence jointly in distribution:

\begin{theorem}[Joint convergence]\label{thm:sim_clt}
Let ${\bf D}_n$ be sampled from $\mathscr{D}$, satisfying Assumption \ref{ass:F}, with $1 < \gamma < 2$ and $\Exp{\mathscr{D}} = \mu$. Let $G_n = \emph{\texttt{CM}}({\bf D}_n)$, $\widehat{G}_n = \emph{\texttt{ECM}}({\bf D}_n)$ and define
$\alpha = (-3\gamma^2 + 6\gamma -4)/2\gamma$. Then there exist slowly-varying functions $\mathcal{L}_1$, $\mathcal{L}_2$ and $\mathcal{L}_3$, such that as $n \to \infty$,
\[
	\left(\hspace{-2pt}\mathcal{L}_1(n)n^{1 - \frac{1}{\gamma}} \, \widehat{r}_n, \, \frac{ C_n}{\mathcal{L}_2(n)n^{4/\gamma-3}}, \,
	\frac{\mathcal{L}_3(n)  \widehat{C}_n}{\mathcal{L}(\sqrt{\mu n})^3n^{\alpha}}\right) \dlim
	\left(\hspace{-2pt}-\frac{\mathcal{S}_{\gamma/2}^2}{\mu \mathcal{S}_{\gamma/3}}, \,
	\frac{\mathcal{S}_{\gamma/2}^2-\mathcal{S}_{\gamma/4}}{\mu^3}, \,
	\mu^{-\frac{3}{2} \gamma }\frac{A_\gamma}{\mathcal{S}_{\gamma/2}}\right)\hspace{-2pt},
\]
with $A_\gamma$ as in \eqref{eq:Agamma} and  $\mathcal{S}_{\gamma/2}$, $\mathcal{S}_{\gamma/3}$ and $\mathcal{S}_{\gamma/4}$  given by \eqref{eq:def_S_gamma}.
\end{theorem}

\section{Overview of the proofs}\label{sec:proof_overview}

Here we give an outline for the proofs of our main results for the configuration model and the erased configuration model and explain the main ideas that lead to them. Since the goal is to convey the high-level ideas, we limit technical details in this section and often write $f(n) \approx g(n)$ to indicate that $f(n)$ behaves roughly as $g(n)$. The formal definition and exact details of these statements can be found in Sections~\ref{sec:pearson} and~\ref{sec:clustering} where the proofs are given. The proofs for the rank-1 inhomogeneous random graphs follow very similar lines, and we show how the proofs for the erased configuration model extend to rank-1 inhomogeneous random graphs satisfying Condition~\ref{cond:conhvm} in Section~\ref{sec:proof_hvm}. We start with some results on the number of removed edges in the erased configuration model.


\subsection{The number of removed edges}
\label{sec:erased_edges}


The number of removed edges $Z_n$ in the erased configuration model is given by
\[
	Z_n = \sum_{i = 1}^n X_{ii} + \sum_{1 \le i < j \le n}\left( X_{ij} - \ind{X_{ij} > 0}\right),
\]
where $X_{ij}$ again denotes the number of edges between vertices $i$ and $j$.
For the analysis of the \texttt{ECM} it is important to understand the behavior of this number. In particular we are interested in the scaling of $Z_n$ with respect to $n$. Here we give an asymptotic upper bound, which implies that, up to sub-linear terms, $Z_n$ scales no faster than $n^{2 - \gamma}$. The proof can be found in Section \ref{ssec:proof_erased_edges}.

\begin{theorem}\label{thm:scaling_erased_edges}
Let ${\bf D}_n$ be sampled from $\mathscr{D}$ with $1< \gamma < 2$ and $\widehat{G}_n = \emph{\texttt{ECM}}({\bf D}_n)$. Then for any $\delta > 0$,
\[
	\frac{Z_n}{n^{2 - \gamma + \delta}} \plim 0
\]
\end{theorem}

The scaling $n^{2-\gamma}$ is believed to be sharp, up to some slowly-varying function. We therefore conjecture that for any $\delta > 0$,
\[
	\frac{Z_n}{n^{2 - \gamma - \delta}} \plim \infty.
\]

From Theorem \ref{thm:scaling_erased_edges} we obtain several useful results, summarized in the corollary below. The proof can be found in Section \ref{ssec:proof_erased_edges}. Let $Z_{ij}$ be the number of edges between $i$ and $j$ that have been removed and $Y_i$ the number of removed stubs of node $i$. Then we have
\[
	Y_i = \sum_{j = 1}^n Z_{ij} = X_{ii} + \sum_{j \ne i}\left( X_{ij} - \ind{X_{ij} > 0}\right).
\]

\begin{corollary}\label{cor:scaling_degrees_erased_edges}
Let ${\bf D}_n$ be sampled from $\mathscr{D}$ with $1< \gamma < 2$ and $G_n = \emph{\texttt{ECM}}({\bf D}_n)$. Then, for any integer $p \ge 0$ and $\delta > 0$,
\[
	\frac{\sum_{i = 1}^n D_i^p Y_i}{n^{\frac{p}{\gamma} + 2 - \gamma + \delta}} \plim 0 \quad \text{and}
	\quad \frac{\sum_{1 \le i < j \le n} Z_{ij} D_i D_j}{n^{\frac{2}{\gamma} + 2 - \gamma + \delta}} \plim 0.
\]
\end{corollary}

The first result of Corollary \ref{cor:scaling_degrees_erased_edges} gives the scaling of the difference of the sum of powers of degrees, between \texttt{CM} and \texttt{ECM}. To see why, note that since $\widehat{D}_i^q = (D_i - Y_i)^q$ and $Y_i \le D_i$, by the mean value theorem we have, for any integer $q \ge 1$,
\[
	\left|\sum_{i=1}^n D_i^q - \sum_{i=1}^n \widehat{D}_i^q\right| \le q \sum_{i = 1}^n D_i^{q-1} Y_i.
\]
Hence, for any $q \ge 1$ and $\delta > 0$,
\[
	\left|\sum_{i=1}^n D_i^q - \sum_{i=1}^n \widehat{D}_i^q\right| = o_\pr\left(n^{\frac{q-1}{\gamma} + 2 -\gamma +\delta}\right).
\]

\subsection{Results for regularly-varying random variables}

In addition to the number of edges, we shall make use of several results, regarding the scaling of expressions with regularly-varying random variables. We summarize them here, starting with a concentration result on the sum of i.i.d. samples, which is a direct consequence of the Kolmogorov-Marcinkiewicz-Zygmund strong law of large numbers.
\begin{lemma}\label{lem:Kolmogorov_SLLN}
Let $(X_i)_{i \ge 1}$ be independent copies from a non-negative regularly-varying random 
variable $X$ with exponent $\gamma > 1$ and mean $\mu$. Then, with $\kappa = (\gamma - 1)/(1 + \gamma)$,
\[
	\frac{\left|\mu n - \sum_{i = 1}^n X_i\right|}{n^{1 - \kappa}} \plim 0.
\]
\end{lemma}

In particular, since $L_n = \sum_{i = 1}^n D_i$, with all $D_i$ being i.i.d. with regularly-varying distribution \eqref{eq:distribution_degrees} and mean $\mu$, it holds that $n^{\kappa - 1}\left|L_n - \mu n\right| \plim 0$. Therefore, the above lemma allows us to replace $L_n$ with $\mu n$ in our expressions.

The next proposition gives the scaling of sums of different powers of independent copies of a regularly-varying random variable. Recall that $(x \vee y)$ denotes the maximum of $x$ and $y$.

\begin{proposition}[\hspace{-0.5ex}{\cite[Proposition 2.4]{Hoorn2016}}]\label{prop:scaling_sums_powers_regularly_varying}
Let $(X_i)_{i \ge 1}$ be independent copies of a non-negative regularly-varying random 
variable $X$ with exponent $\gamma > 1$. Then,
\begin{enumerate}[\upshape i)]
	\item for any integer $p \ge 1$ and $\delta > 0$,
	\[
		\frac{\sum_{i = 1}^n X_i^p}{n^{\left(\frac{p}{\gamma} \vee 1\right) + \delta}} \plim 0;
	\]
	\item for any integer $p \ge 1$ with $\gamma < p$ and $\delta > 0$,
	\[
		\frac{n^{\frac{p}{\gamma} - \delta}}{\sum_{i = 1}^n X_i^p} \plim 0;
	\]
	\item for any integer $p \ge 1$ and $\delta > 0$
	\[
		\frac{\max_{1 \le i \le n} D_i^p}{n^{\frac{p}{\gamma} + \delta}} \plim 0.
	\]
\end{enumerate}
\end{proposition}

Finally we have following lemma, where we write $f(t) \sim g(t)$ as $t \to \infty$ to denote that $\lim_{t \to \infty} f(t)/g(t) = 1$ and recall that $(x \wedge y)$ denotes the minimum of $x$ and $y$.

\begin{lemma}[\hspace{-0.5ex}{\cite[Lemma 2.6]{Hoorn2016}}]\label{lem:karamata}
Let $X$ be a non-negative regularly-varying random variable with exponent $1 < \gamma < 2$ and slowly-varying function $\mathcal{L}$.
Then,
\[
	\Exp{X \left(1 \wedge \frac{X}{t}\right)} \sim \frac{\gamma}{3\gamma - 2 - gamma^2} 
	\mathcal{L}(t) t^{1 - \gamma}, \quad \text{as } t \to \infty.
\]
\end{lemma}


\subsection{Heuristics of the proof for Pearson in the \texttt{ECM}}\label{ssec:pearson_ecm_heuristics}

Let ${\bf D}_n$ be sampled from $\mathscr{D}$ with $1< \gamma < 2$, consider $G_n = \texttt{CM}({\bf D}_n)$ and let us write $r_n = r^+_n - r^-_n$, where $r^\pm_n$ are positive functions given by
\begin{align}
	r^+_n = \frac{\sum_{i,j = 1}^n X_{ij} D_i D_j}{\sum_{i = 1}^n D_i^3 - 
		\frac{1}{L_n} \left(\sum_{i = 1}^n D_i^2\right)^2}, \quad
	r^-_n = \frac{\frac{1}{L_n} \left(\sum_{i = 1}^n D_i^2\right)^2}{\sum_{i = 1}^n D_i^3 - 
		\frac{1}{L_n} \left(\sum_{i = 1}^n D_i^2\right)^2}. \label{eq:pearson_negative_part}
\end{align}

First note that by the Stable Law CLT, see for instance \cite{whitt2006}, there exist two slowly-varying functions $\mathcal{L}_0$ and $\mathcal{L}_0^\prime$ such that
\begin{align}
	\frac{\sum_{i = 1}^n D_i^2}{\mathcal{L}_0(n) n^{\frac{2}{\gamma}}} \dlim \mathcal{S}_{\gamma/2} 
	\quad \text{and} \quad
	\frac{\sum_{i = 1}^n D_i^3}{\mathcal{L}_0^\prime(n) n^{\frac{3}{\gamma}}} \dlim \mathcal{S}_{\gamma/3},
	\quad \text{as } n \to \infty. \label{eq:degmoments}
\end{align}
Applying this and using that $L_n \approx \mu n$,
\begin{equation}\label{eq:pearson_min_clt_idea}
	\mathcal{L}_1(n) \mu n^{1 - \frac{1}{\gamma}} r^-_n
	\approx \frac{\mathcal{L}_0^\prime(n) n^{\frac{3}{\gamma}}}{\left((\mathcal{L}_0(n) n^{\frac{2}{\gamma}}\right)^2}
	\frac{\left(\sum_{i = 1}^n D_i^2\right)^2}{\sum_{i = 1}^n D_i^3} \dlim
	\frac{\mathcal{S}_{\gamma/2}^2}{\mathcal{S}_{\gamma/3}},
\end{equation}
with $\mathcal{L}_1(n) = \mathcal{L}_0^\prime(n)/(\mathcal{L}_0(n))^2$. 
Note that $r^-_n$ scales roughly as $n^{\frac{1}{\gamma}-1}$ and thus tends to zero. This extends the results in the literature that $r_n$ has a non-negative limit~\cite{Hofstad2014}.

Next, we need to show that this result also holds when we move to the erased model, i.e. when all degrees $D_i$ are replaced by $\widehat{D}_i$. To understand how this works, consider the sum of the squares of the degrees in the erased model $\sum_{i = 1}^n \widehat{D}_i^2$. Recall that $Y_i$ is the number of removed stubs of node $i$. Then we have
\[
	\sum_{i = 1}^n \widehat{D}_i^2 = \sum_{i = 1}^n (D_i - Y_i)^2 = \sum_{i = 1}^n D_i^2 
	+ \sum_{i = 1}^n\left( Y_i^2 - 2D_iY_i\right),
\]
and hence
\begin{equation}
\label{eq:sum_Der_sq-bound}
	\left|\sum_{i = 1}^n \widehat{D}_i^2 - \sum_{i = 1}^n D_i^2\right| \le 2 \sum_{i = 1}^n Y_i D_i.
\end{equation}
Therefore we only need to show that the error vanishes when we divide by $n^{\frac{2}{\gamma}}$. For this we can
use Corollary \ref{cor:scaling_degrees_erased_edges} to get, heuristically,
\begin{equation}\label{eq:Der_D_scaled}
	n^{-\frac{2}{\gamma}}\left|\sum_{i = 1}^n \widehat{D}_i^2 - \sum_{i = 1}^n D_i^2\right| \le
	2n^{-\frac{2}{\gamma}} \sum_{i = 1}^n Y_i D_i 
	\approx n^{-\frac{1}{\gamma} + 2 - \gamma} \to 0.
\end{equation}

These results will be used to prove that when $\widehat{G}_n = \texttt{ECM}({\bf D}_n)$,
\[
	n^{1 - \frac{1}{\gamma}} \left|r^{-}_n - \widehat{r}^{\,-}_n\right| \plim 0,
\]
so that by \eqref{eq:pearson_min_clt_idea}
\[
	\mathcal{L}_1(n) \mu n^{1 - \frac{1}{\gamma}} \widehat{r}^{\,-}_n \dlim
	\frac{\mathcal{S}_{\gamma/2}^2}{\mathcal{S}_{\gamma/3}}, \quad \text{as } n \to \infty.
\]
The final ingredient is Proposition \ref{prop:pearson_ecm_positive_part}, where we show that for some $\delta > 0$, as $n \to \infty$,
\begin{equation}\label{eq:convergence_r_hat_+}
	n^{1 - \frac{1}{\gamma} + \delta} \widehat{r}^{\,+}_n \plim 0 \quad.
\end{equation}
The result then follows, since for $n$ large enough $\mathcal{L}_1(n) \le C n^{\delta}$ and hence
\[
	\mu \mathcal{L}_1(n) n^{1 - \frac{1}{\gamma}} \widehat{r}_n 
	= \mu \mathcal{L}_1(n) n^{1 - \frac{1}{\gamma}} \widehat{r}^{\,+}_n 
	- \mu \mathcal{L}_1(n) n^{1 - \frac{1}{\gamma}}\widehat{r}^{\,-}_n \dlim -\frac{\mathcal{S}^2_{\gamma/2}}{\mathcal{S}_{\gamma/3}}.
\]
To establish \eqref{eq:convergence_r_hat_+}, let $\widehat{X}_{ij}$ denote the number of edges between $i$ and $j$ in the erased graph and note that since we remove self-loops $\widehat{X}_{ii} = 0$, while in the other cases $\widehat{X}_{ij} = \ind{X_{ij} > 0}$. We consider the nominator of $\widehat{r}^{\,+}_n$
\[
	\sum_{1 \le i < j \le n} \widehat{X}_{ij}D_i D_j,
\]
and will show that, as $n \to \infty$,
\[
	n^{1 - \frac{4}{\gamma} + \delta}\sum_{1 \le i < j \le n} \widehat{X}_{ij}D_i D_j \plim 0,
\]
by approximating $\Expn{\widehat{X}_{ij}}$ by $1 - e^{-D_iD_j/L_n}$, see Lemma \ref{prop:ecm_joint_moments}. Since 
the denominator in $\widehat{r}^{\,+}_n$ scales as $n^{3/\gamma}$ we get that $n^{1 - 1/\gamma + \delta} \, \widehat{r}^{\,+}_n \plim 0$.


\subsection{Intuition behind Conjecture~\ref{conj-Pearson-CM}}\label{sec:int_conj}
We note that, by \eqref{eq:pearson} and since $\Big(\sum_{i\in[n]} D_i^2\Big)^2=\op\Big(L_n \sum_{i\in[n]} D_i^3\Big),$
\eqn{
	r_n=\frac{\sum_{i,j\in[n]} D_iD_j \Big[X_{ij}-\frac{D_iD_j}{L_n}\Big]}{\sum_{i\in[n]} D_i^3}(1+\op(1)).
}

We rewrite
\eqn{
	r_n=\frac{\sum_{i,j\in[n]} D_iD_j \Big[X_{ij}-\frac{D_iD_j}{L_n-1}\Big]}{\sum_{i\in[n]} D_i^3}(1+\op(1))+\frac{\sum_{i,j,\in[n]} \frac{D_i^2D_j^2}{L_n(L_n-1)}}{\sum_{i\in[n]} D_i^3}(1+\op(1)).
}
The second term is
\eqn{
	\bigOp{n^{4\gamma-2-3\gamma}}=\bigOp{n^{\gamma-2}}=\op(n^{-1/2}),
}
since $\gamma\in(\tfrac{1}{2},1)$, and can thus be ignored. We are thus left to study the first term. 

Since $\expec_n[X_{ij}]=D_iD_j/(L_n-1)$, this term is centered. Further, the probability that any half-edge incident to vertex $i$ is connected to vertex $j$ equals $D_j/(L_n-1)$. These indicators are weakly dependent, so we assume that we can replace the conditional law of $X_{ij}$ given the degrees by a binomial random variable with $D_i$ experiments and success probability $D_j/(L_n-1)$. We will also assume that these random variables are asymptotically {\it independent}. These are the two main assumptions made in this heuristic explanation.

Since a binomial random variable is close to a normal when the number of experiments tends to infinity, these assumptions then suggest that 
\eqn{
	r_n\approx \mathcal{N}(0,\sigma_n^2),
}
where, with ${\rm Var}_n$ denoting the conditional variance given the degrees,
\eqn{
	\sigma_n^2=\frac{\sum_{i,j\in[n]} D_i^2D_j^2 {\rm Var}_n(X_{ij})}{\Big(\sum_{i\in[n]} D_i^3\Big)^2}.
}
Further, again using that $X_{ij}$ is close to a binomial, ${\rm Var}_{n}(X_{ij})\approx D_i (D_j/(L_n-1))(1-D_j/(L_n-1))\approx D_iD_j/L_n$. This suggests that
\eqn{
	\sigma_n^2\approx \frac{\sum_{i<j\in[n]} D_i^3D_j^3/L_n}{\Big(\sum_{i\in[n]} D_i^3\Big)^2}=\frac{1}{L_n},
}
which supports the conjecture in \eqref{CLT-Pearson-CM} but now with $\sigma^2=1/\mu$.  

It turns out that the above analysis is not quite correct, as $X_{ij}=X_{ji}$ when $i\leq j$, which means that these terms are highly correlated. Since terms with $i<j$ also appear several times, whereas $i=j$ does not, this turns out to change the variance formula slightly, as we discuss in the proof of Lemma~\ref{lem-Pearon-CM-Var} in Section~\ref{sec:lempearsoncm}.

\subsection{Proofs for clustering in \texttt{CM} and \texttt{ECM}}\label{ssec:clustering_heuristics}

The general idea behind the proof for both Theorems \ref{thm:CCM} and \ref{thm:CECM} is that, conditioned on the degrees, the clustering coefficients are concentrated around their conditional mean. We then proceed by analyzing this term using stable laws for regularly-varying random variables to obtain the results. 

\subsubsection{Configuration model}

To construct a triangle, six different half-edges at three distinct vertices need to be paired into a triangle. For a vertex with degree $D_i$, there are $D_i(D_i-1)/2$ ways to choose two half-edges incident to it. The probability that any two half-edges are paired in the configuration model can be approximated by $1/L_n$. Thus, the probability that a given set of six half-edges forms a triangle can be approximated by $1/L_n^3$. We then investigate $\mathcal{I}$, the set of all sets of six half-edges that could possibly form a triangle together. The expected number of triangles can then be approximated by $|\mathcal{I}|/L_n^3$. By computing the size of the set $\mathcal{I}$, we obtain that the conditional expectation for the clustering coefficient can be written as 
\[
	\Expn{C_n} \approx \frac{|\mathcal{I}|}{L_n^3\sum_{i \in[n]}D_i(D_i-1)}\approx \frac{1}{L_n^3}\left(\left(\sum_{i = 1}^n D_i^2\right)^3 - 3\sum_{i = 1}^n D_i^4+2\frac{\sum_{i=1}^nD_i^6}{\sum_{i=1}^nD_i^2}\right).
\]
The full details can be found in Section~\ref{ssec:clustering_cm}. Here the first term describes the expected number of times six half-edges are paired into a triangle. The last two terms exclude triangles including either multi-edges or self-loops. Then by the Stable-Law CLT (\cite{whitt2006}) we have that there exists a slowly-varying function $\mathcal{L}_2$ such that
\[
	\frac{\sum_{i = 1}^n D_i^2}{\mathcal{L}_2(n)n^{\frac{2}{\gamma}}} \dlim \mathcal{S}_{\gamma/2}, \quad 
	\frac{\sum_{i = 1}^n D_i^4}{\mathcal{L}_2(n)^2n^{\frac{4}{\gamma}}} \dlim \mathcal{S}_{\gamma/4}\quad
	\text{and}\quad \frac{\sum_{i=1}^nD_i^6}{\mathcal{L}_2(n)^6n^{\frac{2}{\gamma}}\sum_{i=1}^{n}D_i^2}\dlim \frac{\mathcal{S}_{\gamma/6}}{\mathcal{S}_{\gamma/2}}.
\]
Hence, using that $L_n \approx \mu n$ we obtain that
\[
	\frac{\Expn{C_n}}{\mathcal{L}_2(n) n^{\frac{4}{\gamma} - 3}} \dlim \frac{1}{\mu^3} \left(\mathcal{S}_{\gamma/2}^2
	- 3\mathcal{S}_{\gamma/4}+\frac{2\mathcal{S}_{\gamma/6}}{\mathcal{S}_{\gamma/2}}\right),
\]
where $\mathcal{S}_{\gamma/2}$, $\mathcal{S}_{\gamma/4}$ and $\mathcal{S}_{\gamma/6}$ are given by \eqref{eq:def_S_gamma}. To complete the proof we establish a concentration result for $C_n$, conditioned on the degrees. To be more precise, we employ a careful counting argument, following the approach in the proof of \cite[Proposition 7.13]{VanDerHofstad2016a} to show (see Lemma \ref{lem:exvar}) that there exists a $\delta > 0$ such that
\[
	\frac{n^\delta \textrm{Var}_n(C_n)}{n^{\frac{8}{\gamma} - 6}} \plim 0,
\] 
where $\textrm{Var}_n$ denotes the conditional variance given the degrees. Then, it follows from Chebyshev's inequality, conditioned on the degrees, that
\[
	\frac{\left|C_n - \Expn{C_n}\right|}{\mathcal{L}_2(n) n^{\frac{4}{\gamma} - 3}} \plim 0,
\]
and we conclude that
\[
	\frac{C_n}{\mathcal{L}_2(n) n^{\frac{4}{\gamma} - 3}} \dlim \frac{1}{\mu^3} \left(\mathcal{S}_{\gamma/2}^2
		- 3\mathcal{S}_{\gamma/4}+2\mathcal{S}_{\gamma/6}/\mathcal{S}_{\gamma/2}\right).
\]

\subsubsection{Erased configuration model}

The difficulty for clustering in \texttt{ECM}, compared to \texttt{CM}, is in showing that $\widehat{C}_n$ behaves as its conditional expectation, as well as establishing its scaling. To compute this we first fix an $\varepsilon > 0$ and show in Lemma \ref{lem:sqrt} that the main contribution is given by triples of nodes with degrees $\varepsilon \sqrt{n} \le D \le \frac{\sqrt{n}}{\varepsilon}$, i.e.
\begin{align*}
	\sum_{1 \le i < j < k \le n} \widehat{X}_{ij}\widehat{X}_{jk}\widehat{X}_{ik} 
	&= \sum_{1 \le i < j < k \le n} \widehat{X}_{ij}\widehat{X}_{jk}\widehat{X}_{ik}
	\ind{\varepsilon \sqrt{n} \le D_i, D_j, D_k \le \frac{\sqrt{n}}{\varepsilon}}\\ 
	&\hspace{10pt}+ O(\mathcal{L}(\sqrt{\mu n})^3n^{\frac{3}{2}(2-\gamma)}) \mathcal{E}_1(\varepsilon),
\end{align*}
where $\mathcal{E}_1(\varepsilon)$ is an error function, independent of $n$, with $\lim_{\varepsilon \to 0} \mathscr{E}_1(\varepsilon) = 0$. 
Then we use that approximately $\Expn{\widehat{X}_{ij}} \approx 1 - \me^{-{D_i D_j}/{L_n}}$ to show
that
\[
\Expn{\widehat{C}_n}\approx \frac{6\sum_{1\leq i< j< k\leq n} g_{n, \varepsilon}(D_i, D_j, D_k) +O(\mathcal{L}(\sqrt{\mu n})^3n^{\frac{3}{2}(2-\gamma)}) \mathcal{E}_1(\varepsilon)}
{\sum_{i = 1}^n \widehat{D}_i(\widehat{D}_i - 1)},
\]
where 
\[
g_{n, \varepsilon}(x,y,z) = \left(1 - \me^{-\frac{xy}{L_n}}\right)\left(1 - \me^{-\frac{yz}{L_n}}\right)\left(1 - \me^{-\frac{zx}{L_n}}\right)
\ind{\varepsilon \sqrt{n} \le x, y, z \le \frac{\sqrt{n}}{\varepsilon}}.
\]
Here $\mathbb{E}_n$ denotes again expectation conditioned on $\boldsymbol{D}_n$, hence conditional on the sampled  degrees of the underlying CM. The precise statement can be found in in Lemma~\ref{lem:extriang}.
After that, we show in Lemma~\ref{lem:exvarECM} that $\widehat{C}_n$ concentrates around its expectation conditioned on the sampled degrees, so that conditioned on the sampled degree sequence, we can approximate $\widehat{C}_n\approx \Expn{\widehat{C}_n}$.
We then replace $L_n$ by $\mu n$ in Lemma \ref{lem:g_n_epsilon_to_f_n_epsilon}, so that, conditioned on the degree sequence,
\[
\widehat{C}_n
	\approx \frac{6\sum_{1\leq i< j< k\leq n} f_{n, \varepsilon}(D_i, D_j, D_k)+O(\mathcal{L}(\sqrt{\mu n})^3n^{\frac{3}{2}(2-\gamma)})  \mathcal{E}_1(\varepsilon)}{\sum_{i = 1}^n \widehat{D}_i(\widehat{D}_i - 1)},
\]
with 
\[
	f_{n, \varepsilon}(x,y,z) = \left(1 - \me^{-\frac{xy}{\mu n}}\right)\left(1 - \me^{-\frac{yz}{\mu n}}\right)\left(1 - \me^{-\frac{zx}{\mu n}}\right) \ind{\varepsilon \sqrt{n} \le x, y, z \le \frac{\sqrt{n}}{\varepsilon}}.
\]
We then take the random degrees into account, by showing that
\[
	\frac{1}{\mathcal{L}(\sqrt{\mu n})^3n^{\frac{3}{2}(2-\gamma)}}\sum_{1\leq i< j< k\leq n}f_{n, \varepsilon}(D_1, D_2, D_3)
	\plim \frac{1}{6}\mu^{-\frac{3\gamma}{2}}A_\gamma(\varepsilon) + \mathscr{E}_2(\varepsilon)
\]
where
\[
	A_\gamma(\varepsilon) = \int_\varepsilon^{1/\varepsilon} \int_\varepsilon^{1/\varepsilon}  \int_\varepsilon^{1/\varepsilon}  \frac{1}{(xyz)^{\gamma + 1}} (1 -\me^{-xy}) (1 - \me^{-xy}) (1 -\me^{-xy}) \, \dd x \dd y \dd z
\]
and $\mathcal{E}_2(\varepsilon)$ is a deterministic error function, with $\lim_{\varepsilon \to 0} \mathscr{E}_2(\varepsilon) = 0$. 
Finally, we again replace the $\widehat{D}_i$ with $D_i$ and use the Stable Law CLT to obtain a slowly-varying function $\mathcal{L}_3$, such that
\[
	\frac{\mathcal{L}_3(n) n^{\frac{2}{\gamma}}}{\sum_{i = 1}^n \widehat{D}_i(\widehat{D}_i - 1)} \dlim \frac{1}{\mathcal{S}_{\gamma/2}}.
\]
Combining all these results implies that, for any $\varepsilon > 0$,
\[
	\frac{\mathcal{L}_3(n) \widehat{C}_n}{\mathcal{L}(\sqrt{\mu n})^3 n^{-\frac{3 \gamma}{2} + 3 - \frac{2}{\gamma}}}
	\dlim \mu^{-\frac{3\gamma}{2}}\frac{A_\gamma(\varepsilon)}{\mathcal{S}_{\gamma/2}^2} 
	+ \frac{\mathscr{E}_1(\varepsilon)+\mathcal{E}_2(\varepsilon)}{\mathcal{S}_{\gamma/2}^2},
\]
from which the result follows by taking $\varepsilon \downarrow0$.

\section{Pearson's correlation coefficient}\label{sec:pearson}

In this section we first give the proof of Theorem \ref{thm:clt_pearson_ecm}, where we follow the approach described in Section \ref{ssec:pearson_ecm_heuristics}. We then prove Lemma~\ref{lem-Pearon-CM-Var}, which supports Conjecture~\ref{conj-Pearson-CM} on the behavior of Pearson in the CM.

\subsection{Limit theorem for $r^-_n$}

We first prove a limit theorem for $r^-_n$, when $G_n = \texttt{CM}({\bf D}_n)$. Recall that
\[
	r^-_n = \frac{\frac{1}{L_n} \left(\sum_{i = 1}^n D_i^2\right)^2}{\sum_{i = 1}^n D_i^3 - 
			\frac{1}{L_n} \left(\sum_{i = 1}^n D_i^2\right)^2}.
\]

\begin{proposition}\label{prop:clt_pearson_negative_part}
Let ${\bf D}_n$ be sampled from $\mathscr{D}$ with $1< \gamma < 2$ and $\Exp{\mathscr{D}} = \mu$. Then, if $G_n = \emph{\texttt{CM}}({\bf D}_n)$, there exists a slowly-varying function $\mathcal{L}_0$ such that
\[
	\mu \mathcal{L}_0(n)n^{1-\frac{1}{\gamma}}r^-_n \dlim \frac{\mathcal{S}_{\gamma/2}^2}{\mathcal{S}_{\gamma/3}},
\]
as $n \to \infty$. Here $\mathcal{S}_{\gamma/2}$ and $\mathcal{S}_{\gamma/3}$ are given by \eqref{eq:def_S_gamma}.
\end{proposition}

\begin{proof}
We will first show that there exists a slowly-varying function $\mathcal{L}_0$ such that
\begin{equation}\label{eq:clt_pearson_cm_negative_main}
	\mu\mathcal{L}_0(n)n^{1 - \frac{1}{\gamma}} \frac{\left(\sum_{i = 1}^n 
	D_i^2\right)^2}{\mu n \sum_{i = 1}^n D_i^3} \dlim 
	\frac{\mathcal{S}_{\gamma/2}^2}{\mathcal{S}_{\gamma/3}},
\end{equation}
as $n \to \infty$ and $\mathcal{S}_{\gamma/2}$ and $\mathcal{S}_{\gamma/3}$ defined by \eqref{eq:def_S_gamma}. 

Let $D_{(i)}$ denote the $i$-th largest degree in ${\bf D}_n$, i.e. $D_{(1)} \ge D_{(2)} \ge \dots D_{(n)}$, and let 
$\Gamma_i$ be defined as in \eqref{eq:def_Gamma}.
Then, since $\sum_{i = 1}^n D_{(i)}^p = \sum_{i = 1}^n D_i^p$, for any $p \ge 0$, it follows from \cite[Theorem 2.33]{VanDerHofstad2016a} that for some slowly-varying functions $\mathcal{L}_2$, $\mathcal{L}_3$, $\mathcal{L}_4$ and $\mathcal{L}_6$, as $n \to \infty$,

\begin{equation}\label{eq:joint_convegrence_degree_sums}
\begin{aligned}
	&\hspace{-30pt}\left(\frac{n^{-\frac{4}{\gamma}}}{\mathcal{L}_2(n)^2} \left(\sum_{i = 1}^n D_i^2\right)^2, \, 
		\frac{n^{-\frac{3}{\gamma}}}{\mathcal{L}_3(n)} \sum_{i = 1}^n D_i^3, \,
		\frac{n^{-\frac{4}{\gamma}}}{\mathcal{L}_4(n)}\sum_{i = 1}^n D_i^4, \,
		\frac{n^{-\frac{6}{\gamma}}}{\mathcal{L}_6(n)}\sum_{i = 1}^n D_i^6\right)\\
	&\hspace{30pt}\dlim \left(\left(\sum_{i = 1}^\infty \Gamma_i^{-\frac{2}{\gamma}}\right)^2, \, 
		\sum_{i = 1}^\infty \Gamma_i^{-\frac{3}{\gamma}}, \,
		\sum_{i = 1}^\infty \Gamma_i^{-\frac{4}{\gamma}}, \,
		\sum_{i = 1}^\infty \Gamma_i^{-\frac{6}{\gamma}}\right).
\end{aligned}
\end{equation}
Here we include the fourth and sixth moment, since these will be needed later for proving Theorem~\ref{thm:CCM}.

Note that $\mathcal{L}_0(n) := \mathcal{L}_2(n)/\mathcal{L}_1(n)^2$ is slowly varying and $\Prob{\sum_{i = 1}^\infty \Gamma_i^{-t/\gamma} \le 0} = 0$ for any $t \ge 2$. Hence \eqref{eq:clt_pearson_cm_negative_main} follows from~\eqref{eq:joint_convegrence_degree_sums} and the continuous mapping theorem. Hence, to prove the main result, it is enough to show that
\[
	\mathcal{L}_0(n)n^{1 - \frac{1}{\gamma}}\left|r^-_n - \frac{\left(\sum_{i = 1}^n D_i^2\right)^2}{\mu n \sum_{i = 1}^n D_i^3}\right| \plim 0.
\]

We will prove the stronger statement, 
\begin{equation}\label{eq:clt_pearson_cm_negative_remainder}
	n^{1 - \frac{1}{\gamma} + \frac{\kappa}{2}}\left|r^-_n - \frac{\left(\sum_{i = 1}^n D_i^2\right)^2}{\mu n \sum_{i = 1}^n D_i^3}\right| \plim 0.
\end{equation}
where $\kappa = (\gamma - 1)/(\gamma + 1) > 0$ is the same as in Lemma \ref{lem:Kolmogorov_SLLN}.

Note that by Lemma \ref{lem:Kolmogorov_SLLN} we have that $\mu n/ L_n \plim 1$. Hence, by \eqref{eq:clt_pearson_cm_negative_main}, we have that for any $\delta > 0$,
\[
	n^{1 - \frac{1}{\gamma} - \delta}\frac{\left(\sum_{i = 1}^n D_i^2\right)^2}{L_n \sum_{i = 1}^n D_i^3} 
	=  n^{1 - \frac{1}{\gamma} - \delta} \frac{\left(\sum_{i = 1}^n D_i^2\right)^2}{\mu n \sum_{i = 1}^n D_i^3}\left(\frac{\mu n}{L_n}\right) \plim 0,
\]
from which we conclude that
\begin{equation}\label{eq:clt_pearson_cm_negative_fraction}
	n^{1 - \frac{1}{\gamma} - \delta} \frac{\left(\sum_{i = 1}^n D_i^2\right)^2}{L_n \sum_{i = 1}^n D_i^3
	- \left(\sum_{i = 1}^n D_i^2\right)^2} \plim 0.
\end{equation}

To show \eqref{eq:clt_pearson_cm_negative_remainder}, we write
\begin{align*}
	\hspace{-30pt}\left|\frac{\left(\sum_{i = 1}^n D_i^2\right)^2}{\mu n\sum_{i = 1}^n D_i^3} 
		- r^-_n\right|
	&= \frac{\left(\sum_{i = 1}^n D_i^2\right)^2}{L_n \sum_{i = 1}^n D_i^3}
		\left|\frac{(L_n - \mu n)\sum_{i = 1}^n D_i^3 + \mu n \left(\sum_{i = 1}^n D_i^2\right)^2}
		{\mu n \left(L_n \sum_{i = 1}^n D_i^3 - \left(\sum_{i = 1}^n D_i^2\right)^2\right)}\right|\\
	&\le \frac{\left(\sum_{i = 1}^n D_i^2\right)^2}{L_n \sum_{i = 1}^n D_i^3 
		- \left(\sum_{i = 1}^n D_i^2\right)^2}\frac{\left|\mu n - L_n\right|}{\mu n} \\
	&\hspace{10pt}+ \left(\frac{\left(\sum_{i = 1}^n D_i^2\right)^2}{L_n \sum_{i = 1}^n D_i^3 
		- \left(\sum_{i = 1}^n D_i^2\right)^2}\right)^2.
\end{align*}
For the first term we have, using \eqref{eq:clt_pearson_cm_negative_fraction} and Lemma \ref{lem:Kolmogorov_SLLN},
\begin{align*}
	&\hspace{-30pt}n^{1 - \frac{1}{\gamma} + \frac{\kappa}{2}}\frac{\left(\sum_{i = 1}^n D_i^2\right)^2}
		{L_n \sum_{i = 1}^n D_i^3 - \left(\sum_{i = 1}^n D_i^2\right)^2}
		\frac{\left|\mu n - L_n\right|}{\mu n}\\
	&= n^{1 - \frac{1}{\gamma} - \frac{\kappa}{2}}\left(\frac{\left(\sum_{i = 1}^n D_i^2\right)^2}
		{L_n \sum_{i = 1}^n D_i^3 - \left(\sum_{i = 1}^n D_i^2\right)^2}\right)
		\left(\frac{\left|\mu n - L_n\right|}{\mu n^{1 - \kappa}}\right) \plim 0.
\end{align*}
Now, let $\delta = 1 - 1/\gamma - \kappa/2 > 0$. Then, since $1 - 1/\gamma + \kappa/2 = 2 - 2/\gamma - \delta$, 
it follows from \eqref{eq:clt_pearson_cm_negative_fraction} that
\[
	n^{1 - \frac{1}{\gamma} + \frac{\kappa}{2}}\left(\frac{\left(\sum_{i = 1}^n D_i^2\right)^2}{L_n \sum_{i = 1}^n D_i^3 
			- \left(\sum_{i = 1}^n D_i^2\right)^2}\right)^2
	= \left(\frac{n^{1 - \frac{1}{\gamma} - \frac{\delta}{2}} \left(\sum_{i = 1}^n D_i^2\right)^2}{L_n \sum_{i = 1}^n D_i^3 
				- \left(\sum_{i = 1}^n D_i^2\right)^2}\right)^2 \plim 0,
\]
which finishes the proof of \eqref{eq:clt_pearson_cm_negative_remainder}.
\end{proof}

\subsection{Limit theorem for $\widehat{r}^{\,-}_n$}

We now turn to the \texttt{ECM}. Observe that for $\widehat{G}_n = \texttt{ECM}({\bf D}_n)$, (\ref{eq:sum_Der_sq-bound}) and Corollary \ref{cor:scaling_degrees_erased_edges} that, for any $\delta > 0$,
\[
	\frac{\left|\sum_{i = 1}^n D_i^2 - \sum_{i = 1}^n \widehat{D}_i^2\right|}{n^{\frac{1}{\gamma} + 2 - \gamma + \delta}}
	\le \frac{2\sum_{i = 1}^n D_i Y_i}{n^{\frac{1}{\gamma} + 2 - \gamma + \delta}} \plim 0.
\]
 Since $\frac{2}{\gamma} > \frac{1}{\gamma} + 2 - \gamma$ for all $\gamma > 1$ this result implies that for any $\delta > 0$,

\begin{equation}\label{eq:Dhatsq}
	\frac{\sum_{i = 1}^n \widehat{D}_i^2}{n^{\frac{2}{\gamma} + \delta}} \plim 0.
\end{equation}
This line of reasoning can be extended to sums of $\widehat{D}_i^p$, for any $p > 0$,
proving that the degrees $\widehat{D}_i$ in the ECM satisfy the same scaling results as those for $D_i$. In particular we have the following extension of \eqref{eq:clt_pearson_cm_negative_fraction} to the erased configuration model. Recall that $\mathscr{D}$ denotes an integer-valued random variable with a regularly-varying distribution defined by (see~\eqref{eq:distribution_degrees})
\[
	\Prob{\mathscr{D} > t} = \mathcal{L}(t) t^{-\gamma.}
\]

\begin{lemma}\label{lem:scaling_pearson_ecm_main_fraction}
Let ${\bf D}_n$ be sampled from $\mathscr{D}$ with $1 < \gamma < 3$ and $\widehat{G}_n = \emph{\texttt{ECM}}({\bf D}_n)$.
Then, for any $\delta > 0$,
\[
	n^{1 - \frac{1}{\gamma} - \delta} \frac{\left(\sum_{i = 1}^n \widehat{D}_i^2\right)^2}{L_n \sum_{i = 1}^n \widehat{D}_i^3
		- \left(\sum_{i = 1}^n \widehat{D}_i^2\right)^2} \plim 0.
\]
\end{lemma}

Now recall that $\widehat{r}^{\,-}_n$ denotes the negative part of Pearson's correlation coefficient for the erased configuration model, i.e
\[
	\widehat{r}^{\, -}_n = \frac{\frac{1}{L_n} \left(\sum_{i = 1}^n \widehat{D}_i^2\right)^2}{\sum_{i = 1}^n \widehat{D}_i^3 - \frac{1}{\widehat{L}_n} \left(\sum_{i = 1}^n \widehat{D}_i^2\right)^2}.
\]
The next proposition shows that $|r^{-}_n-\widehat{r}^{\,-}_n| = o_\pr\left(n^{\frac{1}{\gamma} - 1}\right)$.

\begin{proposition}\label{prop:pearson_min_cm_vs_ecm}
Let ${\bf D}_n$ be sampled from $\mathscr{D}$ with $1< \gamma < 2$. Let
$\widehat{G}_n = \emph{\texttt{ECM}}({\bf D}_n)$, denote by $G_n$ the graph before the removal of edges and recall that $r_n^-$ and $\widehat{r}^{\,-}_n$ denote the negative part of Pearson's correlation coefficient in $G_n$ and $\widehat{G}_n$, respectively. Then, 
\[
	n^{1 - \frac{1}{\gamma} + \frac{(\gamma-1)^2}{4 \gamma}}\left|r^{-}_n - \widehat{r}^{\,-}_n\right| \plim 0.
\]
\end{proposition} 

\begin{proof}
The proof consist of splitting the main term into separate terms, which can be expressed in terms of erased stubs or edges, and showing that each of these terms converge to zero. Throughout the proof we let
\[
	\delta = \frac{(\gamma - 1)^2}{4\gamma}.
\]

We start by splitting the main term as follows,
\begin{align}
	\left|r^{-}_n - \widehat{r}^{\,-}_n\right| 
	&\le \frac{\left|\left(\sum_{i = 1}^n D_i^2\right)^2 - \left(\sum_{i = 1}^n \widehat{D}_i^2\right)^2\right|}
		{L_n \sum_{i = 1}^n D_i^3 - \left(\sum_{i = 1}^n D_i^2\right)^2} \label{eq:pearson_min_cm_vs_ecm_D2}\\
	&\hspace{-30pt}+ \left(\sum_{i = 1}^n \widehat{D}_i^2\right)^2\left|\frac{1}{L_n \sum_{i = 1}^n D_i^3 - \left(\sum_{i = 1}^n D_i^2\right)^2} - \frac{1}{\widehat{L}_n \sum_{i = 1}^n \widehat{D}_i^3 - \left(\sum_{i = 1}^n \widehat{D}_i^2\right)^2}\right|. \label{eq:pearson_min_cm_vs_ecm_remainder}
\end{align}

For \eqref{eq:pearson_min_cm_vs_ecm_D2} we use that
\begin{align*}
	\left|\left(\sum_{i = 1}^n D_i^2\right)^2 - \left(\sum_{i = 1}^n \widehat{D}_i^2\right)^2\right|
	&= \left(\sum_{i = 1}^n D_i^2 - \widehat{D}_i^2\right)
		\left(\sum_{i = 1}^n D_i^2 + \widehat{D}_i^2\right)\\
	&\le \left(\sum_{i = 1}^n 2 D_i Y_i\right)\left(\sum_{i = 1}^n 2D_i^2 + Y_i^2\right)
		\le 6 \sum_{i = 1}^n D_i^2 \sum_{j = 1}^n Y_j D_j,
\end{align*}
to obtain
\begin{align*}
	\frac{\left|\left(\sum_{i = 1}^n D_i^2\right)^2 - \left(\sum_{i = 1}^n \widehat{D}_i^2\right)^2\right|}
		{L_n \sum_{i = 1}^n D_i^3 - \left(\sum_{i = 1}^n D_i^2\right)^2}
	&\le \frac{6 \sum_{i = 1}^n D_i^2 \sum_{i = 1}^n Y_i D_i}
		{L_n \sum_{i = 1}^n D_i^3 - \left(\sum_{i = 1}^n D_i^2\right)^2} \\
	&= \frac{6 \sum_{i = 1}^n Y_i D_i}{\sum_{i = 1}^n D_i^2} \frac{\left(\sum_{i = 1}^n D_i^2\right)^2}
		{L_n \sum_{i = 1}^n D_i^3 - \left(\sum_{i = 1}^n D_i^2\right)^2}. \mumberthis \label{eq:pearson_min_cm_vs_ecm_D2_bound}
\end{align*}
Now observe that
\begin{equation}\label{eq:pearson_min_ecm_vs_cm_4delta}
	2 - \frac{1}{\gamma} - \gamma = -(\gamma - 1)^2/\gamma = -4\delta,
\end{equation}
and hence
\begin{align*}
	1 - \frac{1}{\gamma} + \delta = \gamma - 1 -3\delta
	= -\left(\frac{1}{\gamma} + 2 -\gamma + \delta\right)
	+ \left(\frac{2}{\gamma} - \frac{\delta}{2}\right) +\left(1 - \frac{1}{\gamma} - \frac{\delta}{2}\right)
	- \delta,
\end{align*}
with all the three terms inside the brackets positive. 
Therefore, it follows from \eqref{eq:pearson_min_cm_vs_ecm_D2_bound}, together with Corollary \ref{cor:scaling_degrees_erased_edges}, Proposition \ref{prop:scaling_sums_powers_regularly_varying} and \eqref{eq:clt_pearson_cm_negative_fraction} that
\begin{align*}
	&n^{1 -\frac{1}{\gamma} + \delta}\frac{\left|\left(\sum_{i = 1}^n D_i^2\right)^2 - \left(\sum_{i = 1}^n 	
		\widehat{D}_i^2\right)^2\right|}{L_n \sum_{i = 1}^n D_i^3 - \left(\sum_{i = 1}^n D_i^2\right)^2}\\
	&\le n^{-\delta}\left(\frac{6\sum_{i = 1}^n D_i Y_i}{n^{\frac{1}{\gamma} + 2 - \gamma + \delta}}\right)
		\left(\frac{n^{\frac{2}{\gamma} - \frac{\delta}{2}}}{\sum_{i = 1}^n D_i^2}\right)
		\left( \frac{n^{1 - \frac{1}{\gamma} - \frac{\delta}{2}}\left(\sum_{i = 1}^n D_i^2\right)^2}
		{L_n \sum_{i = 1}^n D_i^3 - \left(\sum_{i = 1}^n D_i^2\right)^2}\right) \plim 0.
		\mumberthis \label{eq:pearson_min_cm_vs_ecm_D2_convergence}
\end{align*}

The second term, (\ref{eq:pearson_min_cm_vs_ecm_remainder}), requires more work. Let us first write
\begin{align*}
	&\left(\sum_{i = 1}^n \widehat{D}_i^2\right)^2
		\left|\frac{1}{L_n \sum_{i = 1}^n D_i^3 - \left(\sum_{i = 1}^n D_i^2\right)^2} 
		- \frac{1}{\widehat{L}_n \sum_{i = 1}^n \widehat{D}_i^3 - \left(\sum_{i = 1}^n \widehat{D}_i^2\right)^2}\right|\\
	&:= \frac{\left(\sum_{i = 1}^n \widehat{D}_i^2\right)^2}{\widehat{L}_n \sum_{i = 1}^n \widehat{D}_i^3 
			- \left(\sum_{i = 1}^n \widehat{D}_i^2\right)^2} \left(I_n^{(1)} + I_n^{(2)} + I_n^{(3)}\right),
\end{align*}
with 
\begin{align*}
	I_n^{(1)} &:=  \frac{\left(\sum_{i = 1}^n D_i^2\right)^2 - \left(\sum_{i = 1}^n \widehat{D}_i^2\right)^2}
		{L_n \sum_{i = 1}^n D_i^3 - \left(\sum_{i = 1}^n D_i^2\right)^2}\\
	I_n^{(2)} &:= \frac{Z_n \sum_{i = 1}^n \widehat{D}_i^3}{L_n \sum_{i = 1}^n D_i^3 - \left(\sum_{i = 1}^n D_i^2\right)^2}\\
	I_n^{(3)} &:= \frac{L_n \left|\sum_{i = 1}^n D_i^3 - \sum_{i = 1}^n \widehat{D}_i^3\right|}
				{L_n \sum_{i = 1}^n D_i^3 - \left(\sum_{i = 1}^n D_i^2\right)^2},
\end{align*}
and we recall that $Z_n = L_n - \widehat{L}_n$ denotes the total number of removed edges.
Note that
\[
	n^{1 - \frac{1}{\gamma} - \frac{\delta}{2}}\frac{\left(\sum_{i = 1}^n \widehat{D}_i^2\right)^2}{\widehat{L}_n \sum_{i = 1}^n \widehat{D}_i^3 - \left(\sum_{i = 1}^n \widehat{D}_i^2\right)^2} \plim 0.
\]
Therefore, in order to complete the proof, it suffices to show that
\[
	n^{\frac{3\delta}{2}} I_n^{(t)} \plim 0, \quad \text{for } t = 1,2, 3. 
\]
For $t=1$ this follows from \eqref{eq:pearson_min_cm_vs_ecm_D2_convergence}, since,
\[
	\frac{1}{\gamma} - 1 + \frac{\delta}{2} = \frac{\gamma^2 - 10\gamma + 9}{8 \gamma} = \frac{(\gamma - 1)(\gamma - 9)}{8\gamma} < 0,
\]
and hence
\[
	\frac{3\delta}{2} < 1 - \frac{1}{\gamma} + \delta.
\]

For $t = 2$ we use that $\widehat{D}_i \le D_i$ to obtain,
\begin{align*}
	I_n^{(2)}&\le \frac{E_n \sum_{i = 1}^n {D}_i^3}{{L}_n \sum_{i = 1}^n {D}_i^3 
		- \left(\sum_{i = 1}^n D_i^2\right)^2}
	= \frac{E_n}{L_n}\left(1 - \frac{\left(\sum_{i = 1}^n D_i^2\right)^2}{L_n \sum_{i = 1}^n {D}_i^3}\right)^{-1}
\end{align*}
By Proposition \ref{prop:clt_pearson_negative_part} it follows that
\begin{equation}\label{eq:convergence_pearson_negaitve_fraction}
	\left(1 - \frac{\left(\sum_{i = 1}^n D_i^2\right)^2}{L_n \sum_{i = 1}^n {D}_i^3}\right)^{-1} \plim 1.
\end{equation}
In addition we have that $\varepsilon := \gamma - 1 - 3\delta/2 > 0$ and hence, by Theorem \ref{thm:scaling_erased_edges} and the strong law of large numbers,
\[
	n^{\frac{3\delta}{2}}I_n^{(2)} \le \left(n^{\gamma - 2 - \varepsilon} \frac{E_n}{\mu}\right) \frac{\mu n}{L_n}
	\left(1 - \frac{\left(\sum_{i = 1}^n D_i^2\right)^2}{L_n \sum_{i = 1}^n {D}_i^3}\right)^{-1} \plim 0.
\]
Finally, for $I_n^{(3)}$ we first compute
\[
	\left|\sum_{i = 1}^n D_i^3 - \sum_{i = 1}^n \widehat{D}_i^3\right| 
	= \sum_{i = 1}^n Y_i^3 + 3D_i^2 Y_i - 3D_iY_i^2
	\le 4 \sum_{i = 1}^n Y_i D_i^2,
\]
and hence,
\begin{align*}
	I_n^{(3)} \le \frac{4 \sum_{i = 1}^n Y_i D_i^2}{\sum_{i = 1} D_i^3} \left(1 - \frac{\left(\sum_{i = 1}^n D_i^2\right)^2}{L_n \sum_{i = 1}^n {D}_i^3}\right)^{-1}. 
\end{align*}
By \eqref{eq:convergence_pearson_negaitve_fraction} the last term converges in probability to one. Finally, by \eqref{eq:pearson_min_ecm_vs_cm_4delta},
\[
	\frac{3\delta}{2} \le 2\delta = 4\delta - 2\delta = \left(\gamma - 2 - \frac{2}{\gamma} - \delta\right) 
	+ \left(\frac{3}{\gamma} - \delta\right)
\]
and hence, by Corollary \ref{cor:scaling_degrees_erased_edges} and Proposition \ref{prop:scaling_sums_powers_regularly_varying},
\[
	n^{\frac{3\delta}{2}}I_n^{(3)} \le \left(4n^{\gamma - 2 - \frac{2}{\gamma} - \delta} \sum_{i = 1}^n Y_i D_i^2\right)
	\left(\frac{n^{\frac{3}{\gamma} - \delta}}{\sum_{i = 1} D_i^3}\right) \plim 0,
\]
which finishes the proof.
\end{proof}

\subsection{Convergence of $\widehat{r}^{\,+}_n$}

The next step towards the proof of Theorem \ref{thm:clt_pearson_ecm} is to show that, for some $\delta > 0$, 
\[
	n^{1 - \frac{1}{\gamma} + \delta} \widehat{r}^{\, +}_n \plim 0,
\]
where
\[
	\widehat{r}^{\, +}_n = \frac{\sum_{i,j = 1}^n \widehat{X}_{ij} \widehat{D}_i \widehat{D}_j}{\sum_{i = 1}^n \widehat{D}_i^3 - 
			\frac{1}{\widehat{L}_n} \left(\sum_{i = 1}^n \widehat{D}_i^2\right)^2},
\]
denotes the positive part of Pearson's correlation coefficient in the erased configuration model. Here $\widehat{X}_{ij} = \ind{X_{ij} > 0}$ denotes the event that $i$ and $j$ are connected by at least one edge in the configuration model graph $G_n$.
The main ingredient of this result is the following lemma, which gives an approximation for 
$\sum_{1 \le i < j \le n} \Probn{X_{ij} = 0}D_i D_j$: 

\begin{lemma}\label{prop:ecm_joint_moments}
Let ${\bf D}_n$ be sampled from $\mathscr{D}$ with $1< \gamma < 2$ and $\Exp{\mathscr{D}} = \mu$. Consider graphs 
$G_n = \emph{\texttt{CM}}({\bf D}_n)$ and define
\[
	M_n = \sum_{1 \le i < j \le n}\left|\Probn{X_{ij} = 0} 
	- \exp\left\{-\frac{D_i D_j}{L_n}\right\}\right|D_i D_j
\]
Then, for any $K > 0$ and $0 < \delta < \left(\frac{2 - \gamma}{\gamma} \wedge \frac{\gamma - 1}{\gamma}\right)$,
\[
	n^{1 - \frac{4}{\gamma} + \delta} Z_n \plim 0.
\]
\end{lemma}

In our proofs $M_n$ will be divided by
\[
	\sum_{i = 1}^n D_i^3 - \frac{1}{L_n} \left(\sum_{i = 1}^n D_i^2\right)^2,
\]
which is of the order $n^{3/\gamma}$. Hence the final expression will be of the order $n^{\frac{1}{\gamma} - 1 - \delta} =
o(n^{\frac{1}{\gamma} - 1})$, which is enough to prove the final result.

To prove Lemma \ref{prop:ecm_joint_moments}, we will use the following technical result:

\begin{lemma}[{\cite[Lemma 6.7]{Hoorn2015b}}] \label{lem:taylor_connection_prob}
For any non-negative $x, x_0 > 0$, $y_i, z_i \geq 0$ with $z_i < x$ for all $i$, and any 
$m \geq 1$, we have 
\[
	- \frac{x_0}{x^2} (x_0-x)^+ - \frac{x_0}{2} \max_{1 \leq i \leq m} \frac{z_i}{(x-z_i)^2} \leq 
	\prod_{i=1}^{m} \left( 1- \frac{z_i}{x} \right)^{y_i} - \exp\left\{ - \frac{1}{x_0} 
	\sum_{i=1}^{m} y_i z_i \right\} \leq \frac{|x-x_0|}{(x \wedge x_0)}.
\]
\end{lemma}

\begin{proof}[Proof of Lemma \ref{prop:ecm_joint_moments}]
We will first consider the term $\left|\Probn{X_{ij} = 0} - \exp\left\{-\frac{D_i D_j}{L_n}
\right\}\right|$. It follows from computations done in \cite{hofstad2005} that
\begin{equation}\label{eq:ecm_pearson_diff_Xij_product}
	0 \le \Probn{X_{ij} = 0} - \prod_{t = 0}^{D_i - 1} \left(1 - \frac{D_j}{L_n -2t -1}\right) 
	\le \frac{D_i^2 D_j}{(L_n - 2D_i)^2}.
\end{equation}
For the product term in \eqref{eq:ecm_pearson_diff_Xij_product} we have the following bounds
\[
	\left(1 - \frac{D_j}{L_n -2D_i + 1}\right)^{D_i} 
	\le \prod_{t = 0}^{D_i - 1} \left(1 - \frac{D_j}{L_n -2t -1}\right) \le
	\left(1 - \frac{D_j}{L_n}\right)^{D_i} 
\]
and therefore, using Lemma \ref{lem:taylor_connection_prob} with $m = 1$, we can bound the 
difference between $\Probn{X_{ij} = 0}$ and $\exp\left\{-D_i D_j/L_n\right\}$. For the lower 
bound we take $x = L_n$, $x_0 = L_n + 1 - 2D_i$, $y = D_i$ and $z = D_j$ to get 
\begin{equation}\label{eq:ecm_pearson_diff_X_ij_exp_lower}
	-\frac{L_n(2D_i - 1)}{(L_n - 2D_i + 1)^2} - \frac{D_j}{L_n + 1 - 2D_1 - D_j}
	\le \Probn{X_{ij} = 0} - \exp\left\{-\frac{D_i D_j}{L_n}\right\}
\end{equation}
while changing $x_0$ to $L_n$ yields
\begin{equation}\label{eq:ecm_pearson_diff_X_ij_exp_upper}
	\Probn{X_{ij} = 0} - \exp\left\{-\frac{D_i D_j}{L_n}\right\} 
	\le \frac{D_i^2 D_j}{(L_n - 2D_i)^2}.
\end{equation}
Combining \eqref{eq:ecm_pearson_diff_X_ij_exp_lower} and 
\ref{eq:ecm_pearson_diff_X_ij_exp_upper} gives
\begin{align*}
	&\sum_{1 \le i < j \le n}\left|\Probn{X_{ij} = 0} 
	- \exp\left\{-\frac{D_i D_j}{L_n}\right\}\right|D_i D_j \\
	&\le \sum_{1 \le i < j \le n} \frac{2L_nD_i^2 D_j}{(L_n - 2D_i + 1)^2} 
		+ \sum_{1 \le i < j \le n} \frac{D_i D_j^2}{L_n + 1 - 2D_1 - D_j}
		+ \sum_{1 \le i < j \le n} \frac{D_i^3 D_j^2}{(L_n - 2D_i)^2}\\
	&:= I_n^{(1)} + I_n^{(2)} + I_n^{(3)}.
\end{align*}
We will now show that 
\begin{equation}\label{eq:ecm_joint_moments_main}
	n^{1 - \frac{4}{\gamma} + \delta} I_n^{(t)} \plim 0, \quad \text{for } t = 1, 2, 3,
\end{equation}
which proves the result.

For the remainder of the proof we denote
\[
	D_n^{\mathrm{max}} := \max_{1 \le i \le n} D_i.
\]
and observe that by our choice of $\delta$,
\[
	\varepsilon_1 := \frac{2}{\gamma} - 1 - \delta = \frac{2 - \gamma}{\gamma} - \delta > 0 \quad \text{and} \quad
	\varepsilon_2 := 1 - \frac{1}{\gamma} - \delta = \frac{\gamma - 1}{\gamma} - \delta > 0.
\]

For $t = 1$, we have
\begin{align*}
	I_n^{(1)} = \hspace{-3pt} \sum_{1 \le i < j \le n} \frac{2L_nD_i^2 D_j}{(L_n - 2D_i + 1)^2} 
	\le \frac{2L_n^2 \sum_{i = 1}^n D_i^2}{\left(L_n^2 - 4 L_n D_n^{\mathrm{max}}\right)}
	= \left(2\sum_{i = 1}^n D_i^2 \right)\left(1 - \frac{4D_n^{\mathrm{max}}}{L_n}\right)^{-1}.
\end{align*}
By the strong law of large numbers and Proposition \ref{prop:scaling_sums_powers_regularly_varying}, it follows that
\[
	\frac{M_n}{L_n} = \frac{\mu n}{L_n} \frac{D_n^{\mathrm{max}}}{\mu n} \plim 0,
\]
and hence
\begin{equation}\label{eq:ecm_joint_moments_convergence_Mn_Ln}
	\left(1 - \frac{4D_n^{\mathrm{max}}}{L_n}\right)^{-1} \plim 1.
\end{equation}
Proposition \ref{prop:scaling_sums_powers_regularly_varying} then implies
\[
	n^{1 + \delta - \frac{4}{\gamma}}I_n^{(1)} \le \left(2n^{-\frac{2}{\gamma} - \varepsilon_1}\sum_{i = 1}^n D_i^2 \right)\left(1 - \frac{4D_n^{\mathrm{max}}}{L_n}\right)^{-1} \plim 0.
\]

The analysis for $I_n^{(2)}$ is similar so that we are left with $I_n^{(3)}$. Here, we have
\begin{align*}
	I_n^{(3)} &= \sum_{1 \le i < j \le n} \frac{D_i^3 D_j^2}{(L_n - 2D_i)^2}
		\le \frac{\sum_{i = 1}^n D_i^2 \sum_{j = 1}^n D_j^3}{\left(L_n^2 - 4 L_n D_n^{\mathrm{max}}\right)}\\
	&=\frac{1}{L_n^2}\left(\sum_{i = 1}^n D_i^2\right)\left(\sum_{j = 1}^n D_j^3\right)\left(1 - \frac{4D_n^{\mathrm{max}}}{L_n}\right)^{-1}.
\end{align*}
The last term again converges in probability to one, by \eqref{eq:ecm_joint_moments_convergence_Mn_Ln}. For the remaining terms we use the definition of $\varepsilon_2$ and Proposition \ref{prop:scaling_sums_powers_regularly_varying} to obtain 
\begin{align*}
	n^{1 + \delta - \frac{4}{\gamma}}I_n^{(3)} &\le n^{2 - \frac{5}{\gamma} + \varepsilon_2}\frac{1}{L_n^2}
		\left(\sum_{i = 1}^n D_i^2\right)\left(\sum_{j = 1}^n D_j^3\right)\left(1 - \frac{4D_n^{\mathrm{max}}}{L_n}\right)^{-1} \\
	&= \frac{n^2}{L_n^2}\left(n^{-\frac{2}{\gamma} - \frac{\varepsilon_2}{2}}\sum_{i = 1}^n D_i^2\right)
		\left(n^{-\frac{3}{\gamma} - \frac{\varepsilon_2}{2}}\sum_{j = 1}^n D_j^3\right)
		\left(1 - \frac{4D_n^{\mathrm{max}}}{L_n}\right)^{-1} \plim 0,
\end{align*}
which finishes the proof of \eqref{eq:ecm_joint_moments_main}.
\end{proof}

We proceed to prove the convergence of $\widehat{r}_n^+$:

\begin{proposition}\label{prop:pearson_ecm_positive_part}
Let ${\bf D}_n$ be sampled from $\mathscr{D}$ with $1< \gamma < 2$, $\Exp{\mathscr{D}} = \mu$ and let
$\widehat{G}_n = \emph{\texttt{ECM}}({\bf D}_n)$. Then, for any slowly-varying function $\mathcal{L}$,
\begin{equation}\label{eq:pearson_clt_ecm_r+}
	\mathcal{L}(n) n^{1 - \frac{1}{\gamma}}\widehat{r}^{\, +}_n \plim 0.
\end{equation}
\end{proposition}

\begin{proof}
Let
\[
	\delta = \left(\frac{2 - \gamma}{4\gamma} \wedge \frac{\gamma - 1}{4\gamma}\right).
\]
We will show that
\begin{equation}\label{eq:pearson_ecm_r+_main}
	n^{1 - \frac{1}{\gamma} + \delta} \, \widehat{r}^{\, +}_n \plim 0,
\end{equation}
which then implies \eqref{eq:pearson_clt_ecm_r+}, since by Potter's theorem $\mathcal{L}(n) n^{-\delta} \to 0$,
for any $\delta > 0$.

The main part of the proof of \eqref{eq:pearson_ecm_r+_main} will be to show that
\begin{equation}\label{eq:pearson_ecm_r+_joint_degrees}
	n^{1 - \frac{4}{\gamma} + 2\delta} \sum_{1 \le i < j \le n} \widehat{X}_{ij} D_i D_j \plim 0,
\end{equation}

To see that \eqref{eq:pearson_ecm_r+_joint_degrees} implies \eqref{eq:pearson_ecm_r+_main}, we write
\begin{align*}
	\widehat{r}^{\,+}_n &\le \frac{\sum_{1 \le i < j \le n} \widehat{X}_{ij} D_i
		D_j}{\sum_{i = 1}^n \widehat{D}_i^3 - \frac{1}{\widehat{L}_n} 
		\left(\sum_{i = 1}^n \widehat{D}_i^2\right)^2}\\
	&= \sum_{1 \le i < j \le n} \widehat{X}_{ij} D_i D_j \left(\frac{1}{\sum_{i = 1}^n \widehat{D}_i^3}\right)
		\left(1 - \frac{\left(\sum_{i = 1}^n \widehat{D}_i^2\right)^2}{\widehat{L}_n
		\sum_{i = 1}^n \widehat{D}_i^3}\right)^{-1}.
\end{align*}
By Lemma \ref{lem:scaling_pearson_ecm_main_fraction},
\[
	\left(1 - \frac{\left(\sum_{i = 1}^n \widehat{D}_i^2\right)^2}
	{\widehat{L}_n\sum_{i = 1}^n \widehat{D}_i^3}\right)^{-1} \plim 1,
\]
and hence, using \eqref{eq:pearson_ecm_r+_joint_degrees} and Proposition \ref{prop:scaling_sums_powers_regularly_varying},
\[
	n^{1 - \frac{1}{\gamma} + \delta} \, \widehat{r}^{\,+}_n \le \left(n^{1 - \frac{4}{\gamma} + 2\delta} 
	\hspace{-6pt}\sum_{1 \le i < j \le n} \widehat{X}_{ij} D_i D_j\right)
	\hspace{-4pt}\left(\frac{n^{-\frac{3}{\gamma} - \delta}}{\sum_{i = 1}^n \widehat{D}_i^3}\right)\hspace{-4pt}
	\left(1 - \frac{\left(\sum_{i = 1}^n \widehat{D}_i^2\right)^2}{\widehat{L}_n
			\sum_{i = 1}^n \widehat{D}_i^3}\right)^{-1} \hspace{-7pt} \plim 0.
\]

To prove \eqref{eq:pearson_ecm_r+_joint_degrees} let $\kappa = (\gamma - 1)/(\gamma + 1)$ and define the events
\begin{align}
	\mathcal{A}_n &= \left\{\left|L_n - \mu n\right| \le n^{1 - \kappa}\right\}, \label{eq:event_A_n} \\
	\mathcal{B}_n &= \left\{\sum_{1 \le i < j \le n}\left|\Probn{X_{ij = 0}} 
	- \exp\left\{-\frac{D_i D_j}{L_n}\right\}\right|D_i D_j \le n^{\frac{4}{\gamma} - 1- 3\delta} \notag
	\right\}. 
\end{align}
Then, if we set $\Lambda_n = \mathcal{A}_n \cap \mathcal{B}_n$, it follows from Lemma \ref{lem:Kolmogorov_SLLN} and Lemma \ref{prop:ecm_joint_moments} that $\Prob{\Lambda_n} \to 1$ and hence, it is enough to prove that for any $K > 0$
\begin{equation}\label{eq:pearson_ecm_r+_joint_degrees_cond}
	\lim_{n \to \infty} \Prob{n^{1 - \frac{4}{\gamma} + 2\delta}\sum_{1 \le i < j \le n} \widehat{X}_{ij} D_i
			D_j > K, \Lambda_n} = 0.
\end{equation}

First, since $\Expn{\widehat{X}_{ij}} = \Probn{X_{ij} > 0}$ and $\Lambda_n$ is completely determined by the degree sequence,
\begin{align*}
	\Expn{\sum_{1 \le i < j \le n} \widehat{X}_{ij} D_i	D_j\indE{\Lambda_n}} 
	&= \sum_{1 \le i < j \le n} \Probn{X_{ij} > 0} D_i D_j \indE{\Lambda_n}\\
	&= \sum_{1 \le i < j \le n} \left(1 - \Probn{X_{ij} = 0}\right)D_i D_j \indE{\Lambda_n}\\
	&\le \sum_{1 \le i < j \le n} \left(1 - \exp\left\{-\frac{D_iD_j}{L_n}\right\}\right)D_i D_j\indE{\Lambda_n}
		+ n^{\frac{4}{\gamma} - 1 - 3\delta}\\
	&\le \sum_{1 \le i < j \le n} \left(1 - \exp\left\{-\frac{D_iD_j}{\mu n - n^{1 - \kappa}}\right\}\right)
		D_i D_j + n^{\frac{4}{\gamma} - 1 - 3\delta}.
\end{align*}
From this we obtain, using Markov's inequality, that
\begin{align}
	&\hspace{-30pt}\Prob{n^{1 - \frac{4}{\gamma} + 2\delta}\sum_{1 \le i < j \le n} \widehat{X}_{ij} D_i
				D_j > K, \Lambda_n} \notag \\
	&\le \frac{n^{3 - \frac{4}{\gamma} + 2\delta}}
		{K}
		\Exp{\left(1 - \exp\left\{-\frac{\mathscr{D}_1 \mathscr{D}_2}{\mu n - n^{1 - \kappa}}\right\}\right)\mathscr{D}_1 
		\mathscr{D}_2} + \bigO{n^{-\delta}}, \label{eq:pearson_clt_ecm_r+_main_term}
\end{align}
where $\mathscr{D}_1$ and $\mathscr{D}_2$ are two independent copies of $\mathscr{D}$. It follows that
$\mathscr{D}_1\mathscr{D}_2$ is again regularly varying with exponent $1 < \gamma < 2$. Therefore, since
$1 - e^{-x} \le (1 \wedge x)$ and using Lemma \ref{lem:karamata},
\begin{align*}
	&\hspace{-30pt}\frac{n^{3 - \frac{4}{\gamma} + 2\delta}}
		{K}
		\Exp{\left(1 - \exp\left\{-\frac{\mathscr{D}_1 \mathscr{D}_2}{\mu n - n^{1 - \kappa}}\right\}
		\right)\mathscr{D}_1 \mathscr{D}_2} \\
	&\le \frac{n^{3 - \frac{4}{\gamma} + 2\delta}}
	{K}
		\Exp{\left(1 \wedge \frac{\mathscr{D}_1 \mathscr{D}_2}{\mu n - n^{1 - \kappa}}\right)
		\mathscr{D}_1 \mathscr{D}_2} \\
	&= \bigO{n^{4 - \frac{4}{\gamma} - \gamma + 2\delta}}.\mumberthis 			
		\label{eq:pearson_clt_ecm_r+_term_1}
\end{align*}
Now observe that by our choice of $\delta > 0$ and since $2 - \gamma < 1$, we get
\[
	4 - \frac{4}{\gamma} - \gamma + 2\delta = \frac{-(\gamma - 2)^2}{\gamma} + 2\delta
	\le \frac{-(\gamma - 2)^2}{\gamma} + \frac{2 - \gamma}{2\gamma}
	\le \frac{-(\gamma - 2)^2}{2\gamma} < 0,
\]
Plugging this into \eqref{eq:pearson_clt_ecm_r+_term_1}, it follows from \eqref{eq:pearson_clt_ecm_r+_main_term} that
\[
	\Prob{n^{1 - \frac{4}{\gamma} + 2\delta}\sum_{1 \le i < j \le n} \widehat{X}_{ij} D_i D_j 
		> K , \Lambda_n} = \bigO{n^{-\frac{(\gamma - 2)^2}{2 \gamma}} + n^{-\delta}}.
\]
and hence \eqref{eq:pearson_ecm_r+_joint_degrees_cond} follows.
\end{proof}

\subsection{Proving Theorem \ref{thm:clt_pearson_ecm}}

We are now ready to prove the central limit theorem for the ECM. 

\begin{proof}[Proof of Theorem \ref{thm:clt_pearson_ecm}]
Let $\widehat{G}_n = \texttt{ECM}({\bf D}_n)$ and $\mathcal{S}_{\gamma/2}$ and $\mathcal{S}_{\gamma/3}$ be defined as in \eqref{eq:def_S_gamma} and let $\mathcal{L}_0$ be
given by Proposition \ref{prop:clt_pearson_negative_part}. Now we write
\begin{align*}
	\mu\mathcal{L}_0(n)n^{1 - \frac{1}{\gamma}} r(\widehat{G}_n)
	&= \mu\mathcal{L}_0(n)n^{1 - \frac{1}{\gamma}} \widehat{r}^{\,+}_n
		- \mu\mathcal{L}_0(n)n^{1 - \frac{1}{\gamma}} \widehat{r}^{\,-}_n.
\end{align*}
By Proposition \ref{prop:pearson_ecm_positive_part} it follows that the first part converges to zero in probability, as $n \to \infty$. For the second part, let $\delta = (\gamma - 1)^2/(4 \gamma)$ and note that by Potter's theorem \cite[Theorem 1.5.6]{Bingham1989} we have that
$\mathcal{L}_0(n) \le n^{\delta}$ for all large enough $n$. Then, if we denote by $G_n$ the graph before the removal of edges, it follows by Proposition \ref{prop:pearson_min_cm_vs_ecm} that
\[
	\mu\mathcal{L}_0(n)n^{1 - \frac{1}{\gamma}} \left|r^{-}_n - \widehat{r}^{\,-}_n\right| \plim 0.
\]
Finally, we remark that the graph $G_n$ is generated by the CM so that the above and Proposition \ref{prop:clt_pearson_negative_part} now imply
\[
	\mu\mathcal{L}_0(n)n^{1 - \frac{1}{\gamma}} \widehat{r}^{\,-}_n \dlim 
	\frac{\mathcal{S}_{\gamma/2}^2}{\mathcal{S}_{\gamma/3}},
\]
as $n \to \infty$ from which the result follows.
\end{proof}

\subsection{Pearson in the CM: Proof of Lemma~\ref{lem-Pearon-CM-Var}}\label{sec:lempearsoncm}

In this section, we prove Lemma~\ref{lem-Pearon-CM-Var} on the tightness of the conditional variance of Pearson in the CM.
\begin{proof} Since, conditionally on the degrees, the only randomness in $r_n$ is in $(X_{ij})_{1\leq i<j\leq n}$, we use the covariance formula for sums of random variables to compute that
	\eqan{
		{\rm Var}_n(r_n)&=\frac{\sum_{i,j,k,l\in[n]} D_iD_jD_kD_l{\rm Cov}_n(X_{ij}, X_{kl})}{\Big[\sum_{i\in[n]} D_i^3-\Big(\sum_{i\in[n]} D_i^2\Big)^2/L_n\Big]^2}\\
		&=\frac{\sum_{i,j,k,l\in[n]} D_iD_jD_kD_l{\rm Cov}_n(X_{ij}, X_{kl})}{\Big(\sum_{i\in[n]} D_i^3\Big)^2}(1+\op(1)),\nonumber
	}
	since $\sum_{i\in[n]} D_i^3\gg \Big(\sum_{i\in[n]} D_i^2\Big)^2/L_n$ and where we write ${\rm Cov}_n$ for the conditional variance in the CM given the i.i.d.\ degrees.
	We next compute ${\rm Cov}_n(X_{ij}, X_{kl})$, depending on the precise form of $\{i,j,k,l\}$. For this, we note that
	\eqn{
		X_{ij}=\sum_{s=1}^{D_i} \sum_{t=1}^{D_j} I_{st},
	}
	with $I_{st}$ the indicator that the $s$th half-edge incident to $i$ pairs to the $j$th half-edge incident to $j$.
	
	\paragraph*{Case $(i,j)=(k,l)$ with $i<j$.} We compute that
	\eqan{
		{\rm Var}_n(X_{ij})&=\sum_{(s,t)}\sum_{(s',t')} {\rm Cov}_n(I_{st}, I_{s't'})\\
		&=\frac{D_i(D_i-1)D_j(D_j-1)}{(L_n-1)(L_n-3)}+\frac{D_iD_j}{(L_n-1)}-\frac{D_i^2D_j^2}{(L_n-1)^2}\nonumber\\
		&=\frac{D_iD_j}{(L_n-1)}+2\frac{D_i^2D_j^2}{(L_n-1)^2(L_n-3)}-\frac{D_iD_j(D_j-1)}{(L_n-1)(L_n-3)}
		-\frac{D_i^2D_j}{(L_n-1)(L_n-3)}\nonumber\\
		&=\frac{D_iD_j}{L_n}(1+\op(1)).\nonumber
	}
	Thus,
	\eqan{
		n \frac{\sum_{i<j\in[n]} D_i^2D_j^2{\rm Var}_n(X_{ij})}{\Big(\sum_{i\in[n]} D_i^3\Big)^2}
		&=\frac{n}{L_n}(1+\op(1)) \frac{\sum_{i<j\in[n]} D_i^3D_j^3}{\Big(\sum_{i\in[n]} D_i^3\Big)^2}\\
		&=\frac{1}{\mu}(1+\op(1)) \frac{\sum_{i<j\in[n]} D_i^3D_j^3}{\Big(\sum_{i\in[n]} D_i^3\Big)^2}.\nonumber
	}
	Since we sum over all $i,j\in[n]$ and not just $i<j$, this term appears 4 times.

	\paragraph*{Case $|\{i,j,k,l\}|=1$.} In this case, we obtain
	\eqan{
		{\rm Var}_n(X_{ii})&=\frac{D_i(D_i-1)}{(L_n-1)}+\frac{D_i(D_i-1)(D_i-2)(D_i-3)}{(L_n-1)(L_n-3)}-\frac{D_i^4}{(L_n-1)^2}\\
		&=\frac{D_i^2}{L_n}(1+\op(1))+\frac{D_i^4}{L_n^3}(1+\op(1))-6\frac{D_i^3}{L_n^2}(1+\op(1))\nonumber\\
		&=\frac{D_i^2}{L_n}(1+\op(1)),\nonumber
	}
	since $D_i=\bigOp{n^{\gamma}}$ with $\gamma\in(\tfrac{1}{2},1)$. Therefore,
	\eqan{
		n \frac{\sum_{i\in[n]} D_i^4{\rm Var}_n(X_{ii})}{\Big(\sum_{i\in[n]} D_i^3\Big)^2}
		&=\frac{n(1+\op(1))}{L_n}\frac{\sum_{i\in[n]} D_i^6}{\Big(\sum_{i\in[n]} D_i^3\Big)^2}.\nonumber
	}
	\medskip
	
	The above two computations show that these contributions sum up to
	\eqn{
		\frac{1}{\mu}(1+\op(1)) \frac{4\sum_{i<j\in[n]} D_i^3D_j^3}{\Big(\sum_{i\in[n]} D_i^3\Big)^2}
		+\frac{1}{\mu}(1+\op(1)) \frac{\sum_{i\in[n]} D_i^6}{\Big(\sum_{i\in[n]} D_i^3\Big)^2}
		\dlim \frac{2-\mathcal{S}_{\gamma/6}/\mathcal{S}_{\gamma/3}^2}{\mu}.
	}
	We are left to show that all other terms constitute error terms.

	\paragraph*{Case $|\{i,j,k,l\}|=4$.} When $|\{i,j,k,l\}|=4$, we compute that
	\eqan{
		{\rm Cov}_n(X_{ij}, X_{kl})&=D_iD_jD_lD_k \Big(\frac{1}{(L_n-1)(L_n-3)}-\frac{1}{(L_n-1)^2}\Big)\\
		&=\frac{2D_iD_jD_lD_k}{(L_n-1)^2(L_n-3)}.\nonumber
	}
	Thus, the contribution to the variance of $r_n$ of this case equals
	\eqn{
		\frac{\Big(\sum_{i\in[n]} D_i^2\Big)^4}{L_n^3\Big(\sum_{i\in[n]} D_i^3\Big)^2}(1+\op(1))
		=\bigOp{n^{2\gamma-3}}=\op(n^{-1}),
	}
	since $\gamma\in(\tfrac{1}{2},1)$.
	
	\paragraph*{Case $|\{i,j,k,l\}|=3$.} When $|\{i,j,k,l\}|=3$, we compute that
	\eqan{
		{\rm Cov}_n(X_{ij}, X_{il})&=
		\frac{D_i(D_i-1)D_jD_l}{(L_n-1)(L_n-3)}-\frac{D_i^2D_jD_l}{(L_n-1)^2}\nonumber\\
		&=2\frac{D_i^2D_jD_l}{(L_n-1)^2(L_n-3)}-\frac{D_iD_jD_l}{(L_n-1)(L_n-3)}
	}
	Thus, the contribution to the variance of $r_n$ of this case equals
	\eqan{
		&2\frac{\Big(\sum_{i\in[n]} D_i^2\Big)^2}{L_n^3\sum_{i\in[n]} D_i^3}(1+\op(1))-
		\frac{\Big(\sum_{i\in[n]} D_i^2\Big)^3}{L_n^2\Big(\sum_{i\in[n]} D_i^3\Big)^2}(1+\op(1))\\
		&\qquad=\bigOp{n^{\gamma-3}}+\bigOp{n^{-2}}=\op(n^{-1}),\nonumber
	}
	since $\gamma\in(\tfrac{1}{2},1)$.
	
	This completes the proof.
\end{proof}

\section{Clustering coefficient}\label{sec:clustering}
In this section, we prove Theorems~\ref{thm:CCM} and~\ref{thm:CECM} on the clustering coefficient in the configuration model as well as the erased configuration model. In both models, we first study the clustering coefficient when the degree sequence is fixed. We show that the clustering coefficient concentrates around its expected value when the degrees are given. Then, we analyze how the random degrees influence the clustering coefficient.

\subsection{Clustering in the configuration model}\label{ssec:clustering_cm}
In this section, we compute the clustering coefficient for a configuration model with a power-law degree distribution with $\gamma\in(1,2)$. To prove Theorem~\ref{thm:CCM}, we first use a second moment method to show that the number of triangles $\triangle_n$ concentrates on its expected value conditioned on the degrees.
Then we take the random degrees into account and show that the rescaled clustering coefficient converges to the stable distributions from Theorem~\ref{thm:CCM}.

\subsubsection{Concentration for the number of triangles}

The concentration result is formally stated and proved in the next lemma. 
\begin{lemma}\label{lem:exvar}
Let ${\bf D}_n$ be sampled from $\mathscr{D}$ with $1< \gamma < 2$, and
${G}_n = \emph{\texttt{CM}}({\bf D}_n)$. Let $\triangle_n$ denote the number of triangles in in $G_n$. Then, for any $\varepsilon>0$,
\[
	\lim_{n\to\infty}\Probn{\abs{\triangle_n-\Expn{\triangle_n}}>\varepsilon\Expn{\triangle_n}}=0.
\]
\end{lemma}

\begin{proof}  Fix $0<\delta<(\tfrac{2}{\gamma}-1)/6$. Define the event
\begin{equation}\label{eq:event_sumDsq}
	\mathcal{B}_n=\Big\{\sum_{i=1}^{n}D_i^2\geq Kn^{ 2/\gamma-\delta}, \sum_{i=1}^{n}D_i^3\leq Kn^{ 3/\gamma+\delta}\Big \} .
\end{equation} 
Let $\mathcal{A}_n$ be the event defined in~\eqref{eq:event_A_n}, and let $\Lambda_n=\mathcal{B}_n\cap \mathcal{A}_n$.
We have $\lim_{n\to\infty}\Probn{\Lambda_n}=1$, thus, we only need to prove the result on the event $\Lambda_n$.
The proof is similar to the proof of~\cite[Proposition 7.13]{VanDerHofstad2016a}.

Define
\begin{align*}
\mathcal{I}&=\{(s_1t_1,s_2u_1,u_2t_2,i,j,k) :1\leq i<j<k\leq n, 1\leq s_1<s_2\leq { D}_i,1\leq t_1\neq t_2\leq { D}_j,\\
& \quad 1\leq u_1\neq u_2\leq { D}_k\}.
\end{align*}
This is the set of combinations of six distinct half-edges that could possibly form a triangle on three distinct vertices. Thus,  
\begin{align}\label{eq:_absI}
|\mathcal{I}|=\sum_{1\leq i<j<k\leq n}D_i(D_i-1)D_j(D_j-1)D_k(D_k-1)
\end{align}
denotes the number of ways six half-edges could form a triangle.
For $m\in\mathcal{I}$, let $\mathbbm{1}_{m}$ denote the indicator variable that the six half-edges of $m$ form a triangle in the way specified by $m$. 
Then,
\[
\triangle_n = \sum_{m\in\mathcal{I}}\mathbbm{1}_{m}.
\]
The probability that the half-edges in $m$ form a triangle can be written as
\[
	\Probn{\mathbbm{1}_m=1}=\prod_{j=1}^{3}(L_n+1-2j)^{-1}.
\]
This results in
\begin{equation}\label{eq:exptriangcm}
	\Expn{\triangle_n}=\sum_{m\in\mathcal{I}}\Probn{\mathbbm{1}_m=1} = \frac{|\mathcal{I}|}{\prod_{j=1}^{3}(L_n+1-2j)}.
\end{equation}
Furthermore, by~\cite[Theorem 2.5]{VanDerHofstad2016a},
\[
\Expn{\triangle_n(\triangle_n-1)}=\sum_{m_1\neq m_2\in\mathcal{I}}\Probn{\mathbbm{1}_{m_1}=\mathbbm{1}_{m_2}=1}.
\]
When all six pairs of half-edges involved in $m_1$ and $m_2$ are distinct, the probability that these pairs of half-edges that form $m_1$ and $m_2$  are paired in the correct way is
\[
	\Probn{\mathbbm{1}_{m_1}=\mathbbm{1}_{m_2}=1}=\prod_{j=1}^{6}(L_n+1-2j)^{-1},
\]
If $m_1$ and $m_2$ contain one pair of half-edges that is the same (so that $m_1$ and $m_2$ form two triangles merged on one edge), 
\[
\Probn{\mathbbm{1}_{m_1}=\mathbbm{1}_{m_2}=1}=\prod_{j=1}^{5}(L_n+1-2j)^{-1}.
\]
Let $\mathcal{I}_2$ denote the set of combinations of 10 half-edges that could possibly form two triangles merged by one edge. Then, similarly to~\eqref{eq:_absI},
\begin{align}
|\mathcal{I}_2| &= \sum_{1\leq i < j\leq n} D_i(D_i-1)(D_i-2)D_j(D_j-1)(D_j-2) \sum_{1\leq k < l\leq n}D_k(D_k-1)D_l(D_l-1)\nonumber\\
& \leq \Big(\sum_{i \in[n]}D_i^3\Big)^2\Big(\sum_{i \in[n]}D_i^2\Big)^2.
\end{align}
Similarly, when $m_1$ and $m_2$ overlap at two pairs of half-edges (so that $m_1$ and $m_2$ form two triangles merged by two edges)
	\[
\Probn{\mathbbm{1}_{m_1}=\mathbbm{1}_{m_2}=1}=\prod_{j=1}^{4}(L_n+1-2j)^{-1}.
\]
	Let $\mathcal{I}_3$ denote the set of combinations of 8 half-edges that could possibly form two triangles merged by two edges. Then, similarly to~\eqref{eq:_absI},
\begin{align}
|\mathcal{I}_3| &= \sum_{1\leq i < j\leq n} D_i(D_i-1)(D_i-2)D_j(D_j-1)(D_j-2) \sum_{1\leq k \leq n}D_k(D_k-1)\nonumber\\
& \leq \Big(\sum_{i \in[n]}D_i^3\Big)^2\sum_{i \in[n]}D_i^2.
\end{align}
In all other cases the probability of the event $\mathbbm{1}_{m_1}=\mathbbm{1}_{m_2}=1$ then equals zero. These are cases where $m_1$ prescribes some half-edge to be merged to half-edge $j_1$, whereas $m_2$ prescribes it to be merged to some other half-edge $j_2$. 
Therefore,
\begin{align}\label{eq:var3terms}
\Expn{\triangle_n(\triangle_n-1)}\leq\frac{|\mathcal{I}|^2}{\prod_{j=1}^{6}(L_n+1-2j)}+\frac{|\mathcal{I}_2|}{\prod_{j=1}^{5}(L_n+1-2j)}+\frac{|\mathcal{I}_3|}{\prod_{j=1}^{4 }(L_n+1-2j)}.
\end{align}

On the event $\mathcal{B}_n$ defined in~\eqref{eq:event_sumDsq}, 
\begin{align*}
\frac{|\mathcal{I}_2|}{\prod_{j=1}^{5}(L_n+1-2j)}\Big/\frac{|\mathcal{I}|^2}{\prod_{j=1}^{6}(L_n+1-2j)}& =O\Bigg(  \frac{\Big(\sum_{i\in[n]}D_i^3\Big)^2/\Big(\sum_{i\in[n]}D_i^2\Big)^4}{(L_n+1-12)^{-1}}\Bigg)\\
& = \bigO{n \cdot n^{6/\gamma+2\delta-8/\gamma+4\delta}} = \bigO{n^{1-2/\gamma+6\delta }}\\
& = o(1),
\end{align*}
by the choice of $\delta$. In a similar way, we can show that the third term is small compared to the first term of~\eqref{eq:var3terms}.
Therefore,
\[
\Expn{\triangle_n(\triangle_n-1)}\leq\frac{|\mathcal{I}|^2}{\prod_{j=1}^{6}(L_n+1-2j)}(1+o(1)).
\]	
Finally, on the event $\mathcal{B}_n$ we have, 
\[
|\mathcal{I}|= \Theta\Big(\Big(\sum_i D_i^2\Big)^3\Big)= \Omega(n^{6/\gamma-3\delta}).
\]  
Using that $L_n=\mu n(1+o(1))$ on the event $\mathcal{A}_n$ results in
	\begin{align*}
\frac{	\Varn{\triangle_n}}{(\Expn{\triangle_n})^2}&=\frac{ \Expn{\triangle_n(\triangle_n-1)}}{(\Expn{\triangle_n})^2}-1+\frac{\Expn{\triangle_n}}{(\Expn{\triangle_n})^2}\\
& \leq \frac{(L_n-1)(L_n-3)(L_n-5)}{(L_n-7)(L_n-9)(L_n-11)}(1+o(1))-1+ \frac{\prod_{j=1}^{3}(L_n+1-2j)}{|\mathcal{I}|}\\
&=1+o(1)-1+O(n^{3 + 3\delta - \frac{6}{\gamma}})=o(1),
\end{align*}
	for $\gamma\in(1,2)$.  Then by Chebyshev's inequality, on the event $\Lambda_n$
\[
\Probn{\abs{\triangle_n-\Expn{\triangle_n}}>\varepsilon\Expn{\triangle_n}}\leq \frac{\Varn{\triangle_n}}{(\Expn{\triangle_n})^2\varepsilon^2}=o(1),
\]
which gives the result.
\end{proof}

\subsubsection{Proof of Theorem~\ref{thm:CCM}}

\begin{proof}[Proof of Theorem~\ref{thm:CCM}] We again prove the result under the event $\Lambda_n=\mathcal{B}_n\cap \mathcal{A}_n$, where $\mathcal{B}_n$ and $\mathcal{A}_n$ are given respectively in \eqref{eq:event_sumDsq} and \eqref{eq:event_A_n}.

By~\eqref{eq:_absI} and~\eqref{eq:exptriangcm}
\begin{align*}
\Expn{\triangle_n}& = \frac{1}{\mu^3 n^3}\sum_{1\leq i<j<k\leq n}D_i(D_i-1)D_j(D_j-1)D_k(D_k-1)(1+\op(1))\\
&=\frac{1}{\mu^3 n^3}\Big(\frac 16\Big(\sum_{i=1}^nD_i(D_i-1)\Big)^3-\frac{1}{2}\sum_{i=1}^{n}D_i(D_i-1)\sum_{j=1}^{n}D_j^2(D_j-1)^2\\
&\quad +\frac 13 \sum_{i=1}^{n}D_i^3(D_i-1)^3\Big)(1+\op(1))
\end{align*}
Then, the definition of the clustering coefficient in~\eqref{eq:clust} yields
\begin{equation}\label{eq:ccmappr}
\begin{aligned}
	\Expn{C_n}& = \frac{1}{\mu^3 n^3}\Big(\Big(\sum_{i=1}^nD_i(D_i-1)\Big)^2-3 \sum_{i=1}^{n}D_i^2(D_i-1)^2\\
	&\quad +2\frac{\sum_{i=1}^{n}D_i^3(D_i-1)^3}{\sum_{i=1}^{n}D_i(D_i-1)}\Big)(1+\op(1))\\
	& =   \frac{1}{\mu^3 n^3}\Big(\Big(\sum_{i=1}^nD_i^2\Big)^2-3\sum_{i=1}^{n}D_i^4+2\sum_{i=1}^{n}D_i^6/\sum_{i=1}^{n}D_i^2\Big)(1+\op(1)),
\end{aligned}
\end{equation}
Lemma~\ref{lem:exvar} then gives that conditioned on the degree sequence
\[
	C_n=   \frac{1}{\mu^3 n^3}\Big(\Big(\sum_{i=1}^nD_i^2\Big)^2-3\sum_{i=1}^{n}D_i^4+2\sum_{i=1}^{n}D_i^6/\sum_{i=1}^{n}D_i^2\Big)(1+\op(1)).
\]

By~\eqref{eq:joint_convegrence_degree_sums}, there exist a slowly-varying functions $\mathcal{L}_1(n),\mathcal{L}_2(n)$ and $\mathcal{L}_3(n)$ such that 
\begin{align}\label{eq:jointconv2}
	\Big({n^{-2/\gamma}}\frac{\sum_{i=1}^{n}{D}_i^2}{\mathcal{L}_1(n)},
	{n^{-4/\gamma}}\frac{\sum_{i=1}^{n}{D}_i^4}{\mathcal{L}_2(n)},{n^{-6/\gamma}}\frac{\sum_{i=1}^{n}D_i^6}{\mathcal{L}_3(n)}\Big)\dlim 
	\left(\mathcal{S}_{\gamma/2},\mathcal{S}_{\gamma/4},{\mathcal{S}_{\gamma/6}}\right),
\end{align}
where
\[
	\mathcal{S}_{\gamma/2} = \sum_{i=1}^{\infty}\Gamma_i^{-2/\gamma},\quad 
	\mathcal{S}_{\gamma/4} = \sum_{i=1}^{\infty}\Gamma_i^{-4/\gamma},\quad 
	\mathcal{S}_{\gamma/6} = \sum_{i=1}^{\infty}\Gamma_i^{-6/\gamma},\quad 
\]
for the same random variables $\Gamma_i$. 
Furthermore, by~\cite[Eq (5.23)]{whitt2006}, the slowly-varying functions in~\eqref{eq:jointconv2} satisfy for some slowly-varying function $\mathcal{L}_0(n)$,
\[
\mathcal{L}_1(n)=\sqrt{\mathcal{L}_0(n)}(\hat{C}_{\gamma/2})^{2/\gamma},\quad \mathcal{L}_2(n)=\mathcal{L}_0(n)(\hat{C}_{\gamma/4})^{4/\gamma},\quad \mathcal{L}_3(n)=\mathcal{L}_0(n)^{3/2}(\hat{C}_{\gamma/6})^{6/\gamma},
\]
where 
\[
\hat{C}_\alpha = \frac{1-\alpha}{\Gamma(2-\alpha)\cos(\pi\alpha/2)},
\]
with $\Gamma$ the Gamma-function. Therefore,
\begin{align*}
	&\hspace{-30pt}\frac{n^{-4/\gamma}}{\mathcal{L}_0(n)}\Big((\sum_{i=1}^{n}{D}_i^2)^2,
\sum_{i=1}^{n}{D}_i^4,\frac{\sum_{i=1}^{n}D_i^6}{\sum_{i=1}^{n}{D}_i^2}\Big)\\
&\dlim 
\left((\hat{C}_{\gamma/2})^{\gamma/4}\mathcal{S}_{\gamma/2}^2,(\hat{C}_{\gamma/4})^{\gamma/4}\mathcal{S}_{\gamma/4},\frac{(\hat{C}_{\gamma/6})^{\gamma/6}\mathcal{S}_{\gamma/6}}{(\hat{C}_{\gamma/2})^{\gamma/2}\mathcal{S}_{\gamma/2}}\right),
\end{align*}
Combining this with~\eqref{eq:ccmappr} results in~\eqref{eq:ccm}.
\end{proof}

\subsection{ Clustering coefficient in the erased configuration model }\label{ssec:clustering_ecm}
In this section, we study the clustering coefficient in the \texttt{ECM}. Again, we start with the expectation and the variance of the clustering coefficient conditioned on the sampled degree sequence, i.e the sequence ${\bf D}_n = \{D_1, D_2, \dots, D_n\}$ sampled from the distribution~\eqref{eq:distribution_degrees}. Note that this is not the eventual degree sequence of the graph constructed by the erased configuration model.

\paragraph*{Structure of the proof of Theorem~\ref{thm:CECM}.}

We prove Theorem~\ref{thm:CECM} in four steps:

\emph{Step 1.} We show in Lemma~\ref{lem:sqrt} that the expected contribution to the number of triangles from vertices with sampled degrees larger than $\sqrt{n}/\varepsilon$ and smaller than $\varepsilon\sqrt{n}$ is small for fixed $0<\varepsilon<1$. Therefore, in the rest of the proof we focus on only counting triangles between vertices of degrees $[\varepsilon\sqrt{n},\sqrt{n}/\varepsilon]$. 

\emph{Step 2.} We calculate the expected number of triangles between vertices of sampled degrees proportional to $\sqrt{n}$, conditioned on the degree sequence. In Lemma~\ref{lem:extriang}, we show that this expectation can be written as the sum of a function of the degrees.

\emph{Step 3.} We show that the variance of the number of triangles between vertices of sampled degree proportional to $\sqrt{n}$ is small in Lemma~\ref{lem:exvarECM}. Thus, we can replace the number of triangles conditioned on the degrees by its expected value, which we computed in {\it Step 2}.

\emph{Step 4.} We show that the expected number of triangles conditioned on the sampled degrees converges to the value given in Theorem~\ref{thm:CECM}, when taking the random degrees into account.

We will start by proving the three lemma's described above.
Let $B_n(\varepsilon)$ denote the interval $[\varepsilon\sqrt{\mu n}, \sqrt{\mu n}/\varepsilon]$ for some $\varepsilon>0$. Furthermore, let $\widehat{X}_{ij}$ denote the number of edges between vertex $i$ and $j$ in the erased configuration model. Then, we can write the number of triangles as
\begin{align*}
\triangle_n &= \hspace{-6pt}\sum_{1\leq i< j< k\leq n} 
	\hspace{-4pt} \widehat{X}_{ij}\widehat{X}_{jk}\widehat{X}_{ik}\mathbbm{1}_{\{D_i,D_j, D_k\in B_n(\varepsilon)\}}
	+ \hspace{-6pt} \sum_{1\leq i< j< k\leq n} \hspace{-4pt} \widehat{X}_{ij}\widehat{X}_{jk}\widehat{X}_{ik}\mathbbm{1}_{\{D_i, D_j \text{ or }D_k \notin B_n(\varepsilon)\}}\\
&=:\triangle_n(B_n(\varepsilon))+\triangle_n(\bar{B}_n(\varepsilon)) \mumberthis \label{eq:splitX}
\end{align*}
We want to show that the major contribution to $\triangle_n$ comes from $\triangle_n(B_n(\varepsilon))$. 
The following lemma shows that the expected contribution of $\triangle_n(\bar{B}_n(\varepsilon))$ to the number of triangles is small.

\begin{lemma}\label{lem:sqrt}
Let ${\bf D}_n$ be sampled from $\mathscr{D}$ with $1< \gamma < 2$ and
$\widehat{G}_n = \emph{\texttt{ECM}}({\bf D}_n)$. Let $\triangle_n(\bar{B}_n(\varepsilon))$ denote the number of triangles in $\widehat{G}_n$ with at least one of the sampled degrees not in $B_n(\varepsilon)$. Then, 
\[
\limsup_{n\to\infty}\frac{\Exp{\triangle_n(\bar{B}_n(\varepsilon))}}{\mathcal{L}(\sqrt{\mu n})^3n^{\frac 32 (2-\gamma)}}=\bigO{\mathcal{E}_1(\varepsilon)}
\]
for some function $\mathcal{E}_1(\varepsilon)$ not depending on $n$ such that $\mathcal{E}_1(\varepsilon)\to 0$ as $\varepsilon\to 0$. 
\end{lemma}

\begin{proof}
Let $\triangle_{i,j,k}$ denote the event that a triangle is present on vertices $i,j$ and $k$. By~\cite[Lemma 4.1]{hofstad2017d}, 
\begin{align*}
\Probn{\triangle_{i,j,k}}& =\Theta\Big(\prod_{(u,v)\in\{(i,j),(j,k),(i,k)\}}(1-\me^{-D_uD_v/L_n})\ind{D_uD_v<L_n}\Big)\\
& = \Theta\Bigg(\left(\frac{D_iD_j}{L_n} \wedge 1\right) \left(\frac{D_iD_k}{L_n} \wedge 1\right)\left(\frac{D_jD_k}{L_n}\wedge 1\right)\Bigg).
\end{align*}
Therefore, for some $\tilde{K}>0$
\begin{align*}
&\hspace{-20pt}\Expn{\triangle_n(\bar{B}_n(\varepsilon))}\\
&\leq \tilde{K}\sum_{\mathclap{1\leq i< j< k\leq n}}  \left(\frac{D_iD_j}{L_n} \wedge 1\right) \left(\frac{D_iD_k}{L_n} \wedge 1\right)\left(\frac{D_jD_k}{L_n}\wedge 1\right)\ind{D_i,D_j,\text{ or }D_k \in \bar{B}_n(\varepsilon)}.
\end{align*}
Thus,
\begin{equation}\label{eq:extriang}
\begin{aligned}[b]
&\Exp{\triangle_n(\bar{B}_n(\varepsilon))}\\
&\leq  \tilde{K}\tfrac12 n^3\Exp{\left(\frac{{D}_1{D}_2}{\mu n}\wedge 1\right)\left(\frac{{D}_1{D}_3}{\mu n} \wedge 1\right)\left(\frac{{D}_2{D}_3}{\mu n}\wedge 1\right)\ind{{D}_1 \in \bar{B}_n(\varepsilon)}}.
\end{aligned}
\end{equation}
We now show that the contribution to~\eqref{eq:extriang} where $D_1<\varepsilon\sqrt{\mu n}$ is small.
We write
\begin{align*}
&\Exp{\left(\frac{{D}_1{D}_2}{\mu n}\wedge 1\right)\left(\frac{{D}_1{D}_3}{\mu n}\wedge 1\right)\left(\frac{{D}_2{D}_3}{\mu n} \wedge 1\right)\ind{{D}_1<\varepsilon\sqrt{\mu n}}}\\
&\leq \tilde{K}\sum_{t_1=1}^{\varepsilon\sqrt{\mu n}}\sum_{t_2=1}^{n}\sum_{t_3=1}^{n}\Prob{D_1=t_1,D_2=t_2,D_3=t_3}\left(\frac{t_1t_2}{\mu n} \wedge 1\right)\left(\frac{t_1t_3}{\mu n} \wedge 1\right)\left(\frac{t_2t_3}{\mu n} \wedge 1\right)\\
&\leq \tilde{K}K^3 \hspace{-3pt} \sum_{t_1=1}^{\varepsilon\sqrt{\mu n}}\sum_{t_2,t_3=1}^{n} \hspace{-5pt} \mathcal{L}(t_1)\mathcal{L}(t_2)\mathcal{L}(t_3)(t_1t_2t_3)^{-\gamma-1}\left(\frac{t_1t_2}{\mu n} \wedge 1\right) \hspace{-2pt} \left(\frac{t_1t_3}{\mu n} \wedge 1\right) \hspace{-2pt}\left(\frac{t_2t_3}{\mu n} \wedge 1\right)
	\mumberthis \label{eq:exdsmall}
\end{align*}
where we used Assumption \ref{ass:F} and the fact that $D_1, D_2$ and $D_3$ are independent. 
By~\cite[Lemma~4.1(ii)]{hofstad2017d},
	\begin{align}\label{eq:inttriangfinite}
	\int_{0}^{\infty}\int_{0}^{\infty}\int_{0}^{\infty}(xyz)^{-\gamma-1} \left(xy \wedge 1\right) \left(xz \wedge 1\right) \left(yz \wedge 1\right)\dd x \dd y \dd z<\infty,
	\end{align}
for all $\gamma\in(1,2)$. Therefore, by~\cite[Theorem~2]{karamata1962},
\begin{align*}
&\int_{0}^{\varepsilon\sqrt{\mu n}}\!\!\int_{0}^{\infty}\int_{0}^{\infty}\!\!\mathcal{L}(t_1)\mathcal{L}(t_2)\mathcal{L}(t_3)(t_1t_2t_3)^{-\gamma-1} \left(\frac{t_1t_2}{\mu n} \wedge 1\right) \!\!\left(\frac{t_1t_3}{\mu n} \wedge 1\right)\!\! \left(\frac{t_2t_3}{\mu n} \wedge1\right)\!\!\dd t_3 \dd t_2 \dd t_1\\
&= (\mu n)^{-\frac{3}{2}\gamma}\int_{0}^{\varepsilon}\!\!\int_{0}^{\infty}\int_{0}^{\infty}\!\!\mathcal{L}(x\sqrt{\mu n})\mathcal{L}(y\sqrt{\mu n})\mathcal{L}(z\sqrt{\mu n})(xyz)^{-\gamma-1} \left(xy\wedge 1\right) \!\!\left(xz\wedge 1\right)\!\! \left(yz\wedge 1\right)\!\!\dd x \dd y \dd z\\
& = (1+o(1))\mathcal{L}(\sqrt{\mu n})^3\!\!\int_{0}^{\varepsilon}\!\!\int_{0}^{\infty}\int_{0}^{\infty}(xyz)^{-\gamma-1} \left(xy\wedge 1\right) \!\!\left(xz\wedge 1\right)\!\! \left(yz\wedge 1\right)\!\!\dd x \dd y \dd z
\end{align*}

We then bound the sum in~\eqref{eq:exdsmall} as 
\begin{align*}
&\sum_{t_1=1}^{\varepsilon\sqrt{\mu n}}\sum_{t_2=1}^{n}\sum_{t_3=1}^{n}\mathcal{L}(t_1)\mathcal{L}(t_2)\mathcal{L}(t_3)(t_1t_2t_3)^{-\gamma-1}\left(\frac{t_1t_2}{\mu n} \wedge 1\right) \left(\frac{t_1t_3}{\mu n} \wedge 1\right) \left(\frac{t_2t_3}{\mu n} \wedge 1\right)\\
& \leq 2\int_{0}^{\varepsilon\sqrt{\mu n}+1}\int_{0}^{\infty}\int_{0}^{\infty}\mathcal{L}(t_1)\mathcal{L}(t_2)\mathcal{L}(t_3)(t_1t_2t_3)^{-\gamma-1}\\
& \quad \times \left(\frac{t_1t_2}{\mu n} \wedge 1\right) \left(\frac{t_1t_3}{\mu n} \wedge 1\right) \left(\frac{t_2t_3}{\mu n} \wedge1\right)\dd t_3 \dd t_2 \dd t_1\\
	& =2(1+o(1))(\mu n)^{-\frac{3}{2}\gamma}\mathcal{L}(\sqrt{\mu n})^3
\int_{0}^{\varepsilon} \hspace{-3pt} \int_{0}^{\infty} \hspace{-3pt} \int_{0}^{\infty}(xyz)^{-\gamma-1} \left(xy \wedge 1\right) \left(xz \wedge 1\right) \left(yz \wedge 1\right)\dd x \dd y \dd z.
\end{align*}
Therefore,
\begin{equation}\label{eq:exslowly}
\begin{aligned}[b]
&\Exp{\left(\frac{{D}_1{D}_2}{\mu n} \wedge 1\right)\left(\frac{{D}_1{D}_3}{\mu n} \wedge 1\right) \left(\frac{{D}_2{D}_3}{\mu n} \wedge 1\right)\ind{{D}_1<\varepsilon\sqrt{\mu n}}}\\
&\leq K_2 n^{\frac{3}{2}(1-\gamma)}\mathcal{L}(\sqrt{\mu n})^3
\int_{0}^{\varepsilon} \hspace{-3pt} \int_{0}^{\infty} \hspace{-3pt} \int_{0}^{\infty}(xyz)^{-\gamma-1} \left(xy \wedge 1\right) \left(xz \wedge 1\right) \left(yz \wedge 1\right)\dd x \dd y \dd z,
\end{aligned}
\end{equation}
for some constant $K_2$.
Thus, we only need to prove that the last triple integral in~\eqref{eq:exslowly} tends to zero as $\varepsilon\to 0$. 
Using~\eqref{eq:inttriangfinite}, we obtain
\[
\int_{0}^{\varepsilon}\int_{0}^{\infty}\int_{0}^{\infty}(xyz)^{-\gamma-1} \left(xy \wedge 1\right) \left(xz \wedge 1\right) \left(yz \wedge 1\right)\dd x \dd y \dd z := \mathcal{E}_0(\varepsilon)
\]
where $\mathcal{E}_0(\varepsilon)$ is such that $\mathcal{E}_0(\varepsilon)\to 0$ as $\varepsilon\to 0$. 
Thus, by~\eqref{eq:extriang} and~\eqref{eq:exdsmall}, the contribution to the expectation where one of the degrees is smaller than $\varepsilon\sqrt{\mu n}$ is bounded by $\bigOp{\mathcal{L}(\sqrt{\mu n})^3n^{3-\frac 3 2 \gamma}\mathcal{E}_0(\varepsilon)}$. 
Similarly,
\[
\int_{1/\varepsilon}^\infty\int_{0}^{\infty}\int_{0}^{\infty}(xyz)^{-\gamma-1} \left(xy \wedge 1\right) \left(xz \wedge 1\right) \left(yz \wedge 1\right)\dd x \dd y \dd z := \mathcal{E}_0^\prime(\varepsilon)
\]
where again $\mathcal{E}_0^\prime(\varepsilon)$ satisfies $\mathcal{E}_0^\prime(\varepsilon)\to 0$ as $\varepsilon\to 0$. Therefore, we can show in a similar way that the contribution to the expected number of triangles where one of the degrees is larger than $\sqrt{\mu n}/\varepsilon$ is $\bigOp{\mathcal{L}(\sqrt{\mu n})^3 n^{3-\frac 3 2 \gamma}\mathcal{E}_0^\prime(\varepsilon)}$. Then, taking $\mathcal{E}_1(\varepsilon)=\max(\mathcal{E}_0(\varepsilon),\mathcal{E}_0^\prime(\varepsilon))$ proves the lemma.
\end{proof}

The next lemma computes the expected contribution of the vertices of sampled degree proportional to $\sqrt{n}$ to the number of triangles.
Define
\begin{equation}\label{eq:g}
g_{n, \varepsilon} (D_i,D_j,D_k):=\left(1-\me^{-\frac{D_iD_j}{L_n}}\right)\left(1-\me^{-\frac{D_iD_k}{L_n}}\right)\left(1-\me^{-\frac{D_kD_j}{L_n}}\right)\ind{(D_i,D_j,D_k)\in B_n(\varepsilon)},
\end{equation}
and let $\sum^\prime$ denote a sum over distinct indices, such that $i<j<k$.

\begin{lemma}\label{lem:extriang}
		Let ${\bf D}_n$ be sampled from $\mathscr{D}$ with $1< \gamma < 2$, and
		$\widehat{G}_n = \emph{\texttt{ECM}}({\bf D}_n)$. Let $\triangle_n(B_n(\varepsilon))$ denote the number of triangles in $\widehat{G}_n$ with sampled degrees in $B_n(\varepsilon)$. Then, on the event $\mathcal{A}_n$ as defined in~\eqref{eq:event_A_n},
	\[
	\Expn{\triangle_n(B_n(\varepsilon))}=(1+o(1))\sideset{}{'}\sum_{i,j,k} g_{n, \varepsilon}(D_i,D_j,D_k).
		\]
\end{lemma}

\begin{proof}
We can write the expectation as
\begin{equation}\label{eq:expBne}
	\Expn{\triangle_n(B_n(\varepsilon))}= \sideset{}{'}\sum_{i,j,k}\Probn{\triangle_{i,j,k}}\ind{(D_i,D_j,D_k)\in B_n(\varepsilon)},
\end{equation}
where $\Probn{\triangle_{i,j,k}}$ denotes the probability of a triangle between vertices $i,j,k$ being present. This probability can be written as
\begin{equation}\label{eq:Pijk}
\begin{aligned}[b]
\Probn{\triangle_{i,j,k}} &=
1-\Probn{X_{ij}=0}-\Probn{X_{ik}=0}-\Probn{X_{jk}=0}+ \Probn{X_{ij}=X_{ik}=0}\\
&\quad +\Probn{X_{ij}=X_{jk}=0}+\Probn{X_{ik}=X_{jk}=0}\\
&\quad -\Probn{X_{ij}=X_{ik}=X_{jk}=0}.
\end{aligned}
\end{equation}
Because $D_i,D_j,D_k\leq \sqrt{n}/\varepsilon$ and $L_n=\mu n(1+o(1))$ under the event $\mathcal{A}_n$, we can use~\cite[Lemma 3.1]{hofstad2017d}, which calculates the probability that an edge is present conditionally on the presence of other edges in configuration models with arbitrary degree distributions. This results in  
\begin{align*}
&\hspace{-30pt}\Probn{X_{ij}=X_{ik}=X_{jk}=0}\\
&=\Probn{X_{ij}=0}\Probn{X_{jk}=0\mid X_{ij}=0}\Probn{X_{ik}=0\mid X_{ij}=X_{jk}=0}\\
& = \me^{-\frac{D_i D_j}{L_n}}\me^{-\frac{D_j D_k}{L_n}}\me^{-\frac{D_i D_k}{L_n}}(1+o(1)),
\end{align*} 
and similarly
\[
\Probn{X_{ij}=X_{jk}=0}=\me^{-\frac{D_i D_j}{L_n}}\me^{-\frac{D_j D_k}{L_n}}(1+o(1)).
\]

Combining this with~\eqref{eq:Pijk} yields
\begin{align*}
\Probn{\triangle_{i,j,k}}
&=1-\Big(\me^{-\frac{D_i D_j}{L_n}}+\me^{-\frac{D_i D_j}{L_n}}+\me^{-\frac{D_j D_k}{L_n}}\Big)(1+o(1))\\
& \quad +\Big(\me^{-\frac{D_i D_j}{L_n}}\me^{-\frac{D_j D_k}{L_n}}+\me^{-\frac{D_i D_j}{L_n}}\me^{-\frac{D_i D_k}{L_n}}+\me^{-\frac{D_i D_k}{L_n}}\me^{-\frac{D_j D_k}{L_n}}\Big)
(1+o(1))\\
&\quad -\me^{-\frac{D_i D_j}{L_n}}\me^{-\frac{D_j D_k}{L_n}}\me^{-\frac{D_i D_k}{L_n}}(1+\op(1))\\
& = (1+o(1))\left(1-\me^{-\frac{D_i D_j}{L_n}}\right)\left(1-\me^{-\frac{D_j D_k}{L_n}}\right)\left(1-\me^{-\frac{D_j D_k}{L_n}}\right),
	\mumberthis \label{eq:pind}
\end{align*}
where the second equality follows because $\varepsilon\sqrt{n}<D_i,D_j,D_k<\sqrt{n}/\varepsilon$ and $L_n=\mu n(1+o(1))$ under $\mathcal{A}_n$.
For $D_i,D_j,D_k\in[\varepsilon\sqrt{n},1/\varepsilon\sqrt{n}]$, the main term in~\eqref{eq:pind} can be uniformly bounded from above and from below by some functions $f_1(\varepsilon)$ and $f_2(\varepsilon)$ not depending on $n$. 
Combining this with~\eqref{eq:expBne} shows that
\begin{equation}\label{eq:exgn}
	\Expn{\triangle_n(B_n(\varepsilon))}=(1+o(1)) \sideset{}{'}\sum_{i,j,k}g_{n,\varepsilon}(D_i,D_j,D_k),
\end{equation}
on the event $\mathcal{A}_n$, which proves the lemma.
\end{proof}

We now replace the $L_n$ inside the definition of $g_{n,\varepsilon}$ by $\mu n$. In particular, define
\begin{equation}\label{eq:def_f_n_epsilon}
f_{n, \varepsilon}(x,y,z) = (1 - \me^{-\frac{xy}{\mu n}}) (1 - \me^{-\frac{yz}{\mu n}}) (1 - \me^{-\frac{zx}{\mu n}})
\ind{x, y, z \in B_n(\varepsilon)},
\end{equation}
then we have the following result:

\begin{lemma}\label{lem:g_n_epsilon_to_f_n_epsilon}
	Let ${\bf D}_n$ be sampled from $\mathscr{D}$ with $1< \gamma < 2$ and $\kappa = (\gamma - 1)/(1 + \gamma) > 0$. Then, for all $\varepsilon > 0$ and $\delta < \kappa$,
	\[
	n^{\frac{3 \gamma}{2} - 3 + \delta} \left|\sum_{1\leq i< j< k\leq n} g_{n, \varepsilon}(D_i, D_j, D_k) 
	- f_{n, \varepsilon}(D_i, D_j, D_k)\right| \plim 0.
	\]
\end{lemma}

\begin{proof}
Let $\mathcal{A}_n$ be as in \eqref{eq:event_A_n} and note that $\kappa$ is the same as in the statement of the lemma. Then, since $\Prob{\mathcal{A}_n} \to 1$, it is enough to prove that the result conditioned on $\mathcal{A}_n$. Next, note that
$1 - \me^{-a} \le 1$ for all $a \ge 0$. Hence, due to symmetry it suffices to prove
\[
n^{\frac{3 \gamma}{2} - 3 + \delta} \sum_{1\leq i< j< k\leq n} \left|\me^{-\frac{D_i D_j}{L_n}} -\me^{-\frac{D_i D_j}{\mu n}}\right| \ind{D_i, D_j, D_k \in B_n(\varepsilon)} \plim 0.
\]
For this we compute that, on the event $\mathcal{A}_n$,
\[
\left|\me^{-\frac{D_i D_j}{L_n}} - \me^{-\frac{D_i D_j}{\mu n}}\right| \le D_i D_j \frac{|L_n - \mu n|}{\mu n L_n}
\le D_i D_j \frac{n^{- 1 - \kappa}}{(\mu^2 - n^{-\kappa})} = D_i D_j \bigO{n^{-1-\kappa}}.
\]
We recall that by Karamata's Theorem \cite[Theorem 1.5.11]{Bingham1989} it follows that, as $n \to \infty$,
\begin{align*}
\Exp{D_1 \ind{D_1\in B_n(\varepsilon)}} &\le \Exp{D_1 \ind{D_1 > \varepsilon \sqrt{\mu n}}}
\sim \frac{\mathcal{L}\left(\varepsilon\sqrt{\mu n}\right)\varepsilon^{1-\gamma}(\mu n)^{\frac{1-\gamma}{2}}}{\gamma - 1} \\
&= O\left(\mathcal{L}\left(\varepsilon\sqrt{\mu n}\right)n^{\frac{1-\gamma}{2}}\right),
\end{align*}
where $a_n \sim b_n$ as $n \to \infty$ means $\lim_{n\to\infty} a_n/b_n = 1$.

Therefore, using Lemma \ref{lem:plvarying},
\begin{align*}
&\hspace{-30pt}\Exp{\left|\me^{-\frac{D_1 D_2}{L_n}} - \me^{-\frac{D_1 D_2}{\mu n}}\right| \ind{D_1, D_2, D_3 \in B_n(\varepsilon)}}\\
&\le \bigO{n^{-1-\kappa}} \Exp{D_1 \ind{D_1 \in B_n(\varepsilon)}}^2 \Prob{D_3 \in B_n(\varepsilon)} \\
&= \bigO{\mathcal{L}\left(\varepsilon\sqrt{\mu n}\right)^2\mathcal{L}(\sqrt{\mu n}) n^{-(\kappa + \frac{3\gamma}{2})}},
\end{align*}
so that by Markov's inequality, we obtain that for any $K > 0$ and $\varepsilon > 0$,
\begin{align*}
&\hspace{-30pt}\Prob{n^{\frac{3 \gamma}{2} - 3 + \delta} \sum_{1\leq i< j< k\leq n} \left|\me^{-\frac{D_i D_j}{L_n}} - \me^{-\frac{D_i D_j}{\mu n}}\right| \ind{D_i, D_j, D_k \in B_n(\varepsilon)} > K, \mathcal{A}_n} \\
&\le \frac{\binom{n}{3} n^{\frac{3 \gamma}{2} - 3 + \delta}}{K} 
\Exp{\left|\me^{-\frac{D_1 D_2}{L_n}} - \me^{-\frac{D_1 	D_2}{\mu n}}\right| \mathbbm{1}_{\mathcal{A}_n} \ind{D_1, D_2, D_3 \in B_n(\varepsilon)}} \\
&= \bigO{\mathcal{L}\left(\varepsilon\sqrt{\mu n}\right)^2\mathcal{L}(\sqrt{\mu n})n^{\delta - \kappa}} = o(1).
\end{align*}
\end{proof}

We now again study the expected number of triangles of vertices with sampled degrees in $B_n(\varepsilon)$, and investigate the behavior of the expression of Lemma~\ref{lem:extriang}.
\begin{lemma}\label{lem:extriangconv}
	Let ${\bf D}_n$ be sampled from $\mathscr{D}$ with $1< \gamma < 2$, and
	$\widehat{G}_n = \emph{\texttt{ECM}}({\bf D}_n)$. Let $\triangle_n(B_n(\varepsilon))$ denote the number of triangles in $\widehat{G}_n$ with sampled degrees in $B_n(\varepsilon)$. Then, as $n\to\infty$,
	\begin{equation}\label{eq:gconlem}
	\begin{aligned}
	&\hspace{-10pt}\frac{\Expn{\triangle_n(B_n(\varepsilon))}}{\mathcal{L}(\sqrt{\mu n})^3n^{\frac{3}{2}(2-\gamma)} \mu^{-\frac{3}{2}\gamma}}  \\
	&= (1+\op(1))\frac{1}{6}\int_{\varepsilon}^{1/\varepsilon} \int_{\varepsilon}^{1/\varepsilon} \int_{\varepsilon}^{1/\varepsilon}  \frac{\gamma^3}{(t_1t_2t_3)^{\gamma+1}}h(t_1,t_2,t_3)\dd t_1\dd t_2\dd t_3+\bigOp{\varepsilon},
	\end{aligned}
	\end{equation}
	where $h(x,y,z):=(1-\mathrm{e}^{-xy})(1-\mathrm{e}^{-xz})(1-\mathrm{e}^{-yz})$.
	\end{lemma}
\begin{proof}
Combining Lemmas~\ref{lem:extriang} and \ref{lem:g_n_epsilon_to_f_n_epsilon} yields that conditionally on the sampled degrees $(D_i)_{i\in[n]}$,
\begin{equation}\label{eq:triangexp}
\Expn{\triangle_n(B_n(\varepsilon))}=(1+\op(1)) \sum_{1\leq i< j< k\leq n}f_{n,\varepsilon}(D_i,D_j,D_k) + o_\pr\left(\mathcal{L}(\sqrt{\mu n})^3n^{\frac{3}{2}(2-\gamma)}\right)
\end{equation}
where $f_{n,\varepsilon}(D_i,D_j,D_k)$ is as in~\eqref{eq:def_f_n_epsilon}.
To investigate the convergence of~\eqref{eq:triangl} when taking the random degrees into account, define for $b>a\geq0$ the random measure 
\[
N_1^{\sss{(n)}}([a,b])=\frac{(\mu n)^{\frac{\gamma}{2}}}{\mathcal{L}(\sqrt{\mu n})}\frac{1}{n}\sum_{i=1}^n\ind{D_i\in[a\sqrt{\mu n},b\sqrt{\mu n}]},
\]
so that $N_1^{\sss{(n)}}$ counts the number of vertices with degrees in the interval $[a\sqrt{\mu n},b\sqrt{\mu n}]$.
Since every $D_i$ is drawn i.i.d.\ from~\eqref{eq:distribution_degrees}, the number of vertices with degrees between $a\sqrt{\mu n}$ and $b \sqrt{\mu n}$ is binomially distributed and therefore this number concentrates around its mean value, which is large. Combining this with Lemma~\ref{lem:plvarying} yields
\begin{align*}
N_1^{\sss{(n)}}\left([a,b]\right) &=\frac{(\mu n)^{\frac{\gamma}{2}}}{\mathcal{L}(\sqrt{\mu n})}\frac{1}{n}\sum_{i=1}^n\ind{D_i\in[a\sqrt{\mu n},b\sqrt{\mu n}]}\\
&=\frac{(\mu n)^{\frac{\gamma}{2}}}{\mathcal{L}(\sqrt{\mu n})}\Prob{D_i\in[a\sqrt{\mu n},b\sqrt{\mu n}]} (1+\op(1))\\
& = (1+\op(1))\gamma  \int_{a}^{b}t^{-\gamma-1}\dd t= (1+\op(1))(a^{-\gamma}-b^{-\gamma}).
\end{align*}
Therefore,
\begin{equation}\label{eq:lambda}
\begin{aligned}[b]
N_1^{\sss{(n)}}\left([a,b]\right)& \plim \int_{a}^{b}t^{-\gamma-1}\dd t = (a^{-\gamma}-b^{-\gamma}) =:\lambda([a,b]).
\end{aligned}
\end{equation}
Let $N^{\sss{(n)}}$ be the product measure $N_1^{\sss{(n)}}\times N_1^{\sss{(n)}}\times N_1^{\sss{(n)}}$ and $F=[\varepsilon,1/\varepsilon]^3$. Thus, $N^{\sss{(n)}}$ counts the number of triples with all three degrees proportional to $\sqrt{\mu n}$.

By~\eqref{eq:def_f_n_epsilon},
	\begin{align*}
	& \sum_{1\leq i< j< k\leq n}f_{n,\varepsilon}(D_i,D_j,D_k) \\
	& =  \sum_{1\leq i< j< k\leq n}(1 - \me^{-\frac{D_iD_j}{\mu n}}) (1 - \me^{-\frac{D_iD_k}{\mu n}}) (1 - \me^{-\frac{D_jD_k}{\mu n}})
	\ind{D_i,D_j,D_k \in B_n(\varepsilon)}\\
	&=\frac{1}{6}\mathcal{L}(\sqrt{\mu n})^3n^{\frac{3}{2}(2-\gamma)} \mu^{-\frac{3}{2}\gamma} \int_F (1-\mathrm{e}^{-t_1t_2})(1-\mathrm{e}^{-t_1t_3})(1-\mathrm{e}^{-t_2t_3})\dd N^{\sss{(n)}}(t_1,t_2,t_3).
	\end{align*}

The function $h(x,y,z)$ is bounded and continuous on $F=[\varepsilon, 1/\varepsilon]^3$. Therefore, writing the Taylor expansion of $\me^{-x}$ yields
\[
(1-\me^{-xy})=\sum_{i=1}^k\frac{(xy)^i}{i!}(-1)^i+O\left(\frac{\varepsilon^{-2k}}{(k+1)!}\right),
\]
on $x,y\in [\varepsilon,1/\varepsilon]$, where the error term goes to zero as $k\to\infty$. Therefore for any $\delta>0$, we can choose $k_1,k_2,k_3$ such that
\begin{align*}
h(x,y,z)& =\sum_{l_1=1}^{k_1}\sum_{l_2=1}^{k_2}\sum_{l_3=1}^{k_3}\left(\frac{(xy)^{l_1}}{l_1!}(-1)^{l_1}\frac{(yz)^{l_2}}{l_2!}(-1)^{l_2}\frac{(xz)^{l_3}}{l_3!}(-1)^{l_3}\right)+O\left(\delta\right)\\
&=\sum_{l_1=1}^{k_1}\sum_{l_2=1}^{k_2}\sum_{l_3=1}^{k_3}\left(\frac{x^{l_1+l_3}y^{l_1+l_2}z^{l_2+l_3}}{l_1!l_2!l_3!}(-1)^{l_1+l_2+l_3}\right)+O\left(\delta\right).
\end{align*}
Choosing $\delta=\varepsilon^{3\gamma+1}$ gives
\begin{align*}
& \frac{\sum_{1\leq i< j< k\leq n}f_{n,\varepsilon}(D_i,D_j,D_k)}{\mathcal{L}(\sqrt{\mu n})^3n^{\frac{3}{2}(2-\gamma)} \mu^{-\frac{3}{2}\gamma}}
=\frac{1}{6} \int_F h(t_1,t_2,t_3)\dd N^{\sss{(n)}}(t)\\
&=\frac 16 \int_{F}\sum_{l_1=1}^{k_1}\sum_{l_2=1}^{k_2}\sum_{l_3=1}^{k_3}\left(\frac{x^{l_1+l_3}y^{l_1+l_2}z^{l_2+l_3}}{l_1!l_2!l_3!}(-1)^{l_1+l_2+l_3}\right)+O\left(\varepsilon^{3\gamma+1}\right)\dd N^{\sss{(n)}}(t)\\
&=\frac 16 \int_{F}\sum_{l_1=1}^{k_1}\sum_{l_2=1}^{k_2}\sum_{l_3=1}^{k_3}\frac{x^{l_1+l_3}y^{l_1+l_2}z^{l_2+l_3}}{l_1!l_2!l_3!}(-1)^{l_1+l_2+l_3}\dd N^{\sss{(n)}}(t)+\bigOp{\varepsilon}\\
&=\frac 16\sum_{l_1=1}^{k_1}\sum_{l_2=1}^{k_2}\sum_{l_3=1}^{k_3}\frac{(-1)^{l_1+l_2+l_3}}{l_1!l_2!l_3!}\int_{\varepsilon}^{1/\varepsilon}x^{l_1+l_3}\dd N_1^{\sss{(n)}}(x)\int_{\varepsilon}^{1/\varepsilon}y^{l_1+l_2}\dd N_1^{\sss{(n)}}(y)\\
&\quad \cdot \int_{\varepsilon}^{1/\varepsilon}{z^{l_2+l_3}}\dd N_1^{\sss{(n)}}(z)+\bigOp{\varepsilon}.
\end{align*}
Here we used that because $N_1^{\sss{(n)}}[a,b]\plim \lambda[a,b]$,
	\begin{align*}
	\int_{F}\sum_{l_1=1}^{k_1}\sum_{l_2=1}^{k_2}\sum_{l_3=1}^{k_3}\varepsilon^4\dd N^{\sss{(n)}}(t)& = k_1k_2k_3 \varepsilon^{3\gamma+1} (N_1^{\sss{(n)}}([\varepsilon,1/\varepsilon]))^3\\& = \varepsilon^{3\gamma+1}\bigOp{\lambda([\varepsilon,1/\varepsilon])^3} = \varepsilon^{3\gamma+1}\bigOp{(\varepsilon^{-\gamma}-\varepsilon^{\gamma})^3} =\bigOp{\varepsilon}.
	\end{align*}
Since $x^t$ is bounded and continuous on $[\varepsilon,1/\varepsilon]$, we may use~\cite[Lemma~5]{hofstad2017c} to conclude that
\[
\int_{\varepsilon}^{1/\varepsilon}x^{t}\dd N_1^{\sss{(n)}}(x)\plim\int_{\varepsilon}^{1/\varepsilon}x^{t}\dd \lambda(x)
:= \varphi_\varepsilon(t).
\]
Thus,
\begin{align*}
& \frac{\sum_{1\leq i< j< k\leq n}f_{n,\varepsilon}(D_i,D_j,D_k)}{\mathcal{L}(\sqrt{\mu n})^3n^{\frac{3}{2}(2-\gamma)} \mu^{-\frac{3}{2}\gamma}}\\
& =(1+\op(1)) \frac 16\sum_{l_1=1}^{k_1}\sum_{l_2=1}^{k_2}\sum_{l_3=1}^{k_3}\frac{(-1)^{l_1+l_2+l_3}}{l_1!l_2!l_3!} \varphi_\varepsilon(l_1 + l_3) \varphi_\varepsilon(l_1 + l_2) \varphi_\varepsilon(l_2 + l_3)+\bigOp{\varepsilon}\\
&=(1+\op(1))\frac 16\int_F\sum_{l_1=1}^{k_1}\sum_{l_2=1}^{k_2}\sum_{l_3=1}^{k_3}\frac{x^{l_1+l_3}y^{l_1+l_2}z^{l_2+l_3}(-1)^{l_1+l_2+l_3}}{l_1!l_2!l_3!}\dd \lambda(x)\dd \lambda(y)\dd \lambda(z)+\bigOp{\varepsilon}\\
&=(1+\op(1))\frac 16 \int_F h(x,y,z)\dd \lambda(x)\dd \lambda(y)\dd \lambda(z)+\bigOp{\varepsilon},
\end{align*}
which concludes the proof together with~\eqref{eq:triangexp}.
\end{proof}

The following lemma bounds the variance of the number of triangles between vertices of sampled degrees proportional to $\sqrt{n}$. 
\begin{lemma}\label{lem:exvarECM}
Let ${\bf D}_n$ be sampled from $\mathscr{D}$ with $1< \gamma < 2$, and
$\widehat{G}_n = \emph{\texttt{ECM}}({\bf D}_n)$. Let $\triangle_n(B_n(\varepsilon))$ denote the number of triangles in $\widehat{G}_n$ with sampled degrees in $B_n(\varepsilon)$. Then, as $n\to\infty$,
\[
	\frac{\Varn{\triangle_n(B_n(\varepsilon))}}{\Expn{\triangle_n(B_n(\varepsilon))}^2}\plim 0.
\]
\end{lemma}

\begin{proof}
Choose $0<\delta<\tfrac{1}{5}(1-\tfrac{1}{2}\gamma)$. Denote 
\[
	\mathcal{B}_n=\Big\{\sum_{i=1}^{n}\ind{D_i>\varepsilon\sqrt{\mu n}}< n^{\frac{1}{2}(2-\gamma)+\delta}\Big\},
\]
let $\mathcal{A}_n$ be as defined in~\eqref{eq:event_A_n}, and set $\Lambda_n=\mathcal{A}_n\cap\mathcal{B}_n$.
Because the sampled degrees are i.i.d.\ samples from~\eqref{eq:distribution_degrees}, $\Prob{\Lambda_n}\to 1$. Therefore, we work on the event $\Lambda_n$ in the rest of the proof.
By Lemma~\ref{lem:extriangconv}, $\Expn{\triangle_n(B_n(\varepsilon))}=\Theta_{\sss{\mathbb{P}}}(n^{3-\frac{3}{2}\gamma}\mathcal{L}(\sqrt{n}))^3$. Thus, we need to prove that
\[
	\Varn{\triangle_n(B_n(\varepsilon))}=o_{\pr}\left(n^{6-3\gamma}\mathcal{L}(\sqrt{n})^6\right).
\]
We can write the variance $\Varn{\triangle_n(B_n(\varepsilon))}$ as
\begin{align}
	&\hspace{-10pt}\sum_{(i,j,k),(s,t,u)} \hspace{-4pt}\left(\Probn{\triangle_{i,j,k},\triangle_{s,t,u}}-\Probn{\triangle_{i,j,k}}\Probn{\triangle_{s,t,u}}\right)\ind{D_i,D_j,D_j,D_s,D_t,D_u\in B_n(\varepsilon)}, \label{eq:var6}
\end{align}
where $\triangle_{i,j,k}$ denotes the event that there is a triangle between vertices $i,j$ and $k$. 
This splits into several cases, depending on $\abs{\{i,j,k,s,t,u\}}$. Let the part of the variance where $\abs{\{i,j,k,s,t,u\}}=m$ be denoted by $V^{\sss{(m)}}$. so that
\[
\Varn{\triangle_n(B_n(\varepsilon))}=V^{\sss{(6)}}+V^{\sss{(5)}}+V^{\sss{(4)}}+V^{\sss{(3)}}
\]
Let $M_n(\varepsilon)=\sum_i\ind{D_i\in B_n(\varepsilon)}$.  Then
\[
V^{\sss{(m)}}\leq \sum_{i,j,k,s,t: |\{i,j,k,s,t\}|=m}\ind{D_i,D_j,D_j,D_s,D_t\in B_n(\varepsilon)}= M_n(\varepsilon)^m, \quad m=3,4,5.
\]
Under $\mathcal{B}_n$, 
\[
M_n(\varepsilon)\leq \sum_{i=1}^{n}\ind{D_i>\varepsilon\sqrt{\mu n}}<n^{\frac{1}{2}(2-\gamma)+\delta}.
\]
Thus, by the choice of $\delta$,
\[
	V^{\sss{(5)}} \leq M_n(\varepsilon)^5=\bigO{n^{5-5\gamma/2+5\delta}}=o(n^{6-3\gamma}),
\]
as required. Similar bounds show that $V^{\sss{(4)}}$ and $V^{\sss{(3)}}$ are $o(n^{6-3\gamma})$. Thus, the contributions  $V^{\sss{(m)}}$, $m=3,4,5$, to the variance are sufficiently small.
	
Now we investigate the case where 6 different indices are involved and show that $V^{\sss{(6)}}=\op(n^{6-3\gamma}\mathcal{L}(\sqrt{n})^6)$.
Equation~\eqref{eq:exgn} computes the second term inside the brackets in~\eqref{eq:var6}. To compute the first term inside the brackets, we make a very similar computation that leads to~\eqref{eq:exgn}. 
A similar computation as in~\eqref{eq:pind} yields that on the event $\mathcal{A}_n$
\[
\Probn{\triangle_{i,j,k}, \triangle_{s,t,u}} = (1+\op(1))g_{n,\varepsilon}(D_i,D_j,D_k)g_{n,\varepsilon}(D_s,D_t,D_u).
\]
Hence,
\[
\Probn{\triangle_{i,j,k},\triangle_{s,t,u}} -\Probn{\triangle_{i,j,k}}\Probn{\triangle_{s,t,u}}
= \op(1)g_{n,\varepsilon}(D_i,D_j,D_k)g_{n,\varepsilon}(D_s,D_t,D_u).
\]
When $D_i,D_j,D_k\in[\varepsilon,1/\varepsilon]\sqrt{n}$, $g_{n,\varepsilon}(D_i,D_j,D_k)\in[f_1(\varepsilon),f_2(\varepsilon)]$, uniformly in $i,j,k$.
Therefore, by Lemma~\ref{lem:extriang},
\begin{align*}
	V^{\sss{(6)}} &= \hspace{-3pt} \sum_{(i,j,k),(s,t,u)} \hspace{-3pt} o(1)g_{n,\varepsilon}(D_i,D_j,D_k)g_{n,\varepsilon}(D_s,D_t,D_u) \\
	&=\op \left(\Expn{\triangle_n(B_n(\varepsilon))}^2\right) = \op \left(n^{6-3\gamma}\mathcal{L}(\sqrt{n})^6\right),
\end{align*}
and therefore also the contribution to the variance where $\abs{\{i,j,k,s,t,u\}}=6$ is small enough.
\end{proof}

\begin{proof}[Proof of Theorem~\ref{thm:CECM}]
	First, we look at the denominator of the clustering coefficient in~\eqref{eq:clust}. By~\eqref{eq:Der_D_scaled} and the Stable Law Central Limit Theorem~\cite[Theorem 4.5.2]{whitt2006}, there exists a slowly-varying function $\mathcal{L}_0$ such that
\begin{equation}\label{eq:d2ecm}
\frac{\sum_{i=1}^n\widehat{D}_i(\widehat{D}_i-1)}{\mathcal{L}_0(n)n^{2/\gamma}}=\frac{\sum_{i=1}^n{D}_i^2}{\mathcal{L}_0(n)n^{2/\gamma}}(1+\op(1)) \dlim \mathcal{S}_{\gamma/2},
\end{equation}
where $\mathcal{S}_{\gamma/2}$ is a stable distribution.

Now we consider the numerator of the clustering coefficient. We prove the convergence of the number of triangles in several steps. First, we show that the major contribution to the number of triangles comes from the triangles between vertices with degrees proportional to $\sqrt{n}$. Fix $\varepsilon>0$. 
We use~\eqref{eq:splitX}, where we want to show that the contribution of $\triangle_n(\bar{B}_n(\varepsilon))$ is negligible. 
Applying Lemma~\ref{lem:sqrt} with the Markov inequality yields, for every $\delta > 0$,
\begin{equation*}
\limsup_{n\to\infty} \Prob{\triangle_n(\bar{B}_n(\varepsilon))> \delta{ \mathcal{L}(\sqrt{\mu n})^3}n^{\frac32 (2-\gamma)}}= O\left(\frac{\mathcal{E}_1(\varepsilon)}{\delta}\right).
\end{equation*}
Therefore,
\begin{equation}\label{eq:triangnl}
\begin{aligned}[b]
\triangle_n(\bar{B}_n(\varepsilon))=\bigOp{\mathcal{E}_1(\varepsilon) \mathcal{L}(\sqrt{\mu n})^3n^{\frac32 (2-\gamma)}}.
\end{aligned}
\end{equation}
Because $\mathcal{E}_1(\varepsilon)$ tends to zero as $\varepsilon\to 0$, we now focus on $\triangle_n(B_n(\varepsilon))$.
The number of triangles consists of two sources of randomness: the random pairing of the edges, and the random degrees. First we show that $\triangle_n(B_n(\varepsilon))$ concentrates around its mean when conditioned on the degrees. 
 By Lemma~\ref{lem:exvarECM} and Chebyshev's inequality,
 \[
 \frac{\triangle_n(B_n(\varepsilon))}{\Expn{\triangle_n(B_n(\varepsilon))}}\plim 1,
 \]
 conditionally on the degree sequence $(D_i)_{i\in[n]}$.
 Combining this with Lemma~\ref{lem:extriangconv} yields that conditionally on $(D_i)_{i\in[n]}$, 
\begin{equation}\label{eq:triangl}
\begin{aligned}
&\hspace{-15pt}\triangle_n(B_n(\varepsilon)) =\mathcal{L}(\sqrt{\mu n})^3n^{\frac{3}{2}(2-\gamma)} \mu^{-\frac{3}{2}\gamma} \times \\ &\left((1+\op(1))\frac{1}{6}\int_{\varepsilon}^{1/\varepsilon} \int_{\varepsilon}^{1/\varepsilon} \int_{\varepsilon}^{1/\varepsilon} \!\!\! \frac{\gamma^3}{(t_1t_2t_3)^{\gamma+1}}h(t_1,t_2,t_3)\dd t_1\dd t_2\dd t_3+\bigOp{\varepsilon}\right),
\end{aligned}
\end{equation}
where $h(x,y,z):=(1-\mathrm{e}^{-xy})(1-\mathrm{e}^{-xz})(1-\mathrm{e}^{-yz})$ and $\varepsilon$ is the same as in~\eqref{eq:triangnl}.

Combining~\eqref{eq:triangnl} and~\eqref{eq:triangl} gives
\begin{align*}
&\hspace{-15pt}\frac{\triangle_n}{\mathcal{L}(\sqrt{\mu n})^3n^{\frac 32(2-\gamma)}} =\frac{\triangle_n(B_n(\varepsilon))+\triangle_n(\bar{B}_n(\varepsilon))}{\mathcal{L}(\sqrt{\mu n})^3n^{\frac 32(2-\gamma)}}\\
& =(1+\op(1))\frac{\sum_{1\leq i< j< k\leq n}f_{n,\varepsilon}(D_i,D_j,D_k) }{\mathcal{L}(\sqrt{\mu n})^3n^{\frac 32(2-\gamma)}}+\op(1)+\bigOp{\mathcal{E}_1(\varepsilon)}\\
&=(1+\op(1))  \mu^{-\frac 32 \gamma}\frac{1}{6}\int_{\varepsilon}^{1/\varepsilon} \hspace{-5pt} \int_{\varepsilon}^{1/\varepsilon} \hspace{-5pt}\int_{\varepsilon}^{1/\varepsilon}  \frac{\gamma^3}{(t_1t_2t_3)^{\gamma+1}}h(t_1,t_2,t_3)\dd t_1\dd t_2\dd t_3 \\ &\hspace{10pt}+O(\mathcal{E}_1(\varepsilon))+\bigOp{\varepsilon}.
\end{align*}
Taking the limit of $n\to\infty$ and then $\varepsilon\to0$ combined with~\eqref{eq:d2ecm} and the definition of the clustering coefficient in~\eqref{eq:clust} then results in
\begin{align*}
\frac{\hat{C}_n\mathcal{L}_0(n)}{\mathcal{L}(\sqrt{\mu n})^3n^{(-3\gamma^2+6\gamma-4)/(2\gamma)}}&=6\frac{\mathcal{L}_0(n)n^{2/\gamma}}{\sum_{i\in[n]}\hat{D}_i(\hat{D}_i-1)}\frac{\triangle_n}{\mathcal{L}(\sqrt{\mu n})^3n^{\frac{3}{2}(2-\gamma)}}\\
& \dlim \mu^{-\frac 32 \gamma}\frac{ \int_{0}^{\infty} \int_{0}^{\infty} \int_{0}^{\infty}  \frac{\gamma^3}{(t_1t_2t_3)^{\gamma+1}}h(t_1,t_2,t_3)\dd t_1\dd t_2\dd t_3}{\mathcal{S}_{\gamma/2}}
\end{align*}
which proves Theorem~\ref{thm:CECM}.
\end{proof}

\section{Proofs of Theorem~\ref{thm:clt_pearson_hvm} and~\ref{thm:CHVM}}\label{sec:proof_hvm}
We now show how the proofs of Theorems~\ref{thm:clt_pearson_ecm} and~\ref{thm:CECM} can be adapted to prove similar results for the class of rank-1 inhomogeneous random graphs. We remind the reader that in this setting the degrees $D_i$ are no longer an i.i.d. sample from \eqref{eq:distribution_degrees}.

Let $\kappa\le (\gamma-1)/(1+\gamma)$, $0<\delta<1-1/\gamma$ and define the events
\begin{align*}
\mathcal{A}_n = \left\{\big|{\sum_{i \in[n]}w_i-\mu n}\big|\leq n^{1-\kappa}\right\},\quad 
\mathcal{B}_n = \left\{\big|\sum_{i = 1}^nw_i\ind{w_i<n^{\delta}}-\mu n\big|\leq \frac{n}{\log(n)}\Big|\right\}
\end{align*}
Let $\Lambda_n=\mathcal{A}_n\cap\mathcal{B}_n$. Because $\Prob{\Lambda_n}\to 0$, we condition on the event $\Lambda_n$. For the erased configuration model, we used that the degrees and the erased degrees are close to prove Theorems~\ref{thm:clt_pearson_ecm} and~\ref{thm:CECM}. Similarly, for the inhomogeneous random graphs we now show that the weights and the degrees are close. By Condition~\ref{cond:conhvm}(i)
\begin{equation}\label{eq:expwub}
\begin{aligned}
\Expn{D_i}=\sum_{j\neq i}\frac{w_iw_j}{\mu n}h\left(\frac{w_iw_j}{\mu n}\right)\leq \sum_{j \in[n]}\frac{w_iw_j}{\mu n}=w_i(1+o(1)),
\end{aligned}
\end{equation}
where $\e_n$ now denotes expectation conditioned on the weight sequence. 
Furthermore, for $w_i=O(n^{1/\gamma})$, on the event $\mathcal{B}_n$, by Condition~\ref{cond:conhvm}(ii) and (iii)
\begin{equation}\label{eq:expwlb}
\begin{aligned}[b]
\Expn{D_i}& \geq \sum_{j:w_j<n^{\frac{\gamma-1}{2\gamma}}}\frac{w_jw_i}{\mu n}h\left(\frac{w_iw_j}{\mu n}\right) = \sum_{j:w_j<n^{\frac{\gamma-1}{2\gamma}}}\frac{w_jw_i}{\mu n}\Big(1+\bigO{\frac{w_iw_j}{\mu n}}\Big)\\
& = (1+o(1)) \sum_{j:w_j<n^{\frac{\gamma-1}{2\gamma}}}\frac{w_jw_i}{\mu n}=w_i(1+o(1)),
\end{aligned}
\end{equation}
where the first equality follows from a first order Taylor expansion of $h(x)$. 
Combining~\eqref{eq:expwub} and~\eqref{eq:expwlb} yields $\Expn{D_i}=w_i(1+o(1))$. 

Let $X_{ij}$ again denote the indicator that edge $\{i,j\}$ is present. Note that $\Varn{X_{ij}}=p(w_i,w_j)(1-p(w_i,w_j))\leq p(w_i,w_j)$. Because conditioned on the weights, the degree of a vertex $D_i=\sum_{i \neq j}X_{ij}$ is the sum of independent indicators with success probability $p(w_i,w_j)$, $\Varn{D_i}\leq\sum_{i \neq j}p(w_i,w_j)\leq 2w_i$ for $n$ large enough. Then, Bernsteins inequality yields that for $t>0$
\[
\Probn{\abs{D_i-w_i}>t}\leq \exp\left(-\frac{t^2/2}{2w_i+t/3}\right).
\]
Thus, for $w_i>\log(n)$,
\[
D_i=w_i(1+\op(1)).
\]
Therefore,
\begin{align*}
\sum_{i \in [n]}|D_i^2-w_i^2|& =\sum_{i:w_i\leq \log(n)}|D_i^2-w_i^2|+\sum_{i:w_i>\log(n)}|D_i^2-w_i^2|\\
& = \bigOp{n\log^2(n)}+\op\Big(\sum_{i:w_i>\log(n)}w_i^2\Big)=\op\Big(\sum_{i \in[n]}w_i^2\Big),
\end{align*}
and a similar result holds for the third moment of the degrees. In particular, this implies that~\eqref{eq:Dhatsq} and~\eqref{eq:degmoments} also hold for the inhomogeneous random graph under Condition~\ref{cond:conhvm}.

\subsection{Pearson in the rank-1 inhomogeneous random graph}
The analysis of the term $r_n^-$ in~\eqref{eq:pearson_negative_part} is the same as in the erased configuration model, since it only depends on the degrees, and~\eqref{eq:Dhatsq} also holds for the inhomogeneous random graph. We therefore only need to show that Proposition~\ref{prop:pearson_ecm_positive_part} also holds for the rank-1 inhomogeneous random graph. 
This means that we need to show that~\eqref{eq:pearson_ecm_r+_joint_degrees} also holds for the rank-1 inhomogeneous random graph. For all models satisfying Condition~\ref{cond:conhvm}, $p(w_i,w_j)\leq (w_iw_j/(\mu n) \wedge 1)$.  
Because $\Expn{D_i}=w_i(1+o(1))$, $D_i=\bigOp{w_i}$, by Markov's inequality. Thus, 
\begin{equation*}
\sum_{1 \le i < j \le n}D_iD_jX_{ij}=\bigOp{\sum_{1 \le i < j \le n}w_iw_jX_{ij}}= \bigOp{\sum_{1 \le i < j \le n}w_iw_j \left(\frac{w_iw_j}{\mu n} \wedge 1\right)}
\end{equation*}
Because the weights are sampled from~\eqref{eq:distribution_degrees}, this is the exact same bound as in~\eqref{eq:pearson_clt_ecm_r+_term_1}, so that from there we can follow the same lines as the proof of Proposition~\ref{prop:pearson_ecm_positive_part}. Thus, Proposition~\ref{prop:pearson_ecm_positive_part} also holds for rank-1 inhomogeneous random graphs satisfying Condition~\ref{cond:conhvm}. Then we can follow the same lines as the proof of Theorem~\ref{thm:clt_pearson_ecm} to prove Theorem~\ref{thm:clt_pearson_hvm}.

\subsection{Clustering in the rank-1 inhomogeneous random graph}
For the clustering coefficient, note that conditioned on the weights, $\Probn{\triangle_{i,j,k}}=p(w_i,w_j)p(w_j,w_k)p(w_i,w_k)$. Furthermore, Lemma~\ref{lem:sqrt} only requires the bound $p(w_i,w_j)$\\
$\leq (w_iw_j/(\mu n) \wedge 1)$, which also holds for all rank-1 inhomogeneous random graphs satisfying Condition~\ref{cond:conhvm}, so that Lemma~\ref{lem:sqrt} also holds for these rank-1 inhomogeneous random graphs. Furthermore, conditioned on the weights, the probabilities of distinct edges being present are independent, so that
\[
\Expn{\triangle_n(B_n(\varepsilon))}=\sum_{1 \le i < j < k \le n}q\left(\frac{w_iw_j}{\mu n}\right)q\left(\frac{w_iw_k}{\mu n}\right)q\left(\frac{w_jw_k}{\mu n}\right)\ind{(w_i,w_j,w_k)\in B_n(\varepsilon)}
\]
similarly to Lemma~\ref{lem:extriang}, with $q$ as defined in Condition~\ref{cond:conhvm}. Furthermore, the bound on the variance of the number of triangles in the erased configuration model in Lemma~\ref{lem:exvarECM} for 3,4 or 5 contributing vertices only depends on the degrees, so that it also holds for the rank-1 inhomogeneous random graph satisfying Condition~\ref{cond:conhvm} since the weights are also sampled form~\eqref{eq:distribution_degrees}. The contribution of 6 different vertices to the variance is zero, because the presence of distinct edges is independent. Thus, Lemma~\ref{lem:exvarECM} also holds for the rank-1 inhomogeneous random graph. Thus, we can follow the lines of the proof of Theorem~\ref{thm:CECM} until equation~\eqref{eq:lambda}. From there, note that
\begin{align*}
& \frac{\sum_{1 \le i < j < k \le n}q\left(\frac{w_iw_j}{\mu n}\right)q\left(\frac{w_iw_k}{\mu n}\right)q\left(\frac{w_jw_k}{\mu n}\right)\ind{(w_i,w_j,w_k)\in B_n(\varepsilon)}}{\mathcal{L}(\sqrt{\mu n})^3n^{\frac{3}{2}(2-\gamma)}\mu^{-\tfrac{3}{2}\gamma}} \\
& \quad = \frac{1}{6}\int_Fq(t_1t_2)q(t_1t_3)q(t_2t_3)\dd N^{\sss{(n)}}_1(t_1)\dd N^{\sss{(n)}}_1(t_2)\dd N^{\sss{(n)}}_1(t_3).
\end{align*}
Then we use that the function $q(t_1t_2)q(t_1t_3)q(t_2t_3)$ is a bounded, continuous function by Condition~\ref{cond:conhvm}, so that by~\cite[Lemma 5]{hofstad2017c},
\begin{align*}
& \int_Fq(t_1t_2)q(t_1t_3)q(t_2t_3)\dd N^{\sss{(n)}}_1(t_1)\dd N^{\sss{(n)}}_1(t_2)\dd N^{\sss{(n)}}_1(t_3)\\
& \quad \plim \int_Fq(t_1t_2)q(t_1t_3)q(t_2t_3)\dd \lambda(t_1)\dd\lambda(t_2)\dd \lambda(t_3).
\end{align*}
Thus, we can follow the exact same lines of the proof of the clustering coefficient in the erased configuration model, replacing the term $(1-\me^{-t_1t_2})(1-\me^{-t_1t_3})(1-\me^{-t_2t_3})$ by $q(t_1t_2)q(t_1t_3)q(t_2t_3)$, which them proves Theorem~\ref{thm:CHVM}.

\section*{Acknowledgments}
The authors want to thank two anonymous referees for their constructive feedback on earlier versions of the paper. The work of RvdH and CS was supported by NWO TOP grant 613.001.451. The work of RvdH is further supported by the NWO Gravitation Networks grant 024.002.003 and the NWO VICI grant 639.033.806. PvdH and NL were supported by the EU-FET Open grant NADINE 288956. PvdH was further supported by ARO grant W911NF1610391.

\section{Appendix}\label{sec:appendix}

In this section we prove some technical results used in the preceding proofs. 

\subsection{Erased edges}\label{ssec:proof_erased_edges}

We start with the proof for the scaling of the number of erased edges.

\begin{proof}[Proof of Theorem \ref{thm:scaling_erased_edges}]
Let $K, \delta > 0$, $\kappa\le (\gamma-1)/(1+\gamma)$ and define the events
\begin{align*}
	\mathcal{A}_n = \left\{\left|L_n - \mu n\right| \le n^{1 - \kappa}\right\}, \quad
	\mathcal{B}_n = \left\{\max_{1 \le i \le n} D_i \le n^{\frac{1}{\gamma} + \frac{\delta}{2}}\right\},
	\quad \mathcal{C}_n = \left\{\sum_{i = 1}^n D_i^2 \le n^{\frac{2}{\gamma} + \frac{\delta}{2}}\right\}
\end{align*}
and set $\Lambda_n = \mathcal{A}_n \cap \mathcal{B}_n \cap \mathcal{C}_n$. Then by Lemma \ref{lem:Kolmogorov_SLLN} and Proposition 
\ref{prop:scaling_sums_powers_regularly_varying}, $\Prob{\Lambda_n} \to 1$, and hence we only need
to proof the result conditioned on the event $\Lambda_n$.

First recall that
\[
	Z_n = \sum_{i = 1}^n X_{ii} + \sum_{1 \le i < j \le n} \left(X_{ij} - \ind{X_{ij} > 0}\right).
\]
We first consider the conditional expectation of last term $\Expn{\ind{X_{ij} > 0}} = 1 - \Probn{X_{ij} = 0}$. It follows from~\cite[equation 4.9]{hofstad2005} that
\begin{align*}
	\Probn{X_{ij} = 0} 
	&\le \prod_{t=0}^{D_i - 1} \left(1 - \frac{D_j}{L_n - 2D_i - 1}\right) 
		+ \frac{D_i^2 D_j}{(L_n - 2D_i)^2}\\
	&\le \left(1 - \frac{D_j}{L_n - 1}\right)^{D_i} + \frac{D_i^2 D_j}{(L_n - 2D_i)^2}\\
	&\le e^{-\frac{D_i D_j}{L_n - 1}} + \frac{D_i^2 D_j}{(L_n - 2D_i)^2}.
\end{align*}
The additional term essentially comes from the fact that we need to consider the cases were a stub of node $i$ connects to another stub of node $i$.

Next, since $\Exp{X_{ii}} = D_i(D_i-1)/(L_n - 1)$ and $\Exp{X_{ij}} = D_iD_j/(L_n - 1)$ we have
\[
	\Expn{Z_n} \le \sum_{1 \le i < j \le n} \phi\left(
	\frac{D_i D_j}{L_n - 1}\right) + \sum_{i = 1}^n \frac{D_i^2 - D_i}{L_n - 1} 
	+ \sum_{1 \le i < j \le n} \frac{D_i^2 D_j}{(L_n - 2D_i)^2},
\]
where $\phi(x) = x - 1 + \me^{-x}$. Define 
\[
	M_n = \sum_{i = 1}^n \frac{D_i^2 - D_i}{L_n - 1} 
	+ \sum_{1 \le i < j \le n} \frac{D_i^2 D_j}{(L_n - 2D_i)^2}.
\]

Then, on the event $\Lambda_n$,
\begin{align*}
	M_n &\le \frac{n^{\frac{2}{\gamma} + \frac{\delta}{2}}}{\mu n - n^{1 - \kappa} - 1}
		+ \frac{(\mu n - n^{1 - \kappa})
		n^{\frac{2}{\gamma} + \frac{\delta}{2}}}{\left(\mu n - n^{1 - \kappa} -2n^{\frac{1}{\gamma} + \frac{\delta}{2}}\right)^2} \\
	&= \bigO{n^{\frac{2}{\gamma} - 1 + \frac{\delta}{2}}}.
\end{align*}
Note that $\Lambda_n$ is completely determined by the degree sequence. Hence, by Markov's inequality,
we get
\begin{align*}
	\Prob{Z_n > n^{2 - \gamma + \delta}, \Lambda_n} 
	&\le n^{\gamma - 2 - \delta} \Exp{Z_n \indE{\Lambda_n}} 
		= n^{\gamma - 2 - \delta} \Exp{\Expn{Z_n}\indE{\Lambda_n}}\\
	&\le  n^{\gamma - 2 - \delta} \Exp{\sum_{1 \le i < j \le n} \phi\left(
		\frac{D_i D_j}{L_n - 1}\right) \ind{\Lambda_n}} + n^{\gamma - 2 - \delta} \Exp{M_n \indE{\Lambda_n}} \\
	&\le n^{\gamma - 2 - \delta} \Exp{\sum_{1 \le i < j \le n} \phi\left(
		\frac{D_i D_j}{\mu n - 1 - n^{1 - \kappa}}\right)} + \bigO{n^{\gamma + \frac{2}{\gamma} - 3 - \frac{\delta}{2}}}\\
	&\le n^{\gamma - \delta} \Exp{\phi\left(\frac{D_1 D_2}{\mu n -1 - n^{1 - \kappa}}\right)}  
		+ \bigO{n^{\gamma + \frac{2}{\gamma} - 3 - \frac{\delta}{2}}},
\end{align*}
as $n \to \infty$. Since 
\[
	\gamma + \frac{2}{\gamma} - 3 < 0,
\]
for all $1 < \gamma < 2$ the last term goes to zero as $n \to \infty$. For the other term we note that
$D_1$ and $D_2$ are independent regularly-varying random variables with exponent $1 < \gamma < 2$ and therefore
so is $D_1 D_2$. It then follows from \cite{Bingham1974} that for any $\delta > 0$
\[
	\lim_{t \to \infty} \frac{\Exp{\phi\left(\frac{D_1 D_2}{t}\right)}}{t^{-\gamma + \delta}} = 0.
\]
Taking $t = \mu n -1 - n^{1 - \kappa}$ we obtain that
\[
	\lim_{n \to \infty} n^{\gamma} \Exp{\phi\left(\frac{D_1 D_2}{\mu n -1 - n^{1 - \kappa}}\right)} = 0,
\]
from which it follows that
\[
	\lim_{n \to \infty} \Prob{E_n > n^{2 - \gamma + \delta}, \Lambda_n} = 0.
\]
\end{proof}

We proceed with the proof of the corollary.

\begin{proof}[Proof of Corollary \ref{cor:scaling_degrees_erased_edges}]
For the first part we write
\begin{align*}
	\sum_{i = 1}^n D_i^p Y_i \le \max_{1 \le j \le n} D_j^p \sum_{i = 1}^n Y_i \le 2 Z_n \max_{1 \le j \le n} D_j^p,
\end{align*}
and hence, using Theorem \ref{thm:scaling_erased_edges} and
Proposition \ref{prop:scaling_sums_powers_regularly_varying},
\[
	\frac{\sum_{i = 1}^n D_i^p Y_i}{n^{\frac{p}{\gamma} + 2 - \gamma + \delta}}
	\le \left(\frac{2 Z_n}{n^{2 - \gamma + \frac{\delta}{2}}}\right) \left(\frac{\max_{1 \le j \le n} D_j^p}{n^{\frac{p}{\gamma} + \frac{\delta}{2}}}\right) \plim 0.
\]
For the second part we bound the main term by
\[
	\sum_{1 \le i < j \le n} X_{ij} D_i D_j \le \max_{1 \le j \le n} D_j \sum_{1 \le i < j \le n} X_{ij} D_i
	\le \frac{1}{2} \max_{1 \le j \le n} D_j \sum_{i = 1}^n D_i Y_i. 
\]
Hence, using the first part of the corollary and Proposition \ref{prop:scaling_sums_powers_regularly_varying}, it follows that
\[
	\frac{\sum_{1 \le i < j \le n} X_{ij} D_i D_j}{n^{\frac{2}{\gamma} +2 - \gamma + \delta}}
	\le \left(\frac{ \max_{1 \le j \le n} D_j}{2 n^{\frac{1}{\gamma} + \frac{\delta}{2}}}\right)
	\left(\frac{\sum_{i = 1}^n D_i Y_i}{n^{\frac{1}{\gamma} +2 - \gamma + \frac{\delta}{2}}}\right) \plim 0.
\]
\end{proof}

\subsection{Technical results for clustering}

The following result is needed for the proof of Lemma \ref{lem:sqrt}.

\begin{lemma}\label{lem:plvarying}
Let $X$ be a non-negative regularly-varying random variable with distribution~\eqref{eq:distribution_degrees}. Then, for any $0 \le a < b$,
\[
\Prob{X\in[a,b]\sqrt{n}}=\mathcal{L}(\sqrt{\mu n})n^{-\gamma/2}\gamma\int_{a}^b x^{-\gamma-1}\dd x(1+o(1)).
\]
\end{lemma}

\begin{proof}
Because $\mathcal{L}$ is a slowly varying function,
\[
	\mathcal{L}(c\sqrt{n})=\mathcal{L}(\sqrt{\mu n})(1+o(1)),
\]
for any $c\in(0,\infty)$. Furthermore, using the Taylor expansion of $(a\sqrt{n}-x)^{-\gamma}$ at $x=0$ yields 
\[
	(b\sqrt{n}+1)^{-\gamma}=(b\sqrt{n})^{-\gamma}+\gamma(b\sqrt{n})^{-\gamma-1}+O(\gamma(\gamma-1)(b\sqrt{n})^{-\gamma-2}).
\]
Because $\mathcal{L}$ is a slowly-varying function, for every constant $t$, $\lim_{n \to \infty}\mathcal{L}(t\sqrt{n})/\mathcal{L}(\sqrt{n})=1$. Thus, we obtain
\begin{align*}
	\Prob{X\in[a,b]\sqrt{n}}&=\mathcal{L}(a\sqrt{n})(a\sqrt{n})^{-\gamma}-\mathcal{L}(b\sqrt{n}+1)(b\sqrt{n})^{-\gamma}\\
	&= \mathcal{L}(\sqrt{\mu n})(1+o(1))(a\sqrt{n})^{-\gamma} - \mathcal{L}(\sqrt{\mu n})(1+o(1))(b\sqrt{n})^{-\gamma} \\
	&=(1+o(1))\mathcal{L}(\sqrt{\mu n})\left((a\sqrt{n})^{-\gamma}-(b\sqrt{n})^{-\gamma}\right)\\
	&=(1+o(1))\mathcal{L}(\sqrt{\mu n})n^{-\gamma/2}\gamma\int_{a}^{b}x^{-\gamma-1}\dd x.
\end{align*}
\end{proof}

\bibliographystyle{apt}
\bibliography{../central_limit_theorems_cm}

\begin{thebibliography}{10}

\bibitem{Angel2016}
Omer Angel, Remco van~der Hofstad, and Cecilia Holmgren.
\newblock Limit laws for self-loops and multiple edges in the configuration
  model.
\newblock {\em To appear in AHIP}, 2019.

\bibitem{bannink2018}
Tom Bannink, Remco van~der Hofstad, and Clara Stegehuis.
\newblock Switch chain mixing times and triangle counts in simple graphs with
  given degrees.
\newblock {\em Journal of Complex Networks}, page cny013, 2018.

\bibitem{Bingham1974}
Nicholas~H. Bingham and Ron~A. Doney.
\newblock {Asymptotic properties of supercritical branching processes I: The
  Galton-Watson process}.
\newblock {\em Advances in Applied Probability}, pages 711--731, 1974.

\bibitem{Bingham1989}
Nicholas~H. Bingham, Charles~M. Goldie, and Jef~L. Teugels.
\newblock {\em Regular variation}, volume~27.
\newblock Cambridge university press, 1989.

\bibitem{boguna2003}
Mari\'an Bogu\~n\'a and Romualdo Pastor-Satorras.
\newblock Class of correlated random networks with hidden variables.
\newblock {\em Phys. Rev. E}, 68:036112, 2003.

\bibitem{Bollobas1980}
B{\'e}la Bollob{\'a}s.
\newblock A probabilistic proof of an asymptotic formula for the number of
  labelled regular graphs.
\newblock {\em European Journal of Combinatorics}, 1(4):311--316, 1980.

\bibitem{britton2006}
Tom Britton, Maria Deijfen, and Anders Martin-L\"{o}f.
\newblock Generating simple random graphs with prescribed degree distribution.
\newblock {\em J. Stat. Phys.}, 124(6):1377--1397, 2006.

\bibitem{Catanzaro2005}
Michele Catanzaro, Mari{\'{a}}n Bogu{\~{n}}{\'{a}}, and Romualdo
  Pastor-Satorras.
\newblock {Generation of uncorrelated random scale-free networks}.
\newblock {\em Physical Review E}, 71(2):27103, 2005.

\bibitem{chung2002}
F.~Chung and L.~Lu.
\newblock The average distances in random graphs with given expected degrees.
\newblock {\em Proc. Natl. Acad. Sci. USA}, {\bf 99}(25):15879--15882
  (electronic), 2002.

\bibitem{colomer2012}
Pol Colomer-de Simon and Mari\'an Bogu\~n\'a.
\newblock Clustering of random scale-free networks.
\newblock {\em Phys. Rev. E}, 86:026120, Aug 2012.

\bibitem{federico2017}
Lorenzo Federico and Remco van~der Hofstad.
\newblock Critical window for connectivity in the configuration model.
\newblock {\em Combinatorics, Probability and Computing}, 26(5):660–680, 2017.

\bibitem{gao2018}
Jane Gao, Remco van~der Hofstad, Angus Southall, and Clara Stegehuis.
\newblock Counting triangles in power-law uniform random graphs.
\newblock {\em arXiv:1812.04289}, 2018.

\bibitem{VanDerHofstad2016a}
Remco {\swap{Hofstad}{van~der~}}.
\newblock {\em Random graphs and complex networks}, volume~1.
\newblock Cambridge University Press, 2017.

\bibitem{hofstad2005}
Remco {\swap{Hofstad}{~van~der~}}, Gerard Hooghiemstra, and Piet van Mieghem.
\newblock Distances in random graphs with finite variance degrees.
\newblock {\em Random Structures \& Algorithms}, 27(1):76--123, 2005.

\bibitem{hofstad2007}
Remco {\swap{Hofstad}{~van~der~}}, Gerard Hooghiemstra, and Dmitri Znamenski.
\newblock Distances in random graphs with finite mean and infinite variance
  degrees.
\newblock {\em Electron. J. Probab.}, 12(0):703--766, 2007.

\bibitem{hofstad2017b}
Remco {\swap{Hofstad}{~van~der~}}, A.~J. E.~M. Janssen, Johan S.~H. van
  Leeuwaarden, and Clara Stegehuis.
\newblock Local clustering in scale-free networks with hidden variables.
\newblock {\em Phys. Rev. E}, 95:022307, Feb 2017.

\bibitem{Hofstad2014}
Remco {\swap{Hofstad}{~van~der~}} and Nelly Litvak.
\newblock Degree-degree dependencies in random graphs with heavy-tailed
  degrees.
\newblock {\em Internet mathematics}, 10(3-4):287--334, 2014.

\bibitem{hofstad2017d}
Remco {\swap{Hofstad}{~van~der~}}, Johan S.~H. van Leeuwaarden, and Clara
  Stegehuis.
\newblock Optimal subgraph structures in scale-free networks.
\newblock {\em \textup{arXiv:1709.03466}}, 2017.

\bibitem{hofstad2017c}
Remco {\swap{Hofstad}{~van~der~}}, Johan S.~H. van Leeuwaarden, and Clara
  Stegehuis.
\newblock Triadic closures in configuration models with unbounded degree
  fluctuations.
\newblock {\em -}, 2017.

\bibitem{Hoorn2016}
Pim {\swap{Hoorn}{~van~der~}}.
\newblock {\em {Asymptotic analysis of network structures: degree-degree
  correlations and directed paths}}.
\newblock PhD thesis, University of Twente, 2016.

\bibitem{hoorn2017sparse}
Pim {\swap{Hoorn}{~van~der~}}, Gabor Lippner, and Dmitri Krioukov.
\newblock Sparse maximum-entropy random graphs with a given power-law degree
  distribution.
\newblock {\em Journal of Statistical Physics}, pages 1--39, 2017.

\bibitem{Hoorn2015PhysRev}
Pim {\swap{Hoorn}{~van~der~}} and Nelly Litvak.
\newblock Phase transitions for scaling of structural correlations in directed
  networks.
\newblock {\em Physical Review E}, 92(2):022803, 2015.

\bibitem{Hoorn2015}
Pim {\swap{Hoorn}{~van~der~}} and Nelly Litvak.
\newblock {\em Upper Bounds for Number of Removed Edges in the Erased
  Configuration Model}, pages 54--65.
\newblock Springer International Publishing, Cham, 2015.

\bibitem{Hoorn2015b}
Pim {\swap{Hoorn}{~van~der~}} and Mariana Olvera-Cravioto.
\newblock Typical distances in the directed configuration model.
\newblock {\em Ann. Appl. Probab.}, 28(3):1739--1792, 2018.

\bibitem{janson2009b}
Svante Janson.
\newblock On percolation in random graphs with given vertex degrees.
\newblock {\em Electron. J. Probab.}, 14:86--118, 2009.

\bibitem{Janson2009}
Svante Janson.
\newblock The probability that a random multigraph is simple.
\newblock {\em Combinatorics, Probability and Computing}, 18(1-2):205--225,
  2009.

\bibitem{Janson2014}
Svante Janson.
\newblock The probability that a random multigraph is simple. ii.
\newblock {\em Journal of Applied Probability}, 51(A):123--137, 2014.

\bibitem{karamata1962}
Jovan Karamata.
\newblock Some theorems concerning slowly varying functions.
\newblock Technical report, WISCONSIN UNIV-MADISON MATHEMATICS RESEARCH CENTER,
  1962.

\bibitem{Newman2002}
Mark E.~J. Newman.
\newblock Assortative mixing in networks.
\newblock {\em Physical review letters}, 89(20):208701, 2002.

\bibitem{Newman2003a}
Mark E.~J. Newman.
\newblock Mixing patterns in networks.
\newblock {\em Physical Review E}, 67(2):026126, 2003.

\bibitem{newman2003book}
Mark E.~J. Newman.
\newblock The structure and function of complex networks.
\newblock {\em SIAM Review}, 45(2):167--256, 2003.

\bibitem{newman2001}
Mark E.~J. Newman, Steven~H. Strogatz, and Duncan~J. Watts.
\newblock Random graphs with arbitrary degree distributions and their
  applications.
\newblock {\em Phys. Rev. E}, 64(2):026118, 2001.

\bibitem{norros2006}
Ilkka Norros and Hannu Reittu.
\newblock On a conditionally poissonian graph process.
\newblock {\em Adv. Appl. Probab.}, 38(01):59--75, 2006.

\bibitem{ostroumova2016}
Liudmila Ostroumova~Prokhorenkova.
\newblock Global clustering coefficient in scale-free weighted and unweighted
  networks.
\newblock {\em Internet Mathematics}, 12(1-2):54--67, 2016.

\bibitem{prokhorenkova2014}
Liudmila Ostroumova~Prokhorenkova and Egor Samosvat.
\newblock Global clustering coefficient in scale-free networks.
\newblock In {\em Lecture Notes in Computer Science}, pages 47--58. Springer
  International Publishing, 2014.

\bibitem{park2004}
Juyong Park and M.~E.~J. Newman.
\newblock Statistical mechanics of networks.
\newblock {\em Phys. Rev. E}, 70:066117, 2004.

\bibitem{pierce1997}
Robert~D. Pierce.
\newblock Application of the positive alpha-stable distribution.
\newblock In {\em Higher-Order Statistics, 1997., Proceedings of the IEEE
  Signal Processing Workshop on}, pages 420--424. IEEE, 1997.

\bibitem{riordan2012}
Oliver Riordan.
\newblock The phase transition in the configuration model.
\newblock {\em Comb. Probab. Comput.}, 21(1-2):265--299, 2012.

\bibitem{samorodnitsky1994}
Gennady Samorodnitsky and Murad~S. Taqqu.
\newblock {\em Stable non-Gaussian random processes: stochastic models with
  infinite variance}, volume~1.
\newblock CRC press, 1994.

\bibitem{stegehuis2017b}
Clara Stegehuis.
\newblock Degree correlations in scale-free null models.
\newblock {\em \textup{arXiv:1709.01085}}, 2017.

\bibitem{watts1998}
Duncan~J. Watts and Steven~H. Strogatz.
\newblock Collective dynamics of small-world networks.
\newblock {\em Nature}, 393(6684):440--442, 1998.

\bibitem{whitt2006}
Ward Whitt.
\newblock {\em Stochastic-Process Limits}.
\newblock Springer New York, 2006.

\bibitem{Wormald1980}
Nicholas~C. Wormald.
\newblock Some problems in the enumeration of labelled graphs.
\newblock {\em Bulletin of the Australian Mathematical Society},
  21(1):159--160, 1980.

\bibitem{yao2017ANND}
Dong Yao, Pim van~der Hoorn, and Nelly Litvak.
\newblock Average nearest neighbor degrees in scale-free networks.
\newblock {\em Internet Mathematics}, 2018.

\end{thebibliography}

\end{document}